\documentclass[letterpaper,12pt,oneside,reqno]{amsart}
\usepackage{amsmath,amsthm,amsfonts,amssymb,bbm,color}
\usepackage{graphicx,psfrag,subfigure,url}
\usepackage{cite}
\usepackage{mathrsfs}
\usepackage[colorlinks=true]{hyperref}
\usepackage[DIV13]{typearea}

\newcommand{\CwPre}[1]{\widetilde{\mathcal{C}}_{#1}} 
\newcommand{\CsPre}[1]{\widetilde{\mathcal{D}}_{#1}}
\newcommand{\Cv}[1]{\mathcal{C}_{#1}}
\newcommand{\Cs}[1]{\mathcal{D}_{#1}}
\newcommand{\CvEps}[1]{\mathcal{C}^{\epsilon}_{#1}}
\newcommand{\CvEpsR}[1]{\mathcal{C}^{\epsilon}_{#1}}
\newcommand{\CvEpsRGeq}[1]{\mathcal{C}^{\epsilon}_{#1}}
\newcommand{\CvEpsReq}[1]{\mathcal{C}^{\epsilon}_{#1}}
\newcommand{\CvRLeq}[1]{\mathcal{C}_{#1}}

\numberwithin{equation}{section}
\newcommand{\I}{{\rm i}}
\newcommand{\Id}{\mathbbm{1}}
\newcommand{\Or}{\mathcal{O}}
\newcommand{\N}{\mathbb{N}}
\newcommand{\e}[0]{\varepsilon}
\newcommand{\BesselK}{\ensuremath{\mathrm{K}}}
\newcommand{\EE}{\ensuremath{\mathbb{E}}}
\newcommand{\PP}{\ensuremath{\mathbb{P}}}
\newcommand{\R}{\ensuremath{\mathbb{R}}}
\newcommand{\C}{\ensuremath{\mathbb{C}}}
\newcommand{\Z}{\ensuremath{\mathbb{Z}}}
\newcommand{\Rplus}{\ensuremath{\mathbb{R}_{+}}}

\newcommand{\M}{\ensuremath{\mathbf{M}}}
\newcommand{\MM}{\ensuremath{\mathbf{MM}}}
\newcommand{\W}[1]{\ensuremath{\mathbf{W}}_{(#1)}}
\newcommand{\WM}[1]{\ensuremath{\mathbf{WM}}_{(#1)}}
\newcommand{\Zsd}{\ensuremath{\mathbf{Z}}}

\newcommand{\Usd}{\ensuremath{\mathbf{U}}}
\newcommand{\Sym}{\ensuremath{\mathrm{Sym}}}

\newcommand{\LogGdist}{\ensuremath{\log\Gamma}}

\newcommand{\Res}[1]{\underset{{#1}}{\mathrm{Res}}}
\newcommand{\la}[0]{\lambda}

\newcommand{\T}{\mathbb{T}}
\newcommand{\wt}{\widetilde}
\newcommand{\wh}{\widehat}

\renewcommand{\d}{\mathrm d}
\renewcommand{\Re}{\operatorname{Re}}
\renewcommand{\Im}{\operatorname{Im}}
\DeclareMathOperator{\dist}{dist} 
\DeclareMathOperator{\Ai}{Ai} \DeclareMathOperator{\sign}{sgn}
\DeclareMathOperator*{\var}{Var}
\DeclareMathOperator*{\sgn}{sgn}

\newtheorem{theorem}{Theorem}[section]
\newtheorem{proposition}[theorem]{Proposition}

\newtheorem{lemma}[theorem]{Lemma}
\newtheorem{corollary}[theorem]{Corollary}

\newtheorem{remark}[theorem]{Remark}

\newtheorem{definition}[theorem]{Definition}

\title{Height fluctuations for the stationary KPZ equation}

\author[A. Borodin]{Alexei Borodin}
\address{A. Borodin,
Massachusetts Institute of Technology,
Department of Mathematics,
77 Massachusetts Avenue, Cambridge, MA 02139-4307, USA, and Institute for Information Transmission Problems, Bolshoy Karetny per. 19, Moscow 127994, Russia}
\email{borodin@math.mit.edu}

\author[I. Corwin]{Ivan Corwin}
\address{I. Corwin, Columbia University,
Department of Mathematics,
2990 Broadway,
New York, NY 10027, USA,
and Clay Mathematics Institute, 10 Memorial Blvd. Suite 902, Providence, RI 02903, USA,
and Institute Henri Poincare, 11 Rue Pierre et Marie Curie, 75005 Paris, France,
and Massachusetts Institute of Technology,
Department of Mathematics,
77 Massachusetts Avenue, Cambridge, MA 02139-4307, USA}
\email{ivan.corwin@gmail.com}

\author[P.L. Ferrari]{Patrik Ferrari}
\address{P.L. Ferrari,
 Bonn University,
Institute for Applied Mathematics,
Endenicher Allee 60, 53115 Bonn, Germany}
\email{ferrari@uni-bonn.de}

\author[B. Vet\H{o}]{B\'alint Vet\H{o}}
\address{B. Vet\H{o},
Bonn University, Institute for Applied Mathematics,
Endenicher Allee 60, 53115 Bonn, Germany
and MTA--BME Stochastics Research Group,
Egry J.\ u.\ 1., 1111 Budapest, Hungary}
\email{vetob@math.bme.hu}
\begin{document}

\begin{abstract}
We compute the one-point probability distribution for the stationary KPZ equation (i.e.\ initial data $\mathcal{H}(0,X)=B(X)$, for $B(X)$ a two-sided standard Brownian motion) and show that as time $T$ goes to infinity, the fluctuations of the height function $\mathcal{H}(T,X)$ grow like $T^{1/3}$ and converge to those previously encountered in the study of the stationary totally asymmetric simple exclusion process, polynuclear growth model and last passage percolation.

The starting point for this work is our derivation of a Fredholm determinant formula for Macdonald processes which degenerates to a corresponding formula for Whittaker processes. We relate this to a polymer model which mixes the semi-discrete and log-gamma random polymers. A special case of this model has a limit to the KPZ equation with initial data given by a two-sided Brownian motion with drift $\beta$ to the left of the origin and $b$ to the right of the origin. The Fredholm determinant has a limit for $\beta>b$, and the case where $\beta=b$ (corresponding to the stationary initial data) follows from an analytic continuation argument.
\end{abstract}

\sloppy \maketitle \thispagestyle{empty}

\newpage
\setcounter{tocdepth}{2}
\tableofcontents

\section{Introduction}
In their seminal 1986 paper~\cite{KPZ86}, Kardar, Parisi and Zhang (KPZ) proposed the stochastic evolution equation for a height function $\mathcal{H}(T,X)\in \R$ ($T\in \Rplus$ is time and $X\in\R$ is space)
\begin{equation*}
\partial_T \mathcal{H}(T,X)= \tfrac12 \partial_X^2 \mathcal{H}(T,X)+\tfrac12\left(\partial_X \mathcal{H}(T,X)\right)^2+\xi(T,X).
\end{equation*}
The randomness $\xi$ models the deposition mechanism and it is taken to be space-time Gaussian white noise, so that formally $\EE[\xi(T,X)\xi(S,Y)] = \delta(T-S)\delta(X-Y)$.
The Laplacian reflects the smoothing mechanism and the non-linearity reflects the slope-dependent growth velocity of the interface. Using earlier physical work of Forster, Nelson and Stephen~\cite{FNS77}, KPZ predicted that for large time $T$, the height function $\mathcal H(T,X)$ has fluctuations around its mean of order $T^{1/3}$ with spatial correlation length of order $T^{2/3}$. Since then, the exact nature of these fluctuations has been a subject of extensive study. For additional background, see the reviews~\cite{Cor11,Q14,SS10c}.

For general initial data, it is expected that the solutions to the KPZ equation are locally Brownian in space~\cite{QR12, CH13,Hai11}. Therefore, making direct sense of the non-linearity in the equation is a challenge~\cite{BC95,Hai11}. The physically relevant notion~\cite{BG97,ACQ10,Cor11,Q14,Hai11,Dem13,SS10b} of a solution to the KPZ equation is therefore defined indirectly via the well-posed stochastic heat equation (SHE) with multiplicative noise,
\begin{equation*}
\partial_T \mathcal{Z}(T,X)=\tfrac12\partial_X^2 \mathcal{Z}(T,X)+\mathcal{Z}(T,X)\xi(T,X)
\end{equation*}
with initial condition $\mathcal{Z}(0,X)=\mathcal{Z}_0(X)=e^{\mathcal{H}(0,X)}$. The Cole--Hopf solution of the KPZ equation is then defined as  $\mathcal{H}(T,X)=\ln(\mathcal{Z}(T,X))$. On account of this definition, we will talk about the SHE and KPZ equation interchangeably, stating most of our main results (with the exception of those in this first section) in terms of the SHE.

By a version of the Feynman--Kac formula, the solution of the SHE can be formally written as
\begin{equation*}
\mathcal{Z}(T,X) =\EE_{T,X}\left[\mathcal{Z}_0(b(0))\, : \exp:\left\{-\int_0^T \xi(b(S),S) \d S\right\}\right]
\end{equation*}
where the expectation $\EE_{T,X}$ is over a Brownian motion $b(\cdot)$ going backwards in time from $b(T)=X$, and where $:\exp:$ is the Wick ordered exponential~\cite[Section~4.2]{Cor11}. This provides an interpretation for $\mathcal{Z}(T,X)$ as the partition function of the continuum directed random polymer (CDRP)~\cite{ACQ10, AKQ12b}.

Formally, the spatial derivative $\mathcal{U}(T,X)=\partial_X \mathcal{H}(T,X)$ of the KPZ equation satisfies the stochastic Burgers equation
\begin{equation*}
\partial_T \mathcal{U}(T,X)= \tfrac12 \partial_X^2 \mathcal{U}(T,X)+\tfrac12 \partial_X \big(\mathcal{U}(T,X)\big)^2+\partial_X\xi(T,X),
\end{equation*}
which can be thought of as a continuum version of an interacting particle system~\cite{BG97,BQS11}.

Let $B(X)$ be a two-sided Brownian motion with $B(0)=0$ and zero drift. Stationary (zero drift) initial data $\mathcal{H}(0,X)=B(X)$ for the KPZ equation corresponds with SHE initial data $\mathcal{Z}(0,X) = e^{B(X)}$ and stochastic Burgers equation initial data $\mathcal{U}(0,X)=\partial_X B(X)$. This is called stationary, because for any later time $T$, $\mathcal{U}(T,\cdot)$ is marginally distributed as another spatial Gaussian white-noise. In terms of the KPZ equation, for fixed $T>0$, $\mathcal{H}(T,\cdot)$ is marginally distributed as $\tilde{B}(\cdot)+ \mathcal{H}(T,0)$ where $\tilde{B}(\cdot)$ is a two-sided Brownian motion (though not independent of $B$ or $\mathcal H(T,0)$).

The first rigorous confirmation of the $T^{1/3}$ fluctuation scale prediction for the KPZ equation was provided by~\cite{BQS11}, showing that there exist constants $c_0>0$ and $0<c_1<c_2<\infty$ such that for all $T>c_0$,
\begin{equation*}
c_1 T^{2/3} \leq \var\big(\mathcal{H}(T,0)\big) \leq c_2 T^{2/3}.
\end{equation*}
A similar fluctuation scale result was demonstrated recently in~\cite{CH13} (and applies equally well for a broad class of KPZ initial data) based on the KPZ line ensemble construction.

The present work provides an exact formula for the one-point probability distribution of the stationary solution to the KPZ equation, and a limit theorem for $\mathcal{H}(T,X)$ after proper centering and scaling by $T^{1/3}$. The following theorem and corollary are special cases (drift $b=0$ and position $X=0$) of Theorem~\ref{ThmFormulaStationary}, Proposition~\ref{PropInverseMelling} and Theorem~\ref{CorUniversality}.

\begin{theorem}\label{ThmFormulaStationaryIntroVersion}
Let $\mathcal{H}(T,X)$ be the stationary (zero drift) solution to the KPZ equation and let $\BesselK_0$ denote the modified Bessel function~\cite{AS84}. Then, for $T>0$, $\sigma=(2/T)^{1/3}$ and $S>0$,
\begin{equation*}
\EE\left[2\sigma \BesselK_0\left(2\sqrt{S \,\exp\big\{\tfrac{T}{24} +\mathcal{H}(T,0)\big\}}\right)\right]=\Xi\left(S,0,\sigma\right),
\end{equation*}
where the function $\Xi$ is given in Definition~\ref{longdef}.
Equivalently, for any $r\in\R$, we have
\begin{equation*}
\PP\left(\frac{\mathcal{H}(T,0)+\frac{T}{24}}{(T/2)^{1/3}}\le r\right)
=\frac{1}{\sigma^2}\frac{1}{2\pi\I}\int_{-\delta+\I\R} \frac{\d \xi}{\Gamma(-\xi)\Gamma(-\xi+1)} \int_{\R} \d x\, e^{x\xi/\sigma} \Xi\left(e^{-\frac{x+r}\sigma},0,\sigma\right)
\end{equation*}
for any $\delta>0$.
\end{theorem}

\begin{theorem}\label{CorUniversalityIntroVersion}
For any $r\in\R$,
\begin{equation*}
\lim_{T\to\infty}\PP\left(\frac{\mathcal{H}(T,0)+\frac{T}{24}}{(T/2)^{1/3}}\le r\right)=F_0(r),
\end{equation*}
where $F_0$ is given in Definition~\ref{anotherlongone} with $\tau=0$.
\end{theorem}

Inherent in the work of KPZ was the premise that a larger class of growth processes than just their eponymous equation should display the same $T^{1/3}$ and $T^{2/3}$ scaling exponents. The class of such models is referred to as the KPZ universality class. Generally speaking, the universality belief is that a growth model will belong to the KPZ class if it has the same physical properties as the KPZ equation, namely local growth dynamics, a smoothing mechanism and irreversibility arising from the condition that the speed of growth as a function of the slope has non-zero second derivative.

It took a quarter of a century to prove that the KPZ equation was in the KPZ universality class itself (via demonstrating the $1/3$ and $2/3$ exponents)~\cite{BQS11,ACQ10,CQ10,BC11,BCF12,CH13,SS10b}. Before this, starting with the 1999 work of~\cite{BDJ99,Jo00b}, a few growth models in the KPZ universality class were rigorously analyzed. These models were the polynuclear growth model (PNG), totally asymmetric simple exclusion process (TASEP) and last passage percolation (LPP) with special exponential, geometric or Bernoulli weights. Beyond the $T^{1/3}$ and $T^{2/3}$ scaling, the limit distributions and spatial processes for these models were determined. These statistical properties agreed between the models, but depended non-trivially on the type of initial data or geometry for the growth models, such as curved~\cite{BDJ99,Jo00b,BR99,PS02,Jo03,BF08,BF07}, flat~\cite{BR99,Sas05,
BFPS06,BF07,Fer04,BFP06,FSW13,BFS07b} or stationary~\cite{BR00,SI04,PS02b,FS05a,BFP09,FSW14,BFP12}. All these results strongly used the underlying determinantal structure that these models all enjoy (see the reviews~\cite{Fer07,Fer10b,Cor11,BG12,Fer13} for further references and details).

The KPZ equation does not seem to have a full-blown determinantal structure (as opposed to PNG, TASEP and LPP). However, in the last few years a number of new exactly solvability methods have been developed which have led to explicit formulas for the one-point marginal distribution of the solution to the KPZ equation with specific types of initial data and verified the $1/3$ exponent for general initial data (also the $2/3$ exponent has been verified for specific initial data). With the exception of the non-rigorous replica method (method 2, below), the other (rigorous) methods have all proceeded via analysis of exactly solvable discretizations or regularizations of the KPZ equation such as the (partially) asymmetric simple exclusion process (ASEP), the $q$-deformed totally asymmetric simple exclusion process ($q$-TASEP), or the O'Connell-Yor semi-discrete directed random polymer (see the review~\cite{Cor14} and references therein). These stochastic processes converge to the KPZ equation under special {\it weakly asymmetric} or {\it weak noise} scaling. It should be emphasized that the developed methods are presently only adapted to study certain types of initial data (except in the case of method 5, the KPZ line ensemble). As we summarize them below (for a partial list of references to subsequent developments and extensions, see~\cite{Cor14}), we will first focus on narrow wedge initial data for the KPZ equation, which means starting the SHE with $\mathcal{Z}(0,X)=\delta_{X=0}$.
\begin{enumerate}
\item~\cite{TW08,TW08b,TW08c} used Bethe ansatz to compute transition probabilities for the $N$-particle ASEP, extracted a one-point marginal distribution formula suitable for the $N$ to infinity limit corresponding with step initial data, and manipulated the resulting formula into a Fredholm determinant formula amenable to asymptotic analysis. This served as the starting point for the rigorous derivation in~\cite{ACQ10} of the one-point distribution for the KPZ equation with narrow-wedge initial data (see also~\cite{SS10b} for a parallel and independent, though non-rigorous, derivation of this).
\item~\cite{Dot10} and~\cite{CDR10} computed exact formulas for moments of the SHE with $\mathcal{Z}(0,X)=\delta_{X=0}$ using the connection with the delta Bose gas and the Bethe ansatz. From these moments they derived a formula for the Laplace transform of $\mathcal{Z}(T,X)$ and hence, by inverting the transform, the distribution of $\mathcal{H}(T,X)$. This physics replica method derivation suffers from being quite non-rigorous since the moments, in fact, grow too fast to determine the Laplace transform and distribution.
\item~\cite{BC11} introduced Macdonald processes and connected them to certain 2d growth processes (and 1d marginals like $q$-TASEP) as well as provided exact Fredholm determinant formulas for one-point distributions amenable to asymptotic analysis. A limit transition connects these processes to the Whittaker processes which, in~\cite{OCon09}, had been introduced and related to the O'Connell--Yor semi-discrete directed random polymer via a geometric lifting of the RSK correspondence. This method was used in~\cite{BCF12} to rederive the narrow wedge KPZ one-point distribution formula.
\item~\cite{BCS12} used Markov dualities of ASEP and $q$-TASEP, as well as the Bethe ansatz to compute explicit formulas for expectations of a large class of observables of these models, when started from step initial data. From these expectations, they derived a formula for a $q$-deformed Laplace transform of the one-point distribution. This provides an alternative to the methods of~\cite{TW08,TW08b,TW08c} as well as a rigorous regularization of the replica method used in~\cite{Dot10} and~\cite{CDR10}.
\item~\cite{CH13} constructed a line ensemble extension to the fixed time $T$ solution to the narrow wedge initial data KPZ equation which enjoys a distributional invariance called the $H$-Brownian Gibbs property as well as certain uniform regularity under $T^{1/3}, T^{2/3}$ scaling as $T$ goes to infinity. From this they proved the validity of the $2/3$ spatial exponent for the narrow wedge initial data KPZ equation and proved the $1/3$ fluctuation exponent for a wide class of KPZ initial data.
\end{enumerate}
A brief review and comparison of methods 1, 2 and 4 can be found in~\cite{ICtwowaystosolveASEP}, whereas some aspects of method 3 are reviewed in~\cite{BG12,BP13}. See also the review~\cite{Cor14}.

Besides the narrow-wedge initial data, there are a few other types of initial data for which these methods have proved successful in computing exact formulas for KPZ equation one-point distributions.
\begin{enumerate}
\item Half Brownian KPZ initial data corresponds with $\mathcal{Z}(X,0)=e^{B(X)}\Id_{X\geq 0}$, $B(\cdot)$ being a one-sided Brownian motion. It was rigorously analyzed via method 1 in~\cite{CQ10} and method 3 in~\cite{BCF12}, as well as non-rigorously analyzed via method 2 in~\cite{SI11}. A family generalizing half Brownian initial data was further rigorously analyzed via method 3 in~\cite{BCF12}.
\item Flat and half-flat KPZ initial data corresponds with $\mathcal{Z}(X,0) = 1$. It was non-rigorously analyzed via method 2 in~\cite{CLD11,CLD12,LD14}. No rigorous confirmation of these results have appeared yet.
\item Stationary KPZ initial data, the subject of this paper, corresponds with \mbox{$\mathcal{Z}(X,0)= e^{B(X)}$}, $B(\cdot)$ being a two-sided Brownian motion fixed at $B(0)=0$. It was non-rigorously analyzed via method 2 in~\cite{IS12}. In Remark~\ref{comparerem} we address the question of comparing the formula derived therein to that proved in Theorem~\ref{ThmFormulaStationaryIntroVersion}.
\end{enumerate}
Using these exact one-point formulas, it has further been confirmed in all of the above cases of initial data that the large $T$ one-point distribution converges to the same distribution as observed in the determinantal models of PNG, TASEP and LPP. Presently it is only for determinantal models that multi-point distributions and limit processes have been computed (see, however, nonrigorous work of~\cite{SP11,Dot13}). Besides the specific types of initial data discussed above, using method 5,~\cite{CH13} proved that up to certain rather weak hypothesis on initial data, the KPZ equation always has order $T^{1/3}$ fluctuations as $T$ goes to infinity.

\subsection{Outline}
In this paper we build on method 3, Macdonald processes, in order to prove Theorem~\ref{ThmFormulaStationaryIntroVersion}. It is not clear presently how to arrive at this result via the other rigorous methods (1, 4, or 5). Let us outline the main steps to prove Theorem~\ref{ThmFormulaStationaryIntroVersion} as well as make note of some of the other results of interest which we attain herein:

\medskip
\noindent {\bf Section~\ref{sectmodel}:} We introduce the O'Connell--Yor semi-discrete directed random polymer with log-gamma boundary sources and the associated multi-path extensions to its partition functions. Theorem~\ref{ThmFormulaSemiDiscrete} provides a Fredholm determinant formula for the Laplace transform of the partition function of the polymer model. Theorem~\ref{ThmFormulaContinuous} gives the analogue of Theorem~\ref{ThmFormulaSemiDiscrete}, but for the SHE/KPZ equation; Theorem~\ref{ThmFormulaStationary} gives a corresponding formula for the stationary version of the model; and Theorem~\ref{CorUniversality} demonstrates the KPZ universality ($T^{1/3}$ scaling and limiting one-point probability distribution) of the stationary model.

Theorem~\ref{ThmFormulaStationaryIntroVersion} is, in fact, a special case of Theorem~\ref{ThmFormulaStationary}. As such, the rest of the paper divides naturally into two parts. The first part, comprised of Sections~\ref{SectMacdonald},~\ref{SectConvWhitt}, and~\ref{SectSemiDirected}, provides a proof of Theorem~\ref{ThmFormulaSemiDiscrete}. The second part, comprised of Sections~\ref{SectCDRP},~\ref{SectStatSHE}, and~\ref{SectUniversality}, provides asymptotic analysis of our semi-discrete directed random polymer results to prove the SHE/KPZ equation results of Theorems~\ref{ThmFormulaContinuous} and~\ref{ThmFormulaStationary}, as well as Theorem~\ref{CorUniversality}.

\smallskip
\noindent {\bf Section~\ref{SectMacdonald}:} We introduce the $q$-Whittaker processes (equivalently, $t=0$ Macdonald processes) with $q\in(0,1)$. These are measures on interlacing partitions or Gelfand--Tsetlin patterns  $\big\{\lambda^{(k)}_j\}_{1\leq j\leq k\leq N}$. For a certain class of $q$-Whittaker  {\it nonnegative specializations} of the processes (indexed by $\tilde \alpha,\tilde \beta$ and $\tilde \gamma$ parameters) we prove Theorem~\ref{qFredDetThm}, a Fredholm determinant formula for the $e_q$-deformed Laplace transform of the random variable $q^{-\lambda^{(N)}_1}$. This is done following the general approach introduced in~\cite{BC11} and used there to prove a similar type of formula for $q^{\lambda^{(N)}_N}$. Unlike for $q^{\lambda^{(N)}_N}$, studied in~\cite{BC11}, $q^{-\lambda^{(N)}_1}$ is an unbounded random variable which only has finitely many moments. This would appear to be a major impediment in implementing the approach of~\cite{BC11} since it relies upon taking a generating
function of explicit formulas for moments $\EE\Big[\big(q^{-\lambda^{(N)}_1}\big)^k\Big]$ in order to recover the distribution. This issue of moment divergence does not arise for the so-called pure $\tilde \beta$ specializations, and so in that case we can follow the approach of~\cite{BC11} to prove this special case of Theorem~\ref{qFredDetThm} (this is recorded as Proposition~\ref{qFredDetThmbeta}). It is the $\tilde \alpha,\tilde \gamma$ specialization (for which moments diverge) which, however, we are really after due to its relationship with the semi-discrete directed random polymer with log-gamma boundary sources. In order to extend Theorem~\ref{qFredDetThm} to those specializations as well, we observe that the equality of the $\tilde \beta$ specialization $e_{q}$-Laplace transform with the corresponding Fredholm determinant actually implies a formal series (in Newton power sum symmetric polynomials) identity. The $\tilde \alpha,\tilde \gamma$ specialization of this identity yields convergent series on
both sides and
hence proves the equivalence of the $\tilde \alpha,\tilde \beta,\tilde \gamma$ specialized $e_{q}$-Laplace transform with the claimed Fredholm determinant in Theorem~\ref{qFredDetThm}. In this way, we see the power of relating our observable of interest, $q^{-\lambda^{(N)}_1}$, to the larger structure of $q$-Whittaker processes and symmetric polynomials.

This rigorous $e_q$-deformed Laplace transform derivation should be contrasted to the non-rigorous derivations (in method 2) of the Laplace transform of $\mathcal{Z}(T,0)$ from the moments $\EE\Big[\big(\mathcal{Z}(T,0)\big)^k\Big]$ which grow too quickly to uniquely identity the distribution. Under the various limit transitions which relate $q^{-\lambda_1^{(N)}}$ to $\mathcal{Z}(T,0)$ (i.e.\ Theorems~\ref{ThmConnectWhitpoly} and~\ref{ThmDiscreteToContinuous}) we lose the tools of symmetric functions which saved us. In particular, it is not clear how the $\tilde\beta$ specialization behaves under these limit transitions, and the notion of formal series identities seems to be lost.

\smallskip
\noindent {\bf Section~\ref{SectConvWhitt}:} We introduce Whittaker processes, measures on $\big\{T^{(k)}_j\}_{1\leq j\leq k\leq N}$. Theorem~\ref{ThmConnectWhitpoly} uses results of~\cite{OCon09,COSZ11} to relate these processes to the O'Connell--Yor semi-discrete directed random polymer with log-gamma boundary sources from Section~\ref{sectmodel}. In particular, this implies that the random variable $e^{T^{(N)}_1}$ (for a suitable Whittaker process specialization) and the polymer partition function $\mathbf{Z}^{N,M}(\tau)$ have the same distribution. Theorem~\ref{thmweakconv} shows how the $\tilde \alpha,\tilde \gamma$ specialized $q$-Whittaker processes converges as $q\to 1$ to the Whittaker processes (under special scaling), and thus (up to scaling) how $q^{-\lambda^{(N)}_1}$ converges to $e^{T^{(N)}_1}$. The pure $\tilde \gamma$ specialization version of this convergence result was proved as~\cite[Theorem~4.1.21]{BC11}, and the pure $\tilde \alpha$ specialization version was proved (modulo a decay
estimate which was not checked) as~\cite[Theorem~4.2.4]{BC11}. By combining these specializations, it becomes unnecessary to check the omitted decay estimate from~\cite[Theorem~4.2.4]{BC11}. So as not to be too obtuse, we provide the steps in this proof, even though they closely mimic those from~\cite{BC11}.

\smallskip
\noindent {\bf Section~\ref{SectSemiDirected}:}
We prove Theorem~\ref{ThmFormulaSemiDiscrete} by combining Theorem~\ref{ThmConnectWhitpoly} with Theorem~\ref{ThmFormulaWhit}. Theorem~\ref{ThmFormulaWhit} provides a Fredholm determinant formula for the Laplace transform under Whittaker processes of $e^{T^{(N)}_1}$. It is proved in this section by asymptotic analysis of the corresponding Fredholm determinant formula for the $e_q$-Laplace transform under $q$-Whittaker processes of $q^{-\lambda^{(N)}_1}$ (given as Proposition~\ref{propstep2}) along with the process convergence result of Theorem~\ref{thmweakconv}.

\smallskip
\noindent {\bf Section~\ref{SectCDRP}:} We turn here to studying the asymptotic behavior of the semi-discrete directed random polymer with boundary sources, as relates to the SHE/KPZ equation. Theorem~\ref{ThmDiscreteToContinuous} records a result of~\cite{QMR12} showing how the semi-discrete model converges to the SHE/KPZ equation. Theorem~\ref{ThmKbbeta} then provides the corresponding asymptotic analysis of the Fredholm determinant for the semi-discrete model with log-gamma boundary sources coming from Theorem~\ref{ThmFormulaSemiDiscrete}. These considerations prove Theorem~\ref{ThmFormulaContinuous} which gives the Laplace transform of $\mathcal{Z}_{b,\beta}(T,X)$, SHE/KPZ equation solution with initial data $\mathcal{Z}_0(X)=\exp(B(x))$ where $B(X)$ is a two-sided Brownian motion with drift $\beta$ on the left of $0$ and drift $b$ on the right of $0$ for $\beta>b$.

\smallskip
\noindent {\bf Section~\ref{SectStatSHE}:} We now take the limit as $\beta\searrow b$ in order to recover $\mathcal{Z}_{b}(T,X)$, the solution to the SHE/KPZ equation with stationary initial data $\mathcal{Z}_0(X)=\exp(B(x))$ where $B(X)$ is a two-sided Brownian motion with drift $b$ (on both sides). Taking the corresponding limit of Theorem~\ref{ThmFormulaContinuous} requires some care (in particular an analytic continuation argument similar to that used previously in the rigorous analysis of stationary TASEP in~\cite{FS05a} and in the non-rigorous replica analysis of the KPZ equation in~\cite{IS12}) and is given as Theorem~\ref{ThmFormulaStationary}.

\smallskip
\noindent {\bf Section~\ref{SectUniversality}:}
We prove Theorem~\ref{CorUniversality}, which demonstrates a universality result, namely that in the large time limit we recover the $T^{1/3}$ fluctuation scaling and the one-point probability distribution function previously obtained previously for stationary PNG and TASEP~\cite{BR00,SI04,PS02b,FS05a}.

\subsubsection*{Acknowledgments:}
The authors extend thanks to N. O'Connell, J. Quastel and T. Sasamoto for helpful discussions on this subject.
We thank E. Dimitrov for pointing out an issue about the definition of integration contours in an earlier version.
A. Borodin was partially supported by the NSF grant DMS-1056390. I. Corwin was partially supported by the NSF grant DMS-1208998 as well as by Microsoft Research and MIT through the Schramm Memorial Fellowship, by the Clay Mathematics Institute through the Clay Research Fellowship and through the Institute Henri Poincare through the Poincare Chair. P.L. Ferrari was supported by the German Research Foundation via the SFB 1060--B04 project.
B.\ Vet\H o is grateful for the support of the Humboldt Research Fellowship for Postdoctoral Researchers during his stay at the University of Bonn and for the Postdoctoral Fellowship of the Hungarian Academy of Sciences.
His work was partially supported by OTKA (Hungarian National Research Fund) grant K100473.

\section{Models and main results}\label{sectmodel}

\subsection{Semi-discrete directed random polymer with boundary sources}
To obtain our main result, the one-point probability distribution functions for the stationary KPZ equation (Theorem~\ref{ThmFormulaStationaryIntroVersion} and more generally, Theorem~\ref{ThmFormulaStationary}), we start by studying a semi-discrete directed random polymer model with log-gamma boundary sources. This is a mixture of models introduced by O'Connell and Yor~\cite{OCY01} and Sepp\"{a}l\"{a}inen~\cite{Sep09}. Indeed, taking $M=0$ and $\tau>0$ recovers the semi-discrete directed random polymer of~\cite{OCon09} while taking $M>0$ and $\tau=0$ recovers the log-gamma discrete directed random polymer of~\cite{Sep09}.

\begin{figure}
\begin{center}
\psfrag{w11}[cc]{$\omega_{-1,1}$}
\psfrag{w21}[cc]{$\omega_{-2,1}$}
\psfrag{wM1}[cc]{$\omega_{-M,1}$}
\psfrag{w1N}[bc]{$\,\omega_{-1,N}$}
\psfrag{w2N}[bc]{$\!\omega_{-2,N}$}
\psfrag{wMN}[bc]{$\omega_{-M,N}$}
\psfrag{tn}[bc]{$(\tau,N)$}
\psfrag{B1}[lc]{$B_1$}
\psfrag{B2}[lc]{$B_2$}
\psfrag{B3}[lc]{$B_3$}
\psfrag{BN}[lc]{$B_N$}
\psfrag{s2}[cc]{$s_2$}
\psfrag{s3}[cc]{$s_3$}
\psfrag{sNm1}[cc]{$s_{N-1}$}
\psfrag{t}[cc]{$\tau$}
\psfrag{phi}[bc]{$\phi$}
\includegraphics[height=5cm]{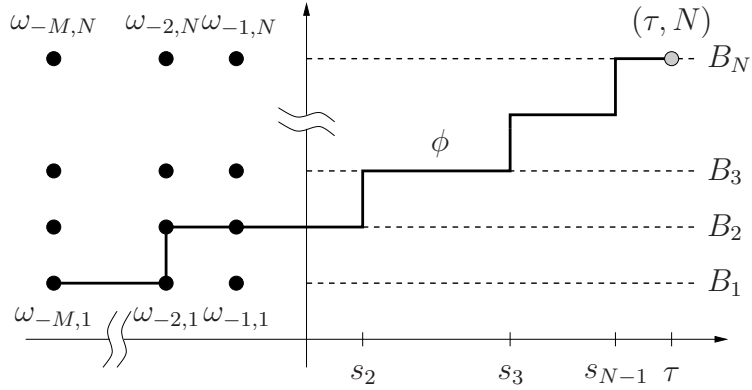}
\caption{Illustration of the semi-discrete directed random polymer with log-gamma boundary sources. The thick solid line is a possible directed random polymer path $\phi$ from $(-M,1)$ to $(\tau,N)$. Its energy is given by \eqref{eqEnergy}. The random variables $\omega_{-k,n}$ are distributed as $-\ln\Gamma(\alpha_k-a_n)$, while the Brownian motions $B_1,\ldots,B_N$ have drifts $a_1,\ldots,a_N$ respectively.}
\label{FigSemiDiscreteDP}
\end{center}
\end{figure}

For $\theta>0$, a random variable $X$ is distributed as $\Gamma(\theta)$ (written $X\sim \Gamma(\theta)$) if it has density with respect to Lebesgue measure given by
$$
\frac{\d}{\d x}\PP(X\le x) = \Id_{\{x>0\}} \frac{1}{\Gamma(\theta)}\, x^{-\theta-1}e^{-x}
$$
and a random variable $W$ is distributed as $-\ln\Gamma(\theta)$ (written $W\sim -\ln\Gamma(\theta)$) if $W = -\ln X$ for $X\sim \Gamma(\theta)$.

Fix $N\geq 1$ and $M\geq 0$. Let $a=(a_1,\ldots,a_N)\in \R^N$ and $\alpha = (\alpha_1,\ldots,\alpha_M)\in \big(\R_{>0}\big)^M$ be such that $\alpha_m-a_n>0$ for all $1\leq n\leq N$ and $1\leq m\leq M$.
Consider the setting as in Figure~\ref{FigSemiDiscreteDP}, where the horizontal axis is discrete on the left of $0$ and continuous on the right of $0$, while the vertical axis is discrete. In this semi-discrete setting we introduce randomness in the following way. For all $1\leq m\leq M$ and $1\leq n\leq N$ let $\omega_{-m,n}\sim -\ln\Gamma(\alpha_m-a_n)$ be independent log-Gamma random variables specified by the parameters $a,\alpha$; and for all $1\leq n\leq N$ let $B_n$ be independent Brownian motions with drift $a_n$. The $\omega_{-m,n}$ can be thought of as sitting at the lattice points $(-m,n)$ while the $B_n$ can be thought of as sitting along the horizontal rays from $(0,n)$. We denote by $\PP$ and $\EE$ the probability measure and expectation with respect to these random variables.

A discrete up-right path $\phi^d$ from $(i_1,j_1)$ to $(i_{\ell},j_{\ell})$ (written as $\phi^d:(i_1,j_1)\nearrow(i_{\ell},j_{\ell})$) is an ordered set of points $\big((i_1,j_1),(i_2,j_2),\ldots, (i_{\ell},j_{\ell})\big)$ with each $(i_k,j_k)\in \Z^2$ and each increment $(i_k,j_k) -(i_{k-1},j_{k-1})$ either $(1,0)$ or $(0,1)$. A semi-discrete up-right path $\phi^{sd}$ from $(0,n)$ to $(\tau,N)$ (written as $\phi^{sd}:(0,n)\nearrow (\tau,N)$) is a union of horizontal line segments $\big((0,n)\to (s_n,n)\big) \cup \big((s_n,n+1)\to (s_{n+1},n+1)\big)\cup \cdots \big((s_{N-1},N)\to (\tau,N)\big)$ where \mbox{$0\leq s_n<s_{n+1}<\cdots <s_{N-1}\leq \tau$}. It is convenient to think of $\phi^{sd}$ as a surjective non-decreasing function from $[0,\tau]$ onto $\{n,\ldots, N\}$.

As we are working with a mixture of a discrete and semi-discrete lattice, our up-right paths $\phi$ will be composed of discrete portions $\phi^d$ adjoined to a semi-discrete portions $\phi^{sd}$ in such a way that for some $1\leq n\leq N$, $\phi^d:(-M,1)\nearrow (-1,n)$ and $\phi^{sd}:(0,n)\nearrow (\tau,N)$. To such a path we associate an energy
\begin{equation}\label{eqEnergy}\begin{aligned}
E(\phi)&=\sum_{(i,j)\in \phi^d} \omega_{i,j} + \int_{0}^{\tau} dB_{\phi^{sd}(s)}(s) \\
&= \sum_{(i,j)\in \phi^d} \omega_{i,j} + B_n(s_n)+\big(B_{n+1}(s_{n+1})-B_{n+1}(s_n)\big)+\ldots+\big(B_N(\tau)-B_N(s_{N-1})\big).
\end{aligned}\end{equation}
This energy is random, as it is a function of the $\omega_{i,j}$ and $B_k$ random variables. We associate a Boltzmann weight $e^{E(\phi)}$ to each path $\phi$. The polymer measure on $\phi$ is proportional to this weight. The normalizing constant, or polymer partition function, is written as $\mathbf{Z}^{N,M}(\tau)$ and is equal to the integral of the Boltzmann weight over the background measure on the path space $\phi$. Explicitly it can be written as
$$
\mathbf{Z}^{N,M}(\tau) =\mathbf{Z}_1^{N,M}(\tau) = \sum_{n=1}^{N} \sum_{\phi^d:(-M,1)\nearrow (-1,n)} \, \int_{\phi^{sd}:(0,n)\nearrow(\tau,N)} e^{E(\phi)} \d\phi^{sd}
$$
where $\d\phi^{sd}$ represents the Lebesgue measure on the simplex $0\leq s_n<s_{n+1}<\cdots <s_{N-1}\leq \tau$ with which $\phi^{sd}$ is identified. Though we do not pursue it, let us note that for $M$ fixed, as a function of $\tau$ and $N$, $\mathbf{Z}^{N,M}(\tau)$ satisfies a semi-discrete SHE (for more on this, see~\cite[Section~5.2]{BC11} or~\cite[Section~6]{BCS12}).

In the spirit of the geometric lifting of the Robinson--Schensted--Knuth correspondence considered in~\cite{OCon09,COSZ11} (and for later use in the statement of Theorem \ref{ThmConnectWhitpoly}) we define a multi-path extension of this polymer and its partition function. For $M\geq 0$ fixed and $1\leq j\leq k\leq N$ define
$$
\mathbf{Z}_j^{k,M}(\tau) = \sum_{1\leq n_1<\cdots<n_j\leq k} \sum_{\substack{\phi_1^d,\ldots,\phi_j^d \\ \phi_a^d\cap \phi_b^d=\emptyset\textrm{ for }a\neq b \\ \phi_a^d:(-M,a)\nearrow (0,n_a)}} \, \int_{(\phi_1^{sd},\ldots \phi_j^{sd})\in D_j^{k,\tau}(n_1,\ldots, n_j)} e^{E(\phi_1)+\cdots+E(\phi_j)} \d\phi_1^{sd}\cdots \d\phi_j^{sd}
$$
where $D_j^{k,\tau}(n_1,\ldots, n_j)$ is the set of $(\phi_1^{sd},\ldots \phi_j^{sd})$ with $\phi_a^{sd}:(0,n_a)\nearrow (\tau,k-j+a)$ such that for all $a\neq b$ and $s\in [0,\tau]$, $\phi_a^{sd}(s)\neq \phi_b^{sd}(s)$ (i.e.\ the paths are non-intersecting). Each $\phi_a^{sd}$ can be identified via the jumping times $0\leq s^{(a)}_{n_a}<\cdots< s^{(a)}_{k-j+a}\leq \tau$, and $\d\phi_1^{sd}\cdots \d\phi_j^{sd}$ is the Lebesgue measure on the Euclidean set $\big(s^{(1)}_{n_1},\ldots, s^{(1)}_{k-j+1}, s^{(2)}_{n_2},\ldots, s^{(2)}_{k-j+2},\ldots, s^{(j)}_{n_j},\ldots, s^{(j)}_{N}\big)$. Note that $\mathbf{Z}^{N,M}(\tau) = \mathbf{Z}_1^{N,M}(\tau)$.

Finally, for $M\geq 0$ fixed and $1\leq j\leq k\leq N$ define
\begin{equation}\label{eqfreeenergy}
\mathbf{F}_j^{k,M}(\tau) = \ln\left(\frac{\mathbf{Z}_j^{k,M}(\tau)}{\mathbf{Z}_{j-1}^{k,M}(\tau)}\right)
\end{equation}
with the convention that $\mathbf{Z}_0^{k,M}(\tau)\equiv 1$.

The following Fredholm determinant formula for the Laplace transform of $\mathbf{Z}^{N,M}(\tau)$, proven in Section~\ref{SectSemiDirected}, is based on the developments of Sections~\ref{SectMacdonald} and~\ref{SectConvWhitt}. The restriction that $N\geq 9$ is likely purely technical and arises in the proof of Proposition~\ref{qFredDetThmbeta} as helpful in establishing certain convergence bounds. Since all of our asymptotics based off of this theorem involve sending $N$ to infinity, this restriction becomes inconsequential.

\begin{theorem}\label{ThmFormulaSemiDiscrete}
Fix $N\geq 9$, $M\geq 0$ and $\tau> 0$. Let $a=(a_1,\ldots,a_N)\in \R^N$ and \mbox{$\alpha = (\alpha_1,\ldots,\alpha_M)\in \big(\R_{>0}\big)^M$} be such that $\alpha_m-a_n>0$ for all $1\leq n\leq N$ and $1\leq m\leq M$.
For $1\leq m\leq M$ and $1\leq n\leq N$ let $\omega_{-m,n}\sim -\ln\Gamma(\alpha_m-a_n)$ be independent log-Gamma random variables and for all $1\leq n\leq N$ let $B_n$ be independent Brownian motions with drift $a_n$. Then for all $u\in \C$ with positive real part
\begin{equation*}
\EE\left[ e^{-u \mathbf{Z}^{N,M}(\tau)} \right] = \det(\Id+ K_{u})_{L^2(\Cv{a;\alpha;\varphi})}
\end{equation*}
where the operator $K_u$ is defined in terms of its integral kernel
\begin{equation*}
K_{u}(v,v') = \frac{1}{2\pi \I}\int_{\Cs{v}}\d s\, \Gamma(-s)\Gamma(1+s) \prod_{n=1}^{N}\frac{\Gamma(v-a_n)}{\Gamma(s+v-a_n)} \prod_{m=1}^M\frac{\Gamma(\alpha_m-v-s)}{\Gamma(\alpha_m-v)}
\frac{ u^s e^{v\tau s+\tau s^2/2}}{v+s-v'}.
\end{equation*}
The contour $\Cv{a;\alpha;\varphi}$ is given in Definition~\ref{DefCaCsdefBis} with any $\varphi\in(0,\pi/4)$, as is the contour $\Cs{v}$.
\end{theorem}

\begin{remark}
Let us make clear our usage of the notion of a Fredholm determinant. Fix a Hilbert space $L^2(X,\mu)$ where $X$ is a measure space and $\mu$ is a measure on $X$. When $X=\Gamma$, a simple (anticlockwise oriented) smooth contour in $\C$, we write $L^2(\Gamma)$ where for $z\in \Gamma$, $\d\mu(z)$ is understood to be $\tfrac{\d z}{2\pi \I}$. When $X$ is the product of a discrete set $D$ and a contour $\Gamma$, $\d\mu$ is understood to be the product of the counting measure on $D$ and $\frac{\d z}{2\pi \I}$ on $\Gamma$. Let $K$ be an {\it integral operator} acting on $f(\cdot)\in L^2(X)$ by $Kf(x) = \int_{X} K(x,y)f(y)\,\d\mu(y)$. $K(x,y)$ is called the {\it kernel} of $K$ and we will assume throughout that $K(x,y)$ is continuous in both $x$ and $y$. Assuming its convergence, the {\it Fredholm determinant expansion} of $\Id+K$ is defined as
\begin{equation*}
\det(\Id+K)_{L^2(X)} = 1+\sum_{n=1}^{\infty} \frac{1}{n!} \int_{X} \cdots \int_{X} \det\left[K(x_i,x_j)\right]_{i,j=1}^{n} \prod_{i=1}^{n} \d\mu(x_i).
\end{equation*}
Note that we do not require $K$ to be trace-class, and only use the notation $\det(\Id+K)_{L^2(X)}$ as a shorthand for the right-hand side of the above equation.
\end{remark}

\begin{remark}
The condition that $\tau>0$ is important to ensure that the integral defining the kernel $K_{u}$ is finite (cf.\ the estimates in Section \ref{prop3sec}). It seems that as long as $M\geq N$, it is possible to take the limit $\tau\to 0$. By continuity of the function $\mathbf{Z}^{N,M}(\tau)$ in $\tau$, this provides a Fredholm determinant formula for $\mathbf{Z}^{N,M}(0)$, or in other words, the  log-gamma polymer partition function. A similar formula to this appeared in \cite{BCR13}, though involving a small (finite) contour in place of $\Cv{a;\alpha;\varphi}$ (see also \cite{LeDoussalThierry}). That formula was used therein for asymptotics of the free energy, though only for a certain range of parameters. The large (infinite) contour formula we arrive at here may be useful in removing that parameter range restriction in a parallel manner as \cite{BCF12} used such contours to remove similar restrictions present in \cite{BC11}.
\end{remark}

The contours in Theorem~\ref{ThmFormulaSemiDiscrete} are defined as follows.
\begin{definition}\label{DefCaCsdefBis}
Let $a=(a_1,\ldots,a_N)\in \R^N$ and \mbox{$\alpha = (\alpha_1,\ldots,\alpha_M)\in \big(\R_{>0}\big)^M$} be such that $\alpha_m-a_n>0$ for all $1\leq n\leq N$ and $1\leq m\leq M$. Set $\mu=\tfrac{1}{2}\max(a)+\tfrac{1}{2}\min(\alpha)$ and $\eta = \tfrac{1}{4}\max(a)+\tfrac{3}{4}\min(\alpha)$. Then, for all $\varphi\in (0,\pi/4)$, we define the contour \mbox{$\Cv{a;\alpha;\varphi}=\{\mu+e^{\I (\pi+\varphi)}y\}_{y\in \Rplus}\cup \{\mu+e^{\I(\pi-\varphi)}y\}_{y\in \Rplus}$}. The contours are oriented so as to have increasing imaginary part. For every \mbox{$v\in \Cv{a;\alpha;\varphi}$} we choose $R=-\Re(v)+\eta$, $d>0$, and define a contour $\Cs{v}$ depending on the value of $R$ as follows. If $R\le1/2$, then $\Cs{v}$ is the vertical line $R+\I\R$; if $R>1/2$, then $\Cs{v}$ goes by straight lines from $R-\I \infty$, to $R-\I d$, to $1/2-\I d$, to $1/2+\I d$, to $R+\I d$, to $R+\I\infty$. The parameter $d$ is taken small enough so that $v+\Cs{v}$ does not intersect $\Cv{a;\alpha;\varphi}$. See Figure~\ref{FigContoursSemiDiscrete} for an illustration.
\end{definition}
\begin{figure}
\begin{center}
\psfrag{Cv}[lb]{$\Cv{a;\alpha;\varphi}$}
\psfrag{v+Cs}[lb]{$v+\Cs{v}$}
\psfrag{Cs}[lb]{$\Cs{v}$}
\psfrag{mu}[cb]{$\mu$}
\psfrag{eta}[cb]{$\eta$}
\psfrag{v}[cb]{$v$}
\psfrag{R}[cb]{$R$}
\psfrag{2d}[lb]{$2d$}
\psfrag{alpha}[cb]{$\alpha$'s}
\psfrag{a}[cb]{$a$'s}
\psfrag{0}[cb]{$0$}
\psfrag{phi}[lb]{$\varphi$}
\includegraphics[height=5cm]{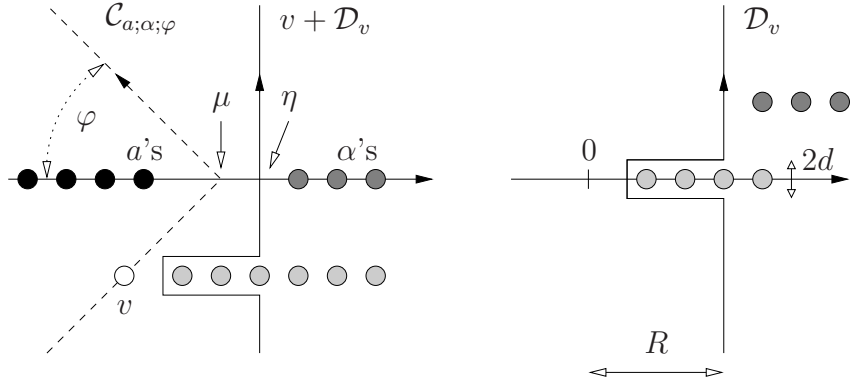}
\end{center}
\caption{(Left) The contour $\Cv{\eta;\varphi}$ (dashed) where the
black dots symbolize the set of singularities of $K_u(v,v')$ in $v$ at $\cup_{1\leq n\leq N}\{a_n,a_n-1,\dots\}$ coming from the factors $\Gamma(v-a_n)$.
The contour $v+\Cs{v}$ is the solid line.
(Right) The contour $\Cs{v}$ where the light gray dots are the singularities at $\{1,2,\dots\}$ and the dark gray dots are those at $\cup_{1\leq m \leq M}\{\alpha_m-v,\alpha_m+1-v,\dots\}$ coming from \mbox{$\Gamma(\alpha_m-v-s)$}.}
\label{FigContoursSemiDiscrete}
\end{figure}

To eventually access the stationary KPZ equation, we need to choose our $a$ and $\alpha$ parameters appropriately.
\begin{definition}\label{defnearstat}
For what follows, we set $M=1$, $a_1=a$, $a_n\equiv 0$ for $n>1$, $\alpha_1=\alpha>a$ and define $\Zsd(\tau,N)$ as the semi-discrete directed random polymer partition function in which the weight $\omega_{-1,1}$ is replaced by zero.
\end{definition}
\begin{corollary}\label{CorollaryForStationarity}
For $\alpha>a$,
\begin{equation}\label{Besseltransform}
\EE\left[2\big(u\,\Zsd(\tau,N)\big)^{\frac{\alpha-a}2}\BesselK_{-(\alpha-a)}\Big(2\sqrt{u\,\Zsd(\tau,N)}\Big)\right]=\Gamma(\alpha-a)\,\EE\left[e^{-u \mathbf{Z}^{N,1}(\tau)} \right]
\end{equation}
where $\BesselK_\nu$ is the modified Bessel function of order $\nu$, cf.~\cite{AS84}.
\end{corollary}
\begin{proof}
Since all polymer paths $\phi$ must go through the point $(-1,1)$, it follows that
\begin{equation}\label{relZsd}
\mathbf{Z}^{N,1}(\tau)=e^{\omega_{-1,1}}\Zsd(\tau,N).
\end{equation}
By the definition of the log-gamma distribution,
$e^{-\omega_{-1,1}}$ has gamma distribution with parameter $\alpha-a$ and density $x^{\alpha-a-1}e^{-x}/\Gamma(\alpha-a)$ on $\R_+$. Using this along with the independence of $\omega_{-1,1}$ and $\Zsd(\tau,N)$, we may rewrite the Laplace transform of $\mathbf{Z}^{N,1}(\tau)$ using \eqref{relZsd} as
\begin{equation*}
\begin{aligned}
\EE\left[e^{-u\mathbf{Z}^{N,1}(\tau)}\right]&=\EE\left[e^{-ue^{\omega_{-1,1}}\Zsd(\tau,N)}\right]\\
&=\EE\left[\int_0^\infty\frac{e^{-u\Zsd(\tau,N)x^{-1}}x^{\alpha-a-1}e^{-x}}{\Gamma(\alpha-a)}\d x\right]\\
&=\frac1{\Gamma(\alpha-a)}\EE\left[2\big(u\Zsd(\tau,N)\big)^{\frac{\alpha-a}2}\BesselK_{-(\alpha-a)}\left(2\sqrt{u\Zsd(\tau,N)}\right)\right].
\end{aligned}
\end{equation*}
The last equation follows from the identity
$$\int_0^\infty e^{-x-cx^{-1}}x^{-\nu-1}\d x=2c^{d/2}\BesselK_{\nu}\big(2\sqrt c\big),\qquad c>0,$$
which can be derived from the integral representation 9.6.24 of the modified Bessel function in~\cite{AS84}.
\end{proof}

\begin{remark}\label{remstat}
If we further specialize $\Zsd(\tau,N)$ so that $\alpha=a$, we arrive at a model which is stationary. This fact can be gathered from the results of~\cite{SV10} and is explicitly explained in Appendix~\ref{AppStationary}.
\end{remark}

At this point we have a choice to make. We seek to study the stationary SHE/KPZ equation. One way to access that is through a suitable scaling limit of the stationary semi-discrete directed random polymer with log-gamma boundary sources, described in Remark~\ref{remstat}. Alternatively, we could take a suitable limit of the semi-discrete directed random polymer with log-gamma boundary sources with $\alpha>a$. This leads the SHE/KPZ equation with nearly stationary initial data (in fact, two sided Brownian initial data with drifts $\beta>b$). Subsequently, we can take $\beta \to b$ to recover the stationary SHE/KPZ equation. We opt for taking the second route. In either case, there is a technical challenge which we must overcome. Let us presently illustrate this issue for taking the limit $\alpha\to a$, even though it is the other route which we actually pursue. The expectation in the right-hand side of \eqref{Besseltransform} is given by a Fredholm determinant in Theorem~\ref{ThmFormulaSemiDiscrete}. In the
limit $\alpha\to a$, this Fredholm determinant goes to zero linearly in $\alpha-a$, compensating the divergence of $\Gamma(\alpha-a)$ so as to have a non-trivial limit. To take this limit, however, it is necessary to analytically continue our formulas in the quantity $\alpha-a$ (initially in $\R_{>0}$) and use uniqueness of analytic continuations to justify the extension to $\alpha-a=0$.

\begin{remark}
In principle, Theorem~\ref{ThmFormulaSemiDiscrete} could be utilized for a variety of other asymptotics which we do not pursue here. For instance, it should be possible to access some of the one-point probability distribution functions which were previously studied in the case of last passage percolation with boundary conditions in~\cite{BP07}. It should also be possible to take limits to study analogous situations for the SHE/KPZ equation which would involve two-sided version of the initial data considered in~\cite{BCF12} with the inclusion of extra log-gamma weights (see also,~\cite{IS12} for a special case of such initial data).
\end{remark}

\subsection{SHE/KPZ equation with two-sided Brownian initial data}
Theorem~\ref{ThmDiscreteToContinuous} (a result quoted from~\cite{QMR12}) describes the special scaling under which the ($M=1$, $a_1=a$, $a_n\equiv 0$ for $n>1$, and $\alpha_1=\alpha>a$) semi-discrete directed random polymer with log-gamma boundary sources converges to the SHE/KPZ equation with two-sided Brownian motion initial data. The following analogue of Corollary~\ref{CorollaryForStationarity} is proven in Section~\ref{SectCDRP}.
\begin{theorem}\label{ThmFormulaContinuous}
Let us denote by $\mathcal Z_{b,\beta}(T,X)$ the solution to the SHE/KPZ equation with initial data $\mathcal{Z}_0(X)= \exp(B(X))$, where $B(X)$ is a two-sided Brownian motion with drift $\beta$ to the left of $0$ and drift $b$ to the right of $0$, with $\beta>b$, that is, \mbox{$B(X)=\mathbf{1}_{X\leq 0} \big(B^l(X)+\beta X\big) + \mathbf{1}_{X>0} \big(B^r(X)+b X\big)$} where $B^l:(-\infty,0]\to \R$ is a Brownian motion without drift pinned at $B^l(0)=0$, and $B^r:[0,\infty)\to \R$ is an independent Brownian motion pinned at $B^r(0)=0$.
Then, for $S>0$,
\begin{multline}\label{Bessel=Fredholm}
\EE\left[2\left(Se^{\frac{X^2}{2T}+\frac T{24}}\mathcal Z_{b,\beta}(T,X)\right)^{\frac{\beta-b}2}\BesselK_{-(\beta-b)}\left(2\sqrt{Se^{\frac{X^2}{2T}+\frac T{24}}\mathcal Z_{b,\beta}(T,X)}\right)\right]\\
=\Gamma(\beta-b)\det(\Id-K_{b+X/T,\beta+X/T})_{L^2(\R_+)}
\end{multline}
where $\BesselK_\nu(z)$ is the modified Bessel function of order $\nu$ and the kernel on the right-hand side is given by
\begin{equation}\label{defKbbeta}
K_{b,\beta}(x,y)=\frac1{(2\pi\I)^2}\int\d w\int\d z \frac{\sigma\pi S^{\sigma(z-w)}}{\sin(\sigma\pi(z-w))} \frac{e^{z^3/3-zy}}{e^{w^3/3-wx}}
\frac{\Gamma(\beta-\sigma z)}{\Gamma(\sigma z-b)} \frac{\Gamma(\sigma w-b)}{\Gamma(\beta-\sigma w)}
\end{equation}
where
\begin{equation}\label{defsigma}
\sigma=(2/T)^{1/3}.
\end{equation}
The integration contour for $w$ is from $-\frac1{4\sigma}-\I\infty$ to $-\frac1{4\sigma}+\I\infty$ and crosses the real axis between $b$ and $\beta$.
The other contour for $z$ goes from $\frac1{4\sigma}-\I\infty$ to $\frac1{4\sigma}+\I\infty$, it also crosses the real axis between $b$ and $\beta$ and it does not intersect the contour for $w$.
\end{theorem}

\begin{remark}\label{comparerem}
Let us compare the above result to that derived (non-rigorously via the replica method) in~\cite[Proposition 1]{IS12}. The initial data considered therein is two-sided Brownian, plus a log-gamma distributed (independent) height shift. We may use Theorem~\ref{ThmFormulaContinuous} and reverse the proof of Corollary~\ref{CorollaryForStationarity} so as to prove a one-point formula for this initial data. Inspection reveals that the resulting formula matches that of~\cite{IS12}. As we soon explain, in order to go from this formula to the stationary initial data formula requires work and the final formula shown in~\cite{IS12} is not as readily compared to the final formula proved herein as Theorem~\ref{ThmFormulaStationary}.
\end{remark}

\subsection{SHE/KPZ equation with stationary initial data}
The stationary initial conditions for the KPZ equation are the two-sided Brownian motions with a fixed drift. For the SHE this means to let $\mathcal{Z}_0(X)=\exp(B(X))$ with $B$ a two-sided Brownian motion with drift $b\in \R$. Call the resulting solution to the SHE $\mathcal{Z}_b(T,X)$. We can get the result by carefully taking the $\beta\to b$ limit in Theorem~\ref{ThmFormulaContinuous}. This limit is accomplished by analytically continuing the expressions on both sides of \eqref{Bessel=Fredholm}. In order to be able to state our main result of the paper we need a few notations.

\begin{definition}\label{longdef}
For $b\in\left(-\frac14,\frac14\right)$, define on $\R_+$ the function
\begin{equation}\label{defqss}\begin{aligned}
q_{b}(x)&=\frac1{2\pi\I}\int_{-\frac1{4\sigma}+\I\R}\d w \frac{\sigma\pi S^{b-\sigma w}}{\sin(\pi(b-\sigma w))} e^{-w^3/3+wx}\frac{\Gamma(\sigma w-b)}{\Gamma(b-\sigma w)}
\end{aligned}\end{equation}
and for $b\in\R$, let
\begin{equation}\label{defrs}
r_b(x)=e^{b^3/(3\sigma^3)-bx/\sigma}.
\end{equation}
Further, for $b\in\left(-\frac14,\frac14\right)$, define the kernel
\begin{equation}\label{defbarKbb}
\bar K_{b}(x,y)=\frac1{(2\pi\I)^2}\int_{-\frac1{4\sigma}+\I\R}\d w\int_{\frac1{4\sigma}+\I\R}\d z \frac{\sigma\pi S^{\sigma(z-w)}}{\sin(\sigma\pi(z-w))} \frac{e^{z^3/3-zy}}{e^{w^3/3-wx}}
\frac{\Gamma(b-\sigma z)}{\Gamma(\sigma z-b)} \frac{\Gamma(\sigma w-b)}{\Gamma(b-\sigma w)}.
\end{equation}
Finally, letting $\gamma_{\rm E}=0.577\ldots$ represent the Euler--Mascheroni constant, define
\begin{equation}\label{defz}
\begin{aligned}
\Xi(S,b,\sigma)=&-\det(\Id-\bar K_{b})\Big[b^2/\sigma^2+\sigma(2\gamma_{\rm E}+\ln S)\\
&+\big\langle(\Id-\bar K_{b})^{-1}(\bar K_{b}r_{-b}+q_{b}),r_b\big\rangle +\big\langle(\Id-\bar K_{b})^{-1}(r_{-b}+q_{b}),q_{-b}\big\rangle\Big].
\end{aligned}
\end{equation}
where the determinants and scalar products are all meant in $L^2(\R_+)$.
\end{definition}
\begin{remark}
By using the general identity
$$\det(\Id-K)\big\langle(\Id-K)^{-1}f,g\big\rangle=\det(\Id-K)-\det(\Id-K-f\otimes g),$$
it is also possible to write $\Xi$ as a linear combination of Fredholm determinants:
\begin{multline*}
\Xi(S,b,\sigma)=\det\big(\Id-\bar K_{b}-(\bar K_{b}r_{-b}+q_{b})\otimes r_b\big)+\det\big(\Id-\bar K_{b}-(r_{-b}+q_{b})\otimes q_{-b}\big)\\
-\det(\Id-\bar K_{b})\big[2+b^2/\sigma^2+\sigma(2\gamma_{\rm E}+\ln S)\big].
\end{multline*}
Note that $\Xi$ also depends on $S$ implicitly through $\bar K_{b}$ and $q_{\pm b}$.
The right-hand side of \eqref{defz} is well-defined for any admissible choice of the parameters, see Remark~\ref{rem:barKproduct}.
\end{remark}

The following result (which implies Theorem~\ref{ThmFormulaStationaryIntroVersion} when $b=X=0$) is proven in Section~\ref{SectStatSHE}.
\begin{theorem}\label{ThmFormulaStationary}
Let $\mathcal{Z}_b(T,X)$ be the solution to the SHE with initial data $\mathcal{Z}_0(X)=e^{B(X)}$ with $B$ a two-sided Brownian motion with $B(0)=0$ and drift $b\in \R$.
Let $\BesselK_0$ denote the modified Bessel function and consider $X\in\R$, $T>0$ and $b\in \R$ such that $b+\frac XT\in\left(-\frac14,\frac14\right)$. For $S>0$,
\begin{equation}\label{K0transform}
\EE\left[2\sigma \BesselK_0\left(2\sqrt{Se^{\frac{X^2}{2T}+\frac T{24}}\mathcal Z_b(T,X)}\right)\right]=\Xi\left(S,b+\frac XT,\sigma\right)
\end{equation}
where the function $\Xi$ is defined in \eqref{defz}.
\end{theorem}
We remark that the condition $b+\frac XT\in\left(-\frac14,\frac14\right)$ could be weakened to $b+\frac XT\in(-1,1)$ in a slightly more technical way, but the formulation \eqref{defz} is not convergent outside the latter regime. See Remark~\ref{rem:bcondition} for more details.

The integral transform that appears on the left-hand side of \eqref{K0transform} is the Mellin transform~\cite{NY92} of the stationary (drift $b$) KPZ equation solution
\begin{equation}\label{defF}
\mathcal{H}_b(T,X)=\ln\mathcal{Z}_b(T,X).
\end{equation}
It is possible to recover the distribution function from \eqref{K0transform} using a double inverse Mellin transform (proven in Appendix~\ref{AppMellin}).
\begin{proposition}\label{PropInverseMelling}
Consider $T>0$, $X\in \R$, and $b\in \R$ such that $b+\frac XT\in\left(-\frac14,\frac14\right)$. For any $r\in\R$,
\begin{multline*}
\PP\left(\frac{\mathcal H_b(T,X)+\frac T{24}+\frac{X^2}{2T}}{(T/2)^{1/3}}\le r\right)\\
=\frac{1}{\sigma^2}\frac{1}{2\pi\I}\int_{-\delta+\I\R} \frac{\d \xi}{\Gamma(-\xi)\Gamma(-\xi+1)} \int_{\R} \d x\, e^{x\xi/\sigma} \Xi\left(e^{-\frac{x+r}\sigma},b+\frac XT,\sigma\right)
\end{multline*}
for any $\delta>0$.
\end{proposition}
\begin{proof}
This follows from applying Proposition~\ref{prop:mellin} with $R=\sigma(\mathcal H_b(T,X)+T/24+X^2/(2T))$ and $x=-\sigma\ln S$.
The finite negative exponential moment $\EE(\exp(-\delta R/\sigma))$ is ensured by Lemma~\ref{lem:Ftailbound} for any $\delta>0$.
\end{proof}

This formula should be compared to~\cite[Theorem~2]{IS12} in which the non-rigorous replica method was utilized to study the stationary KPZ equation (see Remark \ref{comparerem}).

\begin{remark}\label{rem:shiftarg}
Comparing \eqref{K0transform} for different values of $b$ and $X$ shows
\begin{equation}\label{shiftarg}
\mathcal Z_{b-X/T}(T,X)=e^{-\frac{X^2}{2T}}\mathcal Z_b(T,0)
\end{equation}
in distribution. This rotational invariance property can be explained directly from the definition of the SHE, as in~\cite[Section~3.2]{BCF12}.
\end{remark}

In the large $T$ limit one expects, by the universality belief, that limiting one-point probability distribution functions for the KPZ equation should converge to those previously determined in the context of TASEP or in the polynuclear growth model for analogous types of initial data~\cite{BR00,SI04,PS02b,FS05a}. Here we use the same notations as in~\cite[Theorem~1.2]{BFP09} specialized to the one-point setting.
\begin{definition}\label{anotherlongone}
Recall the Airy function $\Ai$, cf.~\cite{AS84}. For $\tau,s\in\R$, define
\begin{align*}
\mathcal R&=s+e^{-\frac23\tau^3}\int_s^\infty\d x\int_0^\infty\d y \Ai(x+y+\tau^2)\,e^{-\tau(x+y)},\\
\Psi(y)&=e^{\frac23\tau^3+\tau y}-\int_0^\infty\d x \Ai(x+y+\tau^2)\,e^{-\tau x},\\
\Phi(x)&=e^{-\frac23\tau^3}\int_0^\infty\d\lambda\int_s^\infty dy \Ai(x+\tau^2+\lambda)\Ai(y+\tau^2+\lambda)\,e^{-\tau y}
-\int_0^\infty\d y\Ai(y+x+\tau^2)\,e^{\tau y}.
\end{align*}
Let $P_s$ be the projection operator $P_s(x)=\Id_{\{x>s\}}$, the Airy kernel with shifted entries by
\begin{equation}\label{defKAi}
\wh
K_{\Ai}(x,y)=\int_0^\infty\d\lambda\Ai(x+\lambda+\tau^2)\Ai(y+\lambda+\tau^2),
\end{equation}
and define the function
\begin{equation}\label{defg}
g(\tau,s)=\mathcal R-\big\langle(\Id-P_s\wh K_{\Ai}P_s)^{-1}P_s\Phi,P_s\Psi\big\rangle.
\end{equation}
Finally, let
\begin{equation}\label{defBFP}
F_\tau(r)=\frac{\partial}{\partial r}\left(g(\tau,r)\det\left(\Id-P_r\wh K_{\Ai}P_r\right)_{L^2(\R)}\right).
\end{equation}
\end{definition}

In the large $T$ limit, the fluctuations of $\mathcal H_b(T,X)$ are governed by $F_\tau$, as shown in the following result (proven in Section~\ref{SectUniversality}).
\begin{theorem}\label{CorUniversality}
Let $b\in(-\frac14,\frac14)$ be fixed and consider any $\tau\in\R$. Define $\sigma=(2/T)^{1/3}$ and consider the scaling
\begin{equation}\label{XTscaling}
X=-bT+\frac{2\tau}{\sigma^2}.
\end{equation}
Then, for any $r\in\R$,
\begin{equation*}
\lim_{T\to\infty}\PP\left(\frac{\mathcal H_b(T,X)+\frac T{24}(1+12b^2)-2^{1/3}b\tau T^{2/3}}{(T/2)^{1/3}}\le r\right)=F_\tau(r).
\end{equation*}
\end{theorem}

\section{Ascending $q$-Whittaker processes}\label{SectMacdonald}

\subsection{Defining the processes}

The ascending $q$-Whittaker processes $\M_{\tilde a;\rho}$ are special cases of the ascending Macdonald processes~\cite{BC11} in which the Macdonald parameters $t=0$ and $q\in (0,1)$. The $q$-Whittaker measures $\MM_{\tilde a;\rho}$ are marginals of the ascending $q$-Whittaker processes. We provide a brief account of these objects as well as the $q$-Whittaker (or Macdonald $t=0$) symmetric functions used to define them. For a more involved discussion and background, see~\cite[Sections~2.2 and 3.1]{BC11}.

Fix $N\geq 1$. The {\it $q$-Whittaker process} $\M_{\tilde a;\rho}$ is a probability measure on sequences of interlacing partitions
\begin{equation*}
\varnothing \prec \lambda^{(1)}\prec \lambda^{(2)}\prec\cdots \prec\lambda^{(N)}
\end{equation*}
(equivalently Gelfand--Tsetlin patterns, or column-strict Young tableaux) parameterized by positive reals\footnote{The reason we use tildes for the parameters of this measure is because they will eventually be expressed in terms of parameters without tildes, when we perform a $q\to 1$ scaling limit to their Whittaker counterparts.} $\tilde a=\{\tilde a_1,\ldots,\tilde a_N\}$, a single $q$-Whittaker nonnegative specialization $\rho$ of the algebra of symmetric function, and the Macdonald parameter $q\in (0,1)$. The probability measure is given by
\begin{equation*}
\M_{\tilde a;\rho}\big(\la^{(1)},\dots,\la^{(N)}\big)= \frac{P_{\la^{(1)}}(\tilde a_1)P_{\la^{(2)}/\la^{(1)}}(\tilde a_2)\cdots P_{\la^{(N)}/\la^{(N-1)}}(\tilde a_N) Q_{\la^{(N)}}(\rho)}{\Pi(\tilde a_1,\ldots, \tilde a_N;\rho)}\,.
\end{equation*}
We write $\EE_{\M_{\tilde a,\rho}}$ for the expectation with respect to this measure (though sometimes may drop the $\M_{\tilde a,\rho}$ subscript when it is clear).

Some explanation of notation is in order. A partition $\lambda$ is an integer sequence \mbox{$\lambda=(\lambda_1\geq \lambda_2\geq \ldots \geq 0)$} with finitely many nonzero entries, and we say that $\mu \prec \lambda$ if the two partitions interlace: $\mu_i\leq \lambda_i \leq \mu_{i-1}$ for all meaningful $i$'s. In Young diagram terminology, $\mu\prec\lambda$ is equivalent to saying that the skew partition $\lambda/\mu$ is a horizontal strip.

The functions $P_{\bullet}$ and $Q_{\bullet}$ are $q$-Whittaker symmetric functions (i.e.\ Macdonald symmetric functions with parameter $t=0$) which are indexed by (skew) partitions and implicitly depend on the Macdonald parameter $q\in (0,1)$. The remarkable properties of Macdonald symmetric functions are developed in~\cite[Section~VI]{Mac79}, and all of the relevant facts to which we appeal are also reviewed in~\cite[Section~2.1]{BC11}. The evaluation of a $q$-Whittaker symmetric function on a positive variable $\tilde a$ (as in $P_{\lambda/\mu}(\tilde a)$) means to restrict the function to a single nonzero variable and then substitute the value $\tilde a$ in for that variable. This is a special case of a $q$-Whittaker nonnegative specialization $\rho$ which is an algebra homomorphism of the algebra of symmetric functions $\Sym\to \C$ that takes skew $q$-Whittaker symmetric functions to nonnegative real numbers (notation: $P_{\lambda/\mu}(\rho)\geq 0$, $Q_{\lambda/\mu}(\rho)\geq 0$ for any partitions $\
\lambda$ and $\mu$). Restricting the $q$-Whittaker symmetric
functions to a finite number of nonzero variables (i.e.\ considering $q$-Whittaker polynomials) and then substituting nonnegative numbers for these variables constitutes such a specialization.
We will work with a more general class of specializations which can be thought of as unions and limits of such finite length specializations as well as {\it dual} specializations. Let $\tilde \alpha=\{\tilde \alpha_i\}_{i\ge 1}$, $\tilde \beta=\{\tilde \beta_i\}_{i\ge 1}$, and $\tilde \gamma$ be nonnegative reals such that $\sum_{i=1}^\infty(\tilde \alpha_i+\tilde \beta_i)<\infty$. Let $\rho=\rho(\tilde \alpha;\tilde \beta;\tilde \gamma)$ be a specialization of $\Sym$ defined by
\begin{equation}\label{specdefeqn}
\sum_{n\ge 0} g_n(\rho) u^n= e^{\tilde \gamma u} \prod_{i\ge 1} \frac{1+\tilde \beta_i u}{(\tilde \alpha_i u;q)_\infty}=: \Pi\big(u;\rho(\tilde \alpha;\tilde \beta;\tilde \gamma)\big).
\end{equation}
Here $u$ is a formal variable, $g_n=Q_{(n)}$ is the $q$-analog of the complete homogeneous symmetric function $h_n$, and $(\tilde a;q)_n=\prod_{i=0}^{n-1} (1-q^i \tilde a)$ is the $q$-Pochhammer symbol (with obvious extension when $n=\infty$). Since $\{g_n\}_{n\geq 0}$ form an algebraic basis of $\Sym$, this uniquely defines the specializations $\rho$. Such $\rho$ are $q$-Whittaker nonnegative (see~\cite[Section~2.2.1]{BC11} for more details). Alternatively, one can specify the above specializations $\rho(\tilde \alpha;\tilde \beta;\tilde \gamma)$ in terms of the values they take on the Newton power sum symmetric functions $p_k = \sum_i (x_i)^k$ via
\begin{equation*}
\begin{aligned}
p_1\big(\rho(\tilde \alpha;\tilde \beta;\tilde \gamma)\big) &\mapsto (1-q)\tilde \gamma + \sum_i \left(\tilde \alpha_i+(1-q)\tilde \beta_i\right),\\
p_k\big(\rho(\tilde \alpha;\tilde \beta;\tilde \gamma)\big) &\mapsto \sum_i \left((\tilde \alpha_i)^k+(-1)^{k-1}(1-q^k)(\tilde \beta_i)^k\right), \qquad k\geq 2.
\end{aligned}
\end{equation*}
We can also express $\Pi(u;\rho)$ in terms of these Newton power sum symmetric functions as
$$
\Pi(u;\rho) = \exp\left(\sum_{k=1}^{\infty} \frac{u^k\, p_k(\rho)}{(1-q^k)k}\right).
$$
When it is clear which specialization we are discussing, we will just write $\rho$ rather than $\rho(\tilde \alpha;\tilde \beta;\tilde \gamma)$.

The normalization for the ascending $q$-Whittaker process is given by
\begin{equation*}
\sum_{\la^{(N)}}P_{\la^{(N)}}(\tilde a) Q_{\la^{(N)}}(\rho) = \Pi(\tilde a;\rho) = \prod_{n=1}^{N} \Pi(\tilde a_n;\rho),
\end{equation*}
as follows from a generalization of Cauchy's identity for Schur functions (corresponding to the case $q=0$). It is not hard to see that for $\rho=\rho(\tilde \alpha;\tilde\beta;\tilde\gamma)$ the condition of the partition function $\Pi(\tilde a;\rho)$ to be finite is equivalent to $\tilde a_n\tilde \alpha_m<1$ for all $n,m$, and hence we will always assume that this holds.

The projection of $\M_{a;\rho}$ to a single partition $\lambda^{(k)}$, $k\in\{1,\dots,N\}$, is the {\it $q$-Whittaker measure} given by
\begin{equation*}
\MM_{\tilde a;\rho}\big(\lambda^{(k)}\big)= \frac{P_{\lambda^{(k)}}(\tilde a_1,\ldots,\tilde a_k) Q_{\lambda^{(k)}}(\rho)} {\Pi(\tilde a_1,\ldots,\tilde a_k;\rho)}\,.
\end{equation*}

In what follows we will be concerned primarily with the marginal distribution of $\lambda^{(N)}_1$.

\subsection{Fredholm determinant formula}

In order to state the main theorem of the section, we must specify parameters for the $q$-Whittaker measure as well as various contours which participate.

\begin{definition}\label{defParams}
For $N\geq 1$ consider non-negative reals $\tilde a=\{\tilde a_1,\ldots, \tilde a_N\}$. We will work with $q$-Whittaker non-negative specializations $\rho = \rho(\tilde \alpha;\tilde \beta;\tilde \gamma)$ as in \eqref{specdefeqn} where $\tilde \alpha=\{\tilde \alpha_1,\ldots,\tilde \alpha_{M_{\alpha}}\}$, $\tilde \beta=\{\tilde \beta_1,\ldots,\tilde \beta_{M_{\beta}}\}$ and $\tilde \gamma$ satisfy that for all $i$, $\tilde \alpha_i,\tilde \beta_i,\tilde\gamma \geq 0$ and $\max(\tilde \alpha),\max(\tilde \beta) <\min(\tilde a^{-1})$, where $\tilde a^{-1} = \{\tilde a_1^{-1},\ldots, \tilde a_N^{-1}\}$.
\end{definition}

\begin{definition}\label{CwPredef}
For $\tilde a$, $\tilde \alpha$ and $\tilde \beta$ as in Definition~\ref{defParams} and an angle $\varphi\in (0,\pi/2]$ define $\CwPre{\tilde a; \tilde \alpha,\tilde \beta; \varphi}=\{\tilde \mu+e^{-\I \varphi\sign(y)}y,y\in \R\}$ (oriented so as to have decreasing imaginary part) where $\tilde \mu = \tfrac{1}{2}\max(\tilde \alpha\cup \tilde \beta) +\tfrac{1}{2}\min(\tilde a^{-1})$.
For $w\in \CwPre{\tilde a; \tilde \alpha,\tilde \beta;\varphi}$, we choose $R>0$ so that it satisfies the equation $|w|q^R=\tfrac34\max(\tilde\alpha\cup\tilde\beta)+\tfrac14\min(\tilde a^{-1})$.
We choose $d>0$ and the contour $\CsPre{w}$ as follows.
If $R\le1/2$, the $\CsPre{w}$ is the vertical line $R+\I\R$; if $R>1/2$, then $\CsPre{w}$ goes by straight lines from $R-\I \infty$, to $R-\I d$, to $1/2-\I d$, to $1/2+\I d$, to $R+\I d$, to $R+\I\infty$.
We choose $R$ and $d$ such that the following holds: For all $s\in\CsPre{w}$, $q^s w$ lies to the left of $\CwPre{\tilde a;\tilde \alpha,\tilde \beta; \varphi}$ and encloses all $\tilde \alpha$ and $\tilde\beta$; and for $|w|$ large, $R\approx \ln|w|$ and $d\approx |w|^{-1}$ (here $\approx$ means up to a positive constant bounded from zero and infinity).
See Figure~\ref{FigqWhitContours} for an illustration of these contours.
\end{definition}

\begin{figure}
\begin{center}
\psfrag{w}[cc]{$w$}
\psfrag{qsw}[lc]{$q^s w$}
\psfrag{beta}[cc]{$-\tilde\beta$}
\psfrag{alpha}[lc]{$\tilde\alpha$}
\psfrag{ainv}[cc]{$\tilde a^{-1}$}
\psfrag{phi}[lc]{$\varphi$}
\psfrag{Dw}[cc]{}
\psfrag{C}[lc]{$\CwPre{\tilde a;\tilde \alpha,\tilde \beta; \varphi}$}
\includegraphics[height=7cm]{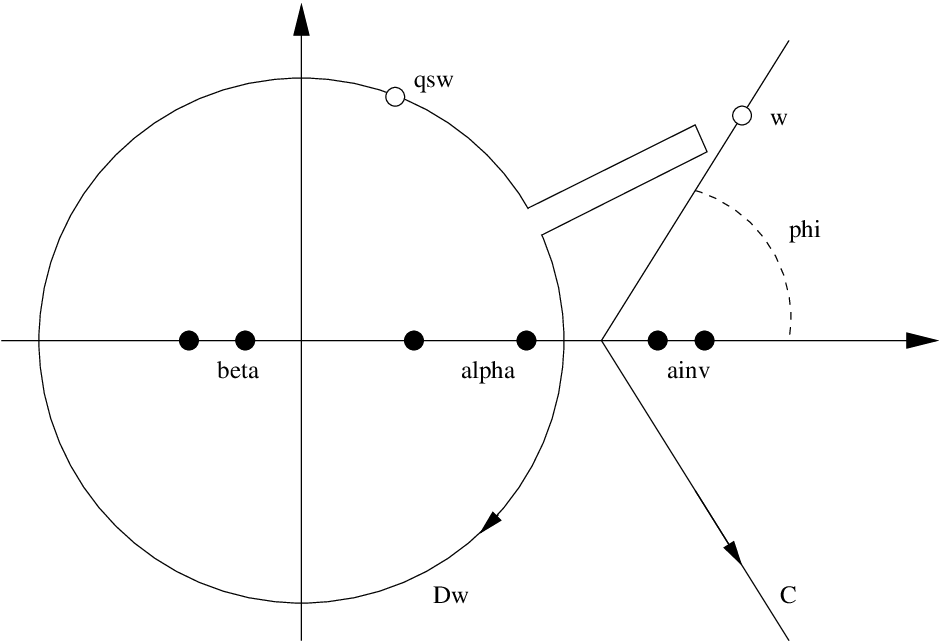}
\end{center}
\caption{The contour $\CwPre{\tilde a;\tilde \alpha,\tilde \beta; \varphi}$ (from Definition~\ref{CwPredef}) is depicted along with the contour corresponding to $q^s w$ for $w\in \CwPre{\tilde a;\tilde \alpha,\tilde \beta; \varphi}$ and $s\in \CsPre{w}$.}
\label{FigqWhitContours}
\end{figure}

We are prepared to state the central result of this section. The (most likely technical) condition that $N\geq 9$ (which comes from some convergence estimates used in the proof of Proposition~\ref{qFredDetThmbeta}) is not much of a limitation since we will ultimately be concerned in studying the large $N$ limit of this formula.

\begin{theorem}\label{qFredDetThm}
Fix $N\geq 9$ and $\tilde a,\tilde\alpha,\tilde\beta,\tilde\gamma$ as in Definition~\ref{defParams}. Then for all $\zeta\in \C\setminus \Rplus$
\begin{equation}\label{thmlaplaceeqn}
\EE_{\M_{\tilde a,\rho(\tilde \alpha;\tilde \beta;\tilde \gamma)}}\Bigg[ \frac{1}{\big(\zeta q^{-\lambda^{(N)}_1};q\big)_{\infty}}\Bigg] = \det(\Id+\tilde K_{\zeta})_{L^2(\CwPre{\tilde a; \tilde \alpha,\tilde\beta;\varphi})}
\end{equation}
where $\CwPre{\tilde a; \tilde \alpha,\tilde\beta;\varphi}$ as in Definition~\ref{CwPredef} with any $\varphi\in (0,\pi/2]$. The operator $\tilde K_{\zeta}$ is defined in terms of its integral kernel
\begin{equation}\label{eqnkzetakernel}
\tilde K_{\zeta}(w,w') = \frac{1}{2\pi \I}\int_{\CsPre{w}} \Gamma(-s)\Gamma(1+s)(-\zeta)^s g_{w,w'}(q^s)\,\d s
\end{equation}
where
\begin{equation}\label{gwwprimeeqn}\begin{aligned}
g_{w,w'}(q^s) &=\frac{1}{q^s w - w'}\, \frac{\Pi(w;a)}{\Pi(q^s w;a)}\, \frac{\Pi\big((q^s w)^{-1};\rho(\tilde\alpha;\tilde\beta;\tilde\gamma)\big)}{\Pi\big((w)^{-1};\rho(\tilde\alpha;\tilde\beta;\tilde\gamma)\big)}\\
&=\frac{\exp\big(\gamma w^{-1}(q^{-s}-1)\big)}{q^s w - w'} \,\prod_{i=1}^{N} \frac{(q^{s}w a_i;q)_{\infty}}{(w a_i;q)_{\infty}}\, \prod_{i=1}^{M_{\alpha}} \frac{\big((w)^{-1}\alpha_i;q\big)_{\infty}}{\big((q^{s}w)^{-1}\alpha_i ;q\big)_{\infty}}\,\prod_{i=1}^{M_{\beta}} \frac{1+ (q^s w)^{-1} \beta_i}{1+ (w)^{-1} \beta_i},
\end{aligned}\end{equation}
the contour $\CsPre{w}$ is as in Definition~\ref{CwPredef} and the function $\Pi$ is defined as in \eqref{specdefeqn}.
\end{theorem}
\begin{remark}
This formula bares many similarities to that for the $e_q$-Laplace transform for $q^{\lambda^{N}_N}$ in \cite[Theorem 3.2.11]{BC11} and \cite[Theorem 4.13]{BCF12}. The $\lambda^{N}_N$ are closely related to the particle system $q$-TASEP \cite[Section 3.3.2]{BC11}, and hence these formulas served as the starting point for large time asymptotic analysis of $q$-TASEP~\cite{FV13,Bar14}. It was shown in \cite{BP13b} that $\lambda^{N}_1$ relates to the particle system $q$-PushTASEP. The above theorem may (in a similar manner as in \cite{FV13,Bar14}) be of use in asymptotic analysis of this system.
\end{remark}

The remainder of this section is devoted to the proof of this theorem. The starting point for this proof is the approach described in~\cite[Section~3.2]{BC11} to compute the $e_{q}$-Laplace transform of $q^{\lambda^{(N)}_N}$. Rather quickly, though, we encounter challenges not previously present requiring new ideas. The approach from~\cite[Section~3.2]{BC11} for the random variable $q^{\lambda^{(N)}_{N}}$ starts by utilizing Macdonald difference operators to compute nested contour integral formulas for the moments $\EE\Big[q^{k \lambda^{(N)}_{N}}\Big]$ (the subscript $\M_{\tilde a,\rho(\tilde \alpha;\tilde \beta;\tilde \gamma)}$ is suppressed here). Since the random variable $q^{\lambda^{(N)}_{N}}\in (0,1]$, its moments determine its distribution, and its $e_q$-Laplace transform can be computed via a suitable moment generating series. Plugging the explicit formulas for these moments into the $e_q$-Laplace transform moment generating series results (after further manipulations) in a Fredholm determinant.

For $q^{-\lambda^{(N)}_{1}}$, this approach (of~\cite[Section~3.2]{BC11}) runs into a major issue in the first step. The random variable $\lambda^{(N)}_1$ is part of a partition and hence a non-negative integer. Since $q\in (0,1)$, it follows that $q^{-\lambda^{(N)}_{1}}\geq 1$. Moreover, if any of the specialization parameters $\tilde\alpha$ are non-zero (i.e.\ $\tilde\alpha_i>0$ for some $i$) then $q^{-\lambda^{(N)}_{1}}$ will only have a finite number of finite moments (Lemma~\ref{lemMomentsexist}). Therefore, recovering the distribution or $e_{q}$-Laplace transform from these finitely many moments is impossible. But we need the case where some of these $\tilde{\alpha}$ parameters are strictly positive as it relates after various limit transition to the stationary KPZ equation. Therefore, we must overcome this apparent obstacle.

In this case (where some $\alpha_i>0$), for $k$ small enough $\EE\Big[q^{-k\lambda^{(N)}_{1}}\Big]<\infty$ and there still exist nested contour integral moment formulas (coming from Macdonald difference operators). These formulas involve $k$ in a straightforward manner and one can try to extend the formula for $k$ to be an arbitrary integer. For those $k$ for which the moments are infinite, there fail to exist suitable choices of contours upon which to integrate. If one neglects this (important) impediment, it is possible to mimic the approach of~\cite[Section~3.2]{BC11}. The outcome of this formal calculation is the statement of Theorem~\ref{qFredDetThm} (with some additional guess to work out which contours to use in the final answer). Of course, this is not a mathematically justified calculation since it involves summing infinitely many terms which are themselves infinite and ill-defined. The outcome, however, is an equality between two finite quantities.

The challenge is to turn this into a meaningful rigorous result, and hence prove Theorem~\ref{qFredDetThm}. This is done in two steps:

\smallskip
\noindent {\bf Step 1:} Apply the approach of~\cite[Section~3.2]{BC11} to prove a Fredholm determinant formula for the $e_q$-Laplace transform of $q^{-\lambda^{(N)}_1}$ in the special case where all $\tilde \alpha_i\equiv 0$ (we also take $\tilde\gamma=0$ for this step). By studying the $q$-Whittaker measure under the pure $\tilde\beta$ specialization, we can prove a priori that $\lambda^{(N)}_1 \leq M_{\beta}$ (recall, $M_{\beta}$ is the number of non-zero entries in $\tilde\beta$). This bound shows that $\EE\Big[q^{-k\lambda^{(N)}_{1}}\Big]\leq q^{-kM_{\beta}}$, and hence we may adapt the approach from~\cite[Section~3.2]{BC11} to prove the pure $\tilde\beta$ specialization Fredholm determinant.

\smallskip
\noindent {\bf Step 2:} Interpret the pure $\tilde\beta$ specialization Fredholm determinant formula for the $e_q$-Laplace transform as a formal series identity in the Newton power-sum symmetric functions. Then, apply the $\rho(\tilde\alpha;\tilde\beta;\tilde\gamma)$ specialization to this formal series identity and observe that both sides of the identity form convergent series, hence proving the desired numerical identity which is Theorem~\ref{qFredDetThm}.

\smallskip

The key fact which lets us succeed here is that we are working with symmetric functions. This situation should be compared to the non-rigorous physics replica method. There, the problem is to compute the distribution (via the Laplace transform) of the solution to the stochastic heat equation $\mathcal{Z}(T,X)$. It is possible to compute similar sorts of moment formulas as those above for $\EE\big[\mathcal{Z}(T,X)^k\big]$. These moments remain finite for all $k$, but grow like $e^{c k^3}$ for some constant $c>0$. This means that the moment problem is ill-posed and these moments do not determine the distribution of $\mathcal{Z}(T,X)$. However, in the replica method calculations (e.g.~\cite{Dot10,CDR10}) one can still try to compute $\EE\big[e^{\zeta \mathcal{Z}(T,X)}\big]$ through these moments via expanding the exponential and interchanging the expectation and infinite summation. Though the Laplace transform is necessarily finite, the associated moment generating series (coming from the mathematically
unjustifiable interchange of expectation and the summation in the Taylor expansion of the exponential) is divergent. After some additional manipulations, this divergent generating function is `summed' to a finite expression -- a Fredholm determinant. At least in the case of $\mathcal{Z}_0(X)=\delta_{X=0}$, the resulting formulas agree with those proved in~\cite{ACQ10}.

In light of these similarities, one might hope to find a way to implement a variant of the rigorous justification we provide in the study of $q^{-\lambda^{(N)}_{1}}$ into the setting of the SHE. However, this seems unlikely. The justification we provide draws heavily upon the relationship between our observable $q^{-\lambda^{(N)}_{1}}$ and the $q$-Whittaker processes / symmetric functions. Such structures do not clearly survive the limit transitions which eventually relate to the SHE (see however~\cite{CH13,OCW11} for some trace of these structures). Furthermore, the pure $\tilde\beta$ specialization which is used to justify the formal identity we prove, does not (as of yet) have any analog (or limit) in the SHE (or even semi-discrete directed polymer) setting.

\subsection{Step 1: Pure $\tilde \beta$ Fredholm determinant formula}
For this step, we will focus on the $q$-Whittaker proceses with specializations $\rho=\rho(0;\tilde\beta;0)$ and $\tilde\beta= (\tilde\beta_1,\ldots\tilde\beta_{M})$ with $M\geq 1$ arbitrary (note that in Step 1b and 1c we return to considering general specializations to provide moment formulas). For these specialization, the $q$-Whittaker function $Q_{\lambda}(\rho)=0$ unless $\lambda_1\leq M$ (see~\cite[Section~2.2.1]{BC11}). This provides the a priori bound that under the $q$-Whittaker process, $q^{-\lambda^{(N)}_1} \leq q^{-M}$. Due to this bound, we will assume in Steps 1a--1e that
\begin{equation}\label{eqnzetabdd}
|\zeta|< (1-q) q^M,
\end{equation}
though in Steps 1d--1e we will impose an additional condition on $|\zeta|$.

\subsubsection{Step 1a: Relating $e_q$-Laplace transform to moments generating series}
Observe that for $\zeta$ satisfying \eqref{eqnzetabdd}, the function
$$
\zeta \mapsto \frac{1}{\big(\zeta q^{-\lambda^{(N)}_1};q)_{\infty}}
$$
(which can be rewritten as $e_q\big(\zeta (1-q)^{-1} q^{-\la^{(N)}_1}\big)$, cf.\ Appendix~\ref{qSec})
is always finite and analytic in $\zeta$. Using the $q$-Binomial theorem (cf.\ Appendix~\ref{qSec}) we may expand
\begin{equation}\label{eqexpandmgs}
\frac{1}{\big(\zeta q^{-\lambda^{(N)}_1};q)_{\infty}} = \sum_{k=0}^{\infty} \frac{\big(\zeta/(1-q)\big)^k}{k_q!} q^{-k\lambda^{(N)}_1}
\end{equation}
where the $q$-deformed factorial is defined as
$$
k_q! = \frac{(q;q)_k}{(1-q)^k}.
$$
Due to \eqref{eqnzetabdd}, it follows that each summand on the right-hand side of \eqref{eqexpandmgs} can be bounded deterministically by a corresponding summand of a convergent geometric series. This justifies the interchange of expectation and summation necessary to establish the equality
\begin{equation}\label{eqninterchanged}
\EE_{\M_{\tilde a,\rho(0;\tilde \beta;0)}}\Bigg[\frac{1}{\big(\zeta q^{-\lambda^{(N)}_1};q)_{\infty}}\Bigg] = \sum_{k=0}^{\infty} \frac{\big(\zeta/(1-q)\big)^k}{k_q!} \EE_{\M_{\tilde a,\rho(0;\tilde \beta;0)}}\Big[q^{-k\lambda^{(N)}_1}\Big]
\end{equation}
for those $\zeta$ satisfying \eqref{eqnzetabdd}.

\subsubsection{Step 1b: Nested contour integral formulas for moments}

In Step 1b--1c, we temporarily return to the general $\rho(\tilde\alpha;\tilde\beta;\tilde\gamma)$ specialization and prove nested contour integral formulas for moments (when they exist).
Let us first describe conditions on $\tilde a,\tilde\alpha,\tilde\beta$ under which moments of $q^{-\lambda^{(N)}_1}$ are finite, or infinite.

\begin{lemma}\label{lemMomentsexist}
Fix $N\geq 1$ and $\tilde a,\tilde\alpha,\tilde\beta,\tilde\gamma$ as in Definition~\ref{defParams}. For $k\geq 1$, if $\max_{i,j} \tilde a_i \tilde \alpha_j <q^{k}$, then
$$
\EE_{\M_{\tilde a,\rho(\tilde\alpha;\tilde \beta;\tilde\gamma)}}\Big[q^{-k\lambda^{(N)}_1}\Big] <+\infty.
$$
On the other hand, if $\max_{i,j} \tilde a_i \tilde \alpha_j >q^{k}$, then
$$
\EE_{\M_{\tilde a,\rho(\tilde\alpha;\tilde \beta;\tilde\gamma)}}\Big[q^{-k\lambda^{(N)}_1}\Big] = +\infty.
$$
\end{lemma}
\begin{proof}
We first address the case that $\max_{i,j} \tilde a_i \tilde \alpha_j <q^{k}$. We can bound
\begin{align*}
\EE_{\M_{\tilde a,\rho(\tilde\alpha;\tilde \beta;\tilde\gamma)}}\big[q^{-k\lambda^{(N)}_1}\big] &= \sum_{\lambda^{(N)}} q^{-k \lambda^{(N)}_{1}} \M_{\tilde a,\rho(\tilde\alpha;\tilde \beta;\tilde\gamma)}(\lambda^{(N)})\\
&\leq \sum_{\lambda^{(N)}} q^{-k |\lambda^{(N)}|} \M_{\tilde a,\rho(\tilde\alpha;\tilde \beta;\tilde\gamma)}(\lambda^{(N)})\\
&= \sum_{\lambda^{(N)}} \M_{q^{-k}\tilde a,\rho(\tilde\alpha;\tilde \beta;\tilde\gamma)}(\lambda^{(N)})<\infty.
\end{align*}
The equality on the first line is by definition; the inequality on the second line is from the fact that $\lambda^{(N)}_1 \leq |\lambda^{(N)}|$, where $|\lambda| = \sum \lambda_i$; the equality on the third line comes from the fact that $c P_{\lambda}(\rho)= P_{\lambda}(c\rho)$; the final inequality on the third line comes from the fact that the $q$-Whittaker process $\M_{q^{-k}\tilde a,\rho(\tilde\alpha;\tilde \beta;\tilde\gamma)}$ is well-defined as long as $q^{-k} \tilde a_i \tilde \alpha_j< 1$.

Turning now to the case that $\max_{i,j} \tilde a_i \tilde \alpha_j >q^{k}$, assume (without loss of generality) that $\tilde a_1 = \max (\tilde a)$ and $\tilde\alpha_1 = \max(\alpha)$. By the interlacing inequalities, $\lambda^{(1)}_1 \leq \lambda^{(N)}_1$. This means that if we show $\EE_{\M_{\tilde a,\rho(\tilde\alpha;\tilde \beta;\tilde\gamma)}}\Big[q^{-k\lambda^{(1)}_1}\Big]=+\infty$, then so too must $\EE_{\M_{\tilde a,\rho(\tilde\alpha;\tilde \beta;\tilde\gamma)}}\Big[q^{-k\lambda^{(N)}_1}\Big]=+\infty$. To further simplify considerations, observe that the $q$-Whittaker process $\M_{\tilde a_1,\rho(\tilde\alpha_1;0;0)}$ is stochastically dominated by the $q$-Whittaker process $\M_{\tilde a_1,\rho(\tilde\alpha;\tilde \beta;\tilde\gamma)}$. This can be seen, for instance, in light of dynamics~\cite[Section~2.3]{BC11} which maps the first process to the second process by only increasing coordinates. Owing to this stochastic domination, it suffices to prove that $\EE_{\M_{\tilde a_1,\rho(\tilde\alpha_1;0;
0)
}}\Big[q^{-k\lambda^{(1)}_1}\Big]=+\infty$. This expectation, however, is computable quite explicitly since $P_{\lambda^{(1)}_1}(\tilde a_1) = (\tilde a_1)^{\lambda^{(1)}_1}$, $Q_{\lambda^{(1)}_1}(\tilde \alpha_1) = (\tilde\alpha_1)^{\lambda^{(1)}_1} (q;q)_{\lambda^{(1)}_1}^{-1}$ and $\Pi(\tilde a_1;\tilde \alpha_1) = (\tilde a_1\tilde\alpha_1;q)_{\infty}^{-1}$. Therefore, we find that
\begin{align*}
\EE_{\M_{\tilde a_1,\rho(\tilde\alpha_1;0;0)}}\Big[q^{-k\lambda^{(1)}_1}\Big] &= \sum_{\lambda^{(1)}_1\geq 0} q^{-k \lambda^{(1)}_{1}} \M_{\tilde a_1,\rho(\tilde\alpha_1;0;0)}(\lambda^{(1)}_{1})\\
&= \sum_{\lambda^{(1)}_1\geq 0} q^{-k \lambda^{(1)}_{1}} (\tilde a_1\tilde\alpha_1)^{\lambda^{(1)}_1} \frac{(\tilde a_1\tilde\alpha_1;q)_{\infty}}{(q;q)_{\lambda^{(1)}_1}},
\end{align*}
which is seen to diverge to $+\infty$ under the condition that $q^{-k} \tilde a_1\tilde\alpha_1 >1$.
\end{proof}

The following proposition provides explicit formulas for those moments which are necessarily finite due to Lemma~\ref{lemMomentsexist}. It should be observed that suitable contours in the statement of the proposition exist under the same conditions as the finiteness of moments.

\begin{proposition}\label{momentprop}
Fix $N\geq 1$ and $\tilde a,\tilde\alpha,\tilde\beta,\tilde\gamma$ as in Definition~\ref{defParams}. For $k\geq 1$, if \mbox{$\max_{i,j} \tilde a_i \tilde \alpha_j <q^{\eta k}$} for some $\eta>1$, then
\begin{equation}\label{eqmoments}
\EE_{\M_{\tilde a,\rho(\tilde \alpha;\tilde \beta;\tilde \gamma)}}\Big[ q^{-k \lambda^{(N)}_1}\Big] = \frac{(-1)^k q^{\frac{k(k-1)}{2}}}{(2\pi \I)^k} \int_{C_1} \cdots \int_{C_k} \prod_{1\leq A<B\leq k}\frac{z_A-z_B}{z_A-q z_B} \prod_{i=1}^{k} \frac{g(z_i)}{g(qz_i)} \frac{\d z_i}{z_i}
\end{equation}
where
\begin{equation}\label{eqg}
g(z) =\prod_{i=1}^{N} \frac{1}{(\tilde a_i z;q)_{\infty}} \, \frac{1}{\Pi\big((z)^{-1};\rho(\tilde \alpha;\tilde \beta;\tilde \gamma)\big)} = \prod_{i=1}^{N} \frac{1}{(\tilde a_i z;q)_{\infty}} \,
e^{\tilde \gamma z^{-1}} \prod_{i=1}^{M_{\alpha}} (\tilde \alpha_i z^{-1};q)_{\infty}\, \prod_{i=1}^{M_{\beta}} \frac{1}{1+\tilde \beta_i z^{-1}}.
\end{equation}
The contours $C_1,\ldots,C_k$ are defined by $C_i= q^{\eta(k-i)} C_k$ where $C_k=\{ c+e^{- \I \varphi \sgn(y)}, y\in \R\}$ (oriented so as to have decreasing imaginary part) with any $\varphi\in (0,\pi/2]$ and with any $c\in \R$ satisfying $q^{-\eta k} \max (\alpha) < c<\min (a^{-1})$.
\end{proposition}
\begin{proof}
We provide a brief proof, as this result has essentially appeared before in~\cite[Remark~2.2.11]{BC11} and~\cite[Theorem~4.6]{BCGS13} (in the more general Macdonald processes language).
The proof is based on a simple observation. Assume we have a linear operator $\mathcal{D}$ on the space of functions in $N$ variables whose restriction to the space of symmetric polynomials acts diagonally in the basis of $q$-Whittaker polynomials: $\mathcal{D} P_\lambda=d_\lambda P_\lambda$ for any partition $\lambda$ with length $\ell(\lambda)\le N$. Then we can apply $\mathcal{D}$ to both sides of the identity (acting in the $\tilde a$ variables)
\begin{equation*}
\sum_{\lambda} P_{\lambda}(\tilde a) Q_{\lambda}(\rho)=\Pi(\tilde a;\rho).
\end{equation*}
Dividing the result by $\Pi(\tilde a;\rho)$, we obtain
\begin{equation*}
\EE_{\M_{\tilde a,\rho}}\left[ d_\lambda \right]=\frac{\mathcal{D}\Pi(\tilde a;\rho)} {\Pi(\tilde a;\rho)}\,.
\end{equation*}
This equality is valid so long as the expectation on the left-hand side is finite. If we apply $\mathcal{D}$ several times, we obtain
\begin{equation*}
\EE_{\M_{\tilde a,\rho}}\left[ (d_\lambda)^k \right]=\frac{\mathcal{D}^k\Pi(\tilde a;\rho)} {\Pi(\tilde a;\rho)}\,.
\end{equation*}
If we have several possibilities for $\mathcal{D}$ we can obtain formulas for averages of the observables equal to products of powers of the corresponding eigenvalues. One of the remarkable features of Macdonald polynomials is that there exists a large family of such operators. These are the Macdonald difference operators.

We will consider a slight variant of the $t=0$ $(N-1)$-st difference operator. For any $u\in \R$ and $1\leq i\leq N$ define the shift operator $T_{u,x_i}$ via its action
\begin{equation*}
\left(T_{u,x_i}F\right)(x_1,\ldots, x_N) = F(x_1,\ldots, u x_i,\ldots, x_{N}).
\end{equation*}
The operator $\tilde{D}$ which we utilize is given by
\begin{equation*}
\tilde{D} = \sum_{j=1}^{N} \prod_{\substack{i=1\\i\neq j}}^{N} \frac{x_j}{x_j-x_i} T_{q^{-1},x_j}.
\end{equation*}
The $q$-Whittaker polynomials diagonalize this operator~\cite[Remark~2.2.11]{BC11}, so that for all $\lambda$ of length $\ell(\lambda)\leq N$,
\begin{equation*}
\tilde{D} P_{\lambda}(x_1,\ldots, x_N) = q^{-\lambda_1} P_{\lambda}(x_1,\ldots, x_N).
\end{equation*}

Thus, using the procedure described above, we find that
\begin{equation}\label{eqmomentsabove}
\EE_{\M_{\tilde a,\rho}}\Big[q^{-k\lambda^{(N)}_1} \Big]=\frac{\tilde{D}^k\Pi(\tilde a;\rho)} {\Pi(\tilde a;\rho)}\,.
\end{equation}
This equality is true assuming that the left-hand side is finite. Lemma~\ref{lemMomentsexist} provides the conditions for $\rho=\rho(\tilde\alpha;\tilde\beta;\tilde\gamma)$ such that this expectation is finite.

To complete the proof we must identify the right-hand side of \eqref{eqmoments} with the right-hand side of \eqref{eqmomentsabove}. This identification follows from residue calculus. The contour $C_k$ (along which $z_k$ is integrated) can be deformed to cross the set $\tilde a^{-1}$. This deformation crosses poles and the integral is thus expanded as a sum of residues at these poles and a remaining integral over a new contour which lies to the right of the $\tilde a^{-1}$. The remaining integral evaluates to zero, as can be seen by using Cauchy's theorem and the at least quadratic decay of the integrand (as $z_k$ goes to infinity in the right half of $\C$). Each of the residue terms involves $k-1$ integrals, and this procedure can be repeated for the $z_{k-1}$ through $z_1$ integrals. The resulting residue expansion of the integrals match exactly the formula on the right-hand side of~\eqref{eqmomentsabove}. See~\cite[Section~2.2.3]{BC11} for more details of this type of residue bookkeeping.
\end{proof}

\subsubsection{Step 1c: Unnesting the integrals}

Proposition~\ref{momentprop} provides a nested contour integral formula for the moments of $q^{-\lambda^{(N)}_1}$ under the general $\rho(\tilde\alpha;\tilde\beta;\tilde\gamma)$ specialization. Here we record the effect of deforming all of the contours to lie upon the same fixed contour.

\begin{proposition}\label{propunnested}
Fix $N\geq 1$ and $\tilde a,\tilde\alpha,\tilde\beta,\tilde\gamma$ as in Definition~\ref{defParams}. For $k\geq 1$, if \mbox{$\max_{i,j} \tilde a_i \tilde \alpha_j <q^{\eta k}$} for some $\eta>1$, then
\begin{equation*}
\EE_{\M_{\tilde a,\rho(\tilde \alpha;\tilde \beta;\tilde \gamma)}}\big[ q^{-k \lambda^{(N)}_1}\big] = \sum_{\mu\vdash k} \frac{1}{m_1! m_2! \cdots} \frac{(1-q)^k}{(2\pi \I)^{\ell(\mu)}} \int_{C}\cdots \int_{C} \det\left[\frac{1}{w_i q^{\mu_i} - w_j}\right]_{i,j=1}^{\ell(\mu)} \prod_{j=1}^{\ell(\mu)} \frac{g(w_j)}{g(q^{\mu_j}w_j)} \d w_j
\end{equation*}
where $\mu$ is a partition of $k$ (hence the notation $\mu\vdash k$) of length $\ell(\mu)$ and multiplicities \mbox{$m_i = |\{j:\mu_j=i\}|$}, the function $g$ is given by \eqref{eqg}, and the contour \mbox{$C=\{ c+e^{- \I \varphi \sgn(y)}, y\in \R\}$} (oriented so as to have decreasing imaginary part) with any $\varphi\in (0,\pi/2]$ and with any $c\in \R$ satisfying $\max (\alpha) < c<\min (a^{-1})$.
\end{proposition}

\begin{proof}
This is essentially proved as~\cite[Proposition~3.2.1]{BC11}, or~\cite[Proposition~7.4]{BCPS13}. The contours $C_{k-1}$ through $C_{1}$ (on the right-hand side of \eqref{eqmomentsabove}) are sequentially deformed to lie along $C_k$. This deformation involves crossing certain strings of poles (recorded by the partition $\mu$). The resulting formula comes from bookkeeping these residues. Note that once all integration contours coincide with $C_k$, these contours can be simultaneously deformed (without crossing any poles or changing the value of the integrals) to any choice of contour $C$ as specified in the statement of the proposition.
\end{proof}

\subsubsection{Step 1d: Summing the moment generating series}

We return now to studying the case of the pure $\tilde\beta$ specialization. Equation \eqref{eqninterchanged} of Step 1a shows that the $e_q$-Laplace transform of $q^{-\lambda^{(N)}_1}$ is equal to a generating series of the moments $\EE\Big[q^{-k\lambda^{(N)}_1}\Big]$ provided $|\zeta|< (1-q) q^M$. Proposition~\ref{propunnested} of Step 1c (in the pure $\tilde\beta$ specialization) provides explicit formulas for these moments. The following proposition shows how plugging these formulas into the moment generating series results in a Fredholm determinant formula. In order for this sum to converge, we must impose some further restrictions on $|\zeta|$. We also assume $N\geq 2$ here as it is helpful for the technical aspects of the argument of the proof to proceed.

\begin{proposition}\label{propFirstFred}
Fix $N\geq 2$, $\tilde a, \tilde\beta$ as in Definition~\ref{defParams}, $\varphi\in (0,\pi/2]$, and a contour $\CwPre{\tilde a;0,\tilde\beta;\varphi}$ as in Definition~\ref{CwPredef}. Specialize $g$ from \eqref{eqg} to the pure $\tilde\beta$ specialization as
\begin{equation*}
g(w) =\prod_{i=1}^{N} \frac{1}{(\tilde a_i w;q)_{\infty}} \, \frac{1}{\Pi\big((w)^{-1};\rho(0;\tilde \beta;0)\big)} = \prod_{i=1}^{N} \frac{1}{(\tilde a_i w;q)_{\infty}} \,
\prod_{i=1}^{M} \frac{1}{1+\tilde \beta_i w^{-1}}.
\end{equation*}
Define $f(w) = \frac{g(w)}{g(qw)}$ and the constant
\begin{equation}\label{eqncconst}
C_1= \sup_{\ell\geq 1, w\in \CwPre{\tilde a; 0,\tilde\beta;\varphi}} |f(q^\ell w)|.
\end{equation}
Then for all $\zeta\in \C\setminus \Rplus$, such that $|\zeta| < \min\big\{ (1-q)q^M, C_1^{-1}\big\}$,
\begin{equation*}
\EE_{\M_{\tilde a,\rho(0;\tilde \beta;0)}}\Bigg[\frac{1}{\big(\zeta q^{-\lambda^{(N)}_1};q\big)_{\infty}}\Bigg] = \det(\Id+K_{\zeta})_{L^2(\Z_{>0}\times \CwPre{\tilde a; 0,\tilde\beta;\varphi})}.
\end{equation*}
The operator $K_{\zeta}$ is defined in terms of its kernel
\begin{equation*}
K_{\zeta}(n_1,w_1;n_2,w_2) = \frac{\zeta^{n_1}}{q^{n_1} w_1-w_2}\, \frac{g(w_1)}{g(q^{n_1}w_1)}.
\end{equation*}
\end{proposition}
\begin{proof}
In light of \eqref{eqninterchanged} and the bound $|\zeta|<(1-q) q^M$, it suffices to prove that
\begin{equation*}
\sum_{k=0}^{\infty} \frac{\big(\zeta/(1-q)\big)^k}{k_q!} \EE_{\M_{\tilde a,\rho(0;\tilde \beta;0)}}\Big[q^{-k\lambda^{(N)}_1}\Big] = \det(\Id+K_{\zeta})_{L^2(\Z_{>0}\times \CwPre{\tilde a; 0,\tilde\beta;\varphi})}.
\end{equation*}
Using Proposition~\ref{propunnested} for the pure $\tilde\beta$ specialization, we can rewrite (see~\cite[Proposition~3.2.8]{BC11} or~\cite[Proposition~4.12]{BCF12}) the $k$-th term in this moment generating series as
\begin{multline*}
\frac{\big(\zeta/(1-q)\big)^k}{k_q!} \EE_{\M_{\tilde a,\rho(0;\tilde \beta;0)}}\Big[q^{-k\lambda^{(N)}_1}\Big] \\
=\sum_{L\geq 0} \frac{1}{L!} \int_{\CwPre{\tilde a; 0,\tilde\beta;\varphi}}\cdots \int_{\CwPre{\tilde a; 0,\tilde\beta;\varphi}} \sum_{\stackrel{n_1,\ldots,n_L\geq 1}{\sum n_i=k}}\det\left[\frac{1}{w_i q^{n_i}-w_j}\right]_{i,j=1}^{L} \prod_{j=1}^{L} \zeta^{n_j} \frac{g(w_j)}{g(q^{n_j}w_j)} \frac{\d w_j}{2\pi \I}.
\end{multline*}
Summing over $k$ yields
\begin{multline}\label{eqfredexpC}
\sum_{k=0}^{\infty} \frac{\big(\zeta/(1-q)\big)^k}{k_q!} \EE_{\M_{\tilde a,\rho(0;\tilde \beta;0)}}\Big[q^{-k\lambda^{(N)}_1}\Big] \\
= \sum_{L\geq 0} \frac{1}{L!} \int_{\CwPre{\tilde a; 0,\tilde\beta;\varphi}}\cdots \int_{\CwPre{\tilde a; 0,\tilde\beta;\varphi}} \sum_{n_1,\ldots, n_L\geq 1} \det\left[\frac{1}{w_i q^{n_i}-w_j}\right]_{i,j=1}^{L} \prod_{j=1}^{L} \zeta^{n_j} \frac{g(w_j)}{g(q^{n_j}w_j)} \frac{\d w_j}{2\pi \I}
\end{multline}
which is the Fredholm determinant expansion of $\det(\Id+K_{\zeta})_{L^2(\Z_{>0}\times \CwPre{\tilde a; 0,\tilde\beta;\varphi})}$.

The convergence of this Fredholm determinant expansion, as well as the manipulations used in reaching it require some justifications. (After all, the manipulations involved rearranging an infinite summation.) Note that by assumptions on the contour $\CwPre{\tilde a; 0,\tilde\beta;\varphi}$, the function $q^{n_i} w_i/w_j-1$ remains bounded from 0 uniformly as $w_i,w_j\in \CwPre{\tilde a; 0,\tilde\beta;\varphi}$ and $n_i\geq 1$ vary. It follows then from Hadamard's inequality that there exists a constant $B_1>0$ such that for all $w_i,w_j\in \CwPre{\tilde a; 0,\tilde\beta;\varphi}$ and $L,n_1,\ldots, n_L\geq 1$
$$
\left| \det\left[\frac{1}{w_i q^{n_i}-w_j}\right]_{i,j=1}^{L} \right| \leq B_1^L L^{L/2}.
$$
We may also show that for all $w_j\in \CwPre{\tilde a; 0,\tilde\beta;\varphi}$ and $n_j\geq 1$,
$$
\left|\frac{g(w_j)}{g(q^{n_j}w_j)}\right|\leq C_1^{n_j} C_2 w_j^{-N}
$$
where $C_1$ is defined in \eqref{eqncconst} and
$$
C_2 = C_1^{-1} \sup_{w\in \CwPre{\tilde a; 0,\tilde\beta;\varphi}} f(w) w^N.
$$
This is shown by writing (recall $f(w)=\tfrac{g(w)}{g(qw)}$)
$$
\frac{g(w_j)}{g(q^{n_j}w_j)} = f(w_j) f(qw_j) \cdots f(q^{n_j-1} w_j)
$$
and using the definition of $C_1$ and $C_2$. The finiteness of these constants is easily verified.

Combining these observations, we may bound in absolute value the $L$-th summand in \eqref{eqfredexpC} by
$$
\frac{1}{L!} B_1^L L^{L/2} \left(\sum_{n\geq 1} C_2 (C_1\zeta)^n \int_{\CwPre{\tilde a; 0,\tilde\beta;\varphi}} \frac{|\d w|}{2\pi} |w|^{-N}\right)^L.
$$
As we have assumed $N\geq 2$, the integral in $|\d w|$ is bounded by a constant $B_2>0$. Since, by hypothesis, $|\zeta|<C_1^{-1}$ we can also bound the summation in $n$ by another constant $B_3>0$. Therefore, the above expression is bounded by
$$
\frac{1}{L!} (C_2 B_1 B_2 B_3)^L L^{L/2}.
$$
Since the summation of this over $L\geq 0$ is finite, the Fredholm determinant expansion \eqref{eqfredexpC} is absolutely convergent. Using similar bounds as described above, we can also justify all of the interchanging of summands necessary to complete the proof of the proposition.
\end{proof}

\subsubsection{Step 1e: Rewriting as a Fredholm determinant of desired type}

We will now prove Theorem~\ref{qFredDetThm} in the case of pure $\tilde\beta$ specialization subject to the condition that $\zeta\in \C\setminus \Rplus$ satisfies $|\zeta| < \min\big\{ (1-q)q^M, C_1^{-1}\big\}$, with $C_1$ from \eqref{eqncconst}.

\begin{proposition}\label{qFredDetThmbetarestricted}
Fix $N\geq 2$ and $\tilde a,\tilde\beta$ as in Definition~\ref{defParams}. Then for all $\zeta\in \C\setminus \Rplus$ satisfying $|\zeta| < \min\big\{ (1-q)q^M, C_1^{-1}\big\}$, with $C_1$ from \eqref{eqncconst}, \begin{equation*}
\EE_{\M_{\tilde a,\rho(0;\tilde \beta;0)}}\Bigg[ \frac{1}{\big(\zeta q^{-\lambda^{(N)}_1};q\big)_{\infty}}\Bigg] = \det(\Id+\tilde K_{\zeta})_{L^2(\CwPre{\tilde a;0,\tilde\beta;\varphi})}
\end{equation*}
with $\CwPre{\tilde a; 0,\tilde\beta;\varphi}$ as in Definition~\ref{CwPredef} with any $\varphi\in (0,\pi/2]$. The operator $\tilde K_{\zeta}$ is defined in terms of its integral kernel given in \eqref{eqnkzetakernel} with $g_{w,w'}(q^{s})$ from \eqref{gwwprimeeqn} explicitly given in the pure $\tilde\beta$ specialization by
\begin{equation}\label{gwwprimeprime}
g_{w,w'}(q^s) = \frac{1}{q^s w-w'} \prod_{i=1}^{N} \frac{ (\tilde a_i q^s w;q)_{\infty}}{(\tilde a_i w;q)_{\infty}} \prod_{i=1}^{M} \frac{ 1+ \tilde\beta_i (q^s w)^{-1}}{ 1+ \tilde\beta_i (w)^{-1}}.
\end{equation}
\end{proposition}
\begin{proof}
The convergent Fredholm expansion of $\det(\Id+K_{\zeta})_{L^2(\Z_{>0}\times \CwPre{\tilde a; 0,\tilde\beta;\varphi})}$ can be written as
$$
\det(\Id+K_{\zeta})_{L^2(\Z_{>0}\times \CwPre{\tilde a; 0,\tilde\beta;\varphi})} = \sum_{L\geq 0} \frac{1}{L!} \int_{\CwPre{\tilde a; 0,\tilde\beta;\varphi}} \frac{\d w_1}{2\pi \I}\cdots \int_{\CwPre{\tilde a; 0,\tilde\beta;\varphi}}\frac{\d w_L}{2\pi \I} \det\left[ \sum_{n=1}^{\infty} \zeta^n g_{w_i,w_j}(q^n)\right]_{i,j=1}^{L}
$$
where $g_{w,w'}(q^s)$ is from the statement of the proposition.

We will have proved the proposition if we can show that
\begin{equation}\label{eqsummingmellin}
\sum_{n=1}^{\infty} \zeta^n g_{w,w'}(q^n) = \frac{1}{2\pi \I}\int_{\CsPre{w}} \Gamma(-s)\Gamma(1+s)(-\zeta)^s g_{w,w'}(q^s)\,\d s.
\end{equation}

To show this, we will apply the following {\it Mellin--Barnes} representation.

\begin{lemma}\label{gammasumlemma}
For all functions $g$, all negatively oriented (with respect to the points $1,2,\ldots$) contours $C_{1,2,\ldots}$ and all $\zeta\in \C\setminus \Rplus$ which satisfy the conditions below, we have the identity
\begin{equation*}
\sum_{n=1}^{\infty} g(q^n) \zeta^n = \frac{1}{2\pi \I} \int_{C_{1,2,\ldots}} \Gamma(-s)\Gamma(1+s)(-\zeta)^s g(q^s)\,\d s
\end{equation*}
where the function $\zeta\mapsto (-\zeta)^s$ on the right-hand side is defined with respect to a branch cut along $\zeta \in \R_{+}$. The conditions which must be satisfied are as follows: for $k$ large, there must exist positively oriented contours $C_k$ which enclose the points $1,2,\ldots, k$, which do not enclose any singularities of $g(q^s)$, and which are such that the integral above taken along the symmetric difference of $C_{1,2,\ldots}$ and $C_k$ goes to zero as $k$ goes to infinity.
\end{lemma}
\begin{proof}
The identity follows from $\Res{s=k}\,\Gamma(-s)\Gamma(1+s) = (-1)^{k+1}$.
\end{proof}

\begin{figure}
\begin{center}
\psfrag{0}[cb]{$0$}
\psfrag{k1}[rt]{$k$}
\psfrag{k2}[lt]{$k+1$}
\psfrag{Cs}[lb]{$C_{k}$}
\psfrag{Carc}[lb]{$C^{arc}_k$}
\psfrag{Cseg}[lb]{$C^{seg}_k$}
\psfrag{R}[cb]{$R$}
\psfrag{2d}[lb]{$2d$}
\psfrag{12}[cb]{$\frac12$}
\includegraphics[height=6cm]{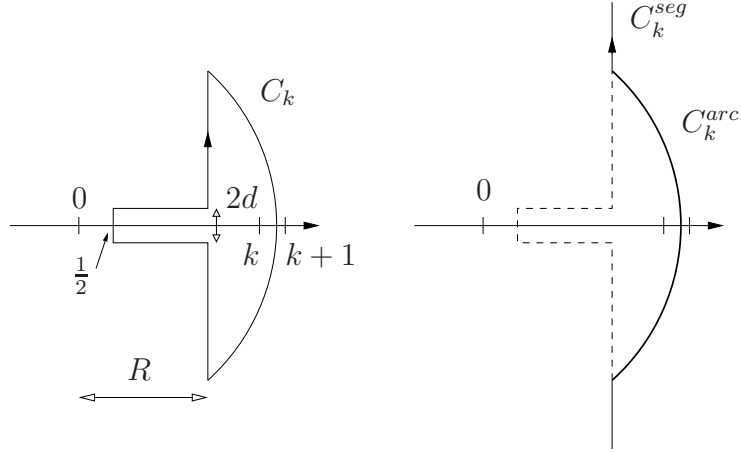}
\caption{Left: The contour $C_{k}$ composed of the union of two parts -- the first part is the portion of the contour $\CsPre{w}$ which lies within the ball of radius $k+1/2$ centered at the origin; the second part is the arc of that ball which causes the union to be a closed contour which encloses $\{1,2,\ldots, k\}$ and no other integers. Right: The symmetric difference between $C_k$ and $\CsPre{w}$ is given by two parts: a semi-circle arc which we call $C^{arc}_k$ and a portion of $R+\I \R$ with magnitude exceeding $k+1/2$ which we call $C^{seg}_k$.}\label{SymDiffContours}
\end{center}
\end{figure}

We apply Lemma~\ref{gammasumlemma} to prove \eqref{eqsummingmellin}. Let $C_{1,2,\ldots} = \CsPre{w}$ and let $C_{k}$ be composed of the union of two parts -- the first part is the portion of the contour $\CsPre{w}$ which lies within the ball of radius $k+1/2$ centered at the origin; the second part is the arc of the boundary of that ball which causes the union to be a closed contour which encloses $\{1,2,\ldots, k\}$ and no other integers. The contours $C_{k}$ are oriented positively and illustrated in the left-hand side of Figure~\ref{SymDiffContours}. We may assume $k$ is large enough so that the circle of radius $k+1/2$ intersects $\CsPre{w}$ on its vertical component. By the definition of the contours $\CwPre{\tilde a; 0,\tilde\beta;\varphi}$ and $\CsPre{w}$ we are assured that the contours $C_k$ do not contain any poles of $(-\zeta)^s g_{w,w'}(q^s)$.
This is due to the fact that the contours have been chosen such that as $s$ varies, $q^sw$ stays entirely to the left of $\CwPre{\tilde a; 0,\tilde\beta;\varphi}$ and hence does not touch $w'$.

In order to apply the above lemma we must estimate the integral along the symmetric difference of $C_{1,2,\ldots}$ and $C_k$. Identify the part of the symmetric difference given by the circular arc as $C^{arc}_k$ and the part given by the portion of $R+\I \R$ with magnitude exceeding $k+1/2$ as $C^{seg}_k$ (see the right-hand side of Figure~\ref{SymDiffContours}). We will estimate the integrand on each of these contours.

Concerning the term $(-\zeta)^s$, we may write $-\zeta = r e^{i\sigma}$ with $\sigma\in (-\pi,\pi)$ and $r>0$. Then we have $(-\zeta)^s =r^s e^{\I s\sigma}$. Writing $s=x+\I y$ we have $|(-\zeta)^s| = r^{x}e^{-y\sigma}$. Note that our assumptions on $\zeta$ imply $r<\min\big\{ (1-q)q^M, C_1^{-1}\big\}$, with $C_1$ from \eqref{eqncconst}, and $\sigma\in(-\pi,\pi)$.

Concerning the product of Gamma functions, one readily confirms that there exists $c>0$ such that for all $s$ with $\dist(s,\Z)\geq 1/2$
\begin{equation*}
\Big| \Gamma(-s)\Gamma(1+s) \Big| \leq \frac{c}{e^{\pi |\Im(s)|}}.
\end{equation*}

Focusing on $s\in C^{seg}_k$, the above bounds imply
$$
\Big|\Gamma(-s)\Gamma(1+s)(-\zeta)^s\Big| \leq c r^R e^{-y\sigma - \pi |y|}
$$
as $\dist(s,\Z)\geq 1/2$ and $x=R$ along $C^{seq}_k$ (recall $s=x+\I y$). As $s$ varies along $C^{seq}_k$, $q^s w$ cycles around a circle of fixed radius containing the origin and hence we may bound $|g_{w,w'}(q^s)|<C$ for some constant $C>0$ independent of $s\in C^{seq}_k$ and $k$. That implies
$$
\left|\frac{1}{2\pi \I} \int_{C^{seq}_{k}} \Gamma(-s)\Gamma(1+s)(-\zeta)^s g_{w,w'}(q^s)\,\d s \right| \leq \int_{C^{seq}_{k}} \frac{|\d s|}{2\pi} C c r^R e^{-y\sigma - \pi |y|}.
$$
Since $\sigma\in (-\pi,\pi)$, the integrand decays exponentially as $|y|$ increases (recall $s=x+\I y$). This means that as $k$ goes to infinity, the above integral converges to zero.

Focusing on $s\in C^{arc}_k$, the earlier bounds imply
$$
\left|\Gamma(-s)\Gamma(1+s)(-\zeta)^s\right| \leq c r^x e^{-y\sigma - \pi |y|}
$$
as $\dist(s,\Z)\geq 1/2$ along $C^{seq}_k$. By inspection, we may bound $|g_{w,w'}(q^s)|<C' q^{-x M}$ for some constant $C'>0$ independent of $s\in C^{arc}_k$ and $k$. To see this, observe that as \mbox{$s\in C^{arc}_k$} and $k$ varies, $(q^s w-w')^{-1}$ stays uniformly bounded, and since $|a_iq^s w|$ stays uniformly bounded strictly below 1, each term $(a_i q^s w;q)_{\infty}$ remains uniformly bounded by a constant (the denominator $(a_i w;q)_{\infty}$ remains constant as $s$ and $k$ vary). The term $1+ \frac{\beta_i}{q^s w}$ is bounded by a constant times $q^{-x}$ where $s= x+\I y$ (and the corresponding term in the denominator $1+\frac{\beta_i}{w}$ remains constant as $s$ and $k$ vary). Combining these considerations leads to the claimed bound.

This bound implies
$$
\left|\frac{1}{2\pi \I} \int_{C^{arc}_{k}} \Gamma(-s)\Gamma(1+s)(-\zeta)^s g_{w,w'}(q^s)\,\d s \right| \leq \int_{C^{arc}_{k}} \frac{|\d s|}{2\pi} C' c r^x q^{-x M} e^{-y\sigma - \pi |y|}.
$$
By assumption, $rq^{-M}< e^{-\nu}$ for some $\nu>0$, implying that $r^x q^{-xM}\leq e^{-\nu x}$. Plugging this in, and using the fact that $\sigma\in (-\pi,\pi)$, we see that the integrand decays exponentially as both $|y|$ and $x$ increase. This means that as $k$ goes to infinity, the above integral converges to zero. This completes the verification necessary to apply Lemma~\ref{gammasumlemma} and hence \eqref{eqsummingmellin} follows and the proposition is proved.
\end{proof}

\subsubsection{Step 1f: Analytic continuation}
We show that for $\tilde a, \tilde\beta$ fixed, we may use analytic continuation to extend the domain of applicability of Proposition~\ref{qFredDetThmbetarestricted} to hold for all $\zeta\in \C\setminus \Rplus$. This proves Theorem~\ref{qFredDetThm} in the case of a pure $\tilde\beta$ specialization.

\begin{proposition}\label{qFredDetThmbeta}
Fix $N\geq 9$ and $\tilde a,\tilde\beta$ as in Definition~\ref{defParams}. Then for all $\zeta\in \C\setminus \Rplus$
\begin{equation}\label{eqnmmthd}
\EE_{\M_{\tilde a,\rho(0;\tilde \beta;0)}}\Bigg[ \frac{1}{\big(\zeta q^{-\lambda^{(N)}_1};q\big)_{\infty}}\Bigg] = \det(\Id+\tilde K_{\zeta})_{L^2(\CwPre{\tilde a;0,\tilde\beta;\varphi})}
\end{equation}
with $\CwPre{\tilde a; 0,\tilde\beta;\varphi}$ as in Definition~\ref{CwPredef} with any $\varphi\in (0,\pi/2]$. The operator $\tilde K_{\zeta}$ is defined in terms of its integral kernel given in \eqref{eqnkzetakernel} with $g_{w,w'}(q^{s})$ from \eqref{gwwprimeeqn} explicitly given in the pure $\tilde\beta$ specialization by \eqref{gwwprimeprime}.
\end{proposition}
\begin{proof}
In order to prove this result, we demonstrate that both sides of \eqref{eqnmmthd} are analytic in $\zeta$ as it varies within $\C\setminus\Rplus$. The identity for $|\zeta|$ small enough follows from Proposition~\ref{qFredDetThmbetarestricted} and the general $\zeta$ result then follows from uniqueness of analytic continuations. Throughout, let $\tilde a,\tilde\beta$ be fixed and let $D\subset \C\setminus\Rplus$ be any compact domain bounded away from $\R_{+}$. Also, let $\varphi\in (0,\pi/2]$ be fixed as well as the contours $\CwPre{\tilde a;0,\tilde\beta;\varphi}$ and $\CsPre{w}$ for each $w\in \CwPre{\tilde a;0,\tilde\beta;\varphi}$.

To establish the analyticity of the right-hand side of \eqref{eqnmmthd} observe that
$$
\EE_{\M_{\tilde a,\rho(0;\tilde \beta;0)}}\Bigg[ \frac{1}{\big(\zeta q^{-\lambda^{(N)}_1};q\big)_{\infty}}\Bigg] = \sum_{n=0}^{\infty}
\frac{\M_{\tilde a,\rho(0;\tilde \beta;0)}\big(\lambda^{(N)}_1 = n\big)}{(\zeta q^{-n};q)_{\infty}}
$$
is analytic over $\zeta\in\C\setminus\{q^{\ell}\}_{\ell\in\Z}$. This follows from the fact that for any region of $\C$ bounded away from $\{q^{\ell}\}_{\ell\in \Z}$, the function $\zeta\to (\zeta;q)_{\infty}$ is uniformly bounded from zero and analytic. This means that the above series is uniformly convergent in any such region. Since each term is analytic in $\zeta$, the series is as well.

To establish the analyticity of the left-hand side of \eqref{eqnmmthd}, we show that \mbox{$\zeta\mapsto \det(\Id+\tilde K_{\zeta})_{L^2(\CwPre{\tilde a; 0,\tilde\beta;\varphi})}$} is an analytic function of $\zeta$ in any domain $D$ bounded away from $\R_{+}$. Consider the Fredholm determinant expansion
\begin{equation}\label{eqnfredexpansionseries}
\det(\Id+\tilde K_{\zeta})_{L^2(\CwPre{\tilde a;0,\tilde\beta;\varphi})} = \sum_{L\geq 0} \frac{1}{L!} \int_{\CwPre{\tilde a;0,\tilde\beta;\varphi}}\frac{\d w_1}{2\pi \I} \cdots \int_{\CwPre{\tilde \alpha,\varphi}} \frac{\d w_L}{2\pi \I} \det\big[\tilde K_{\zeta}(w_i,w_j)\big]_{i,j=1}^L.
\end{equation}
We will show that for each $L\geq 0$, the corresponding summand is an analytic function of $\zeta\in D$ and that uniformly over $\zeta\in D$, the above series in $L$ is absolutely convergent. From this it will follow that the series itself is also analytic. Let us write $F_{L}(\zeta)$ for the $L^{th}$ summand in \eqref{eqnfredexpansionseries}:
\begin{multline*}
F_L(\zeta) = \frac{1}{L!} \int_{\CwPre{\tilde a;0,\tilde\beta;\varphi}}\frac{\d w_1}{2\pi \I}\cdots \int_{\CwPre{\tilde a;0,\tilde\beta;\varphi}}\frac{\d w_L}{2\pi \I} \int_{\CsPre{w_1}} \frac{\d s_1}{2\pi \I} \cdots \int_{\CsPre{w_n}} \frac{\d s_L}{2\pi \I} \det\left[\frac{1}{q^{s_i}w_i -w_j}\right]_{i,j=1}^{L}\\
\times \prod_{j=1}^{L} \left(\Gamma(-s_j)\Gamma(1+s_j) (-\zeta)^{s_j} \prod_{i=1}^{N}\frac{(q^{s_j}w_j \tilde a_i;q)_{\infty}}{(w_j \tilde a_i;q)_{\infty}}\prod_{i=1}^{M}\frac{1+\tilde\beta_i (q^{s_j}w_j)^{-1}}{1+\tilde\beta_i (w_j)^{-1}} \right).
\end{multline*}

We utilize the following readily proved estimates for the integrand in \eqref{eqnfredexpansionseries}. There exists $C_0>0$ such that for all $w\in \CwPre{\tilde a;0,\tilde\beta;\varphi}$ and all $s\in \CsPre{w}$
$$
\left|\prod_{i=1}^{N}\frac{(q^{s_j}w_j \tilde a_i;q)_{\infty}}{(w_j \tilde a_i;q)_{\infty}}\prod_{i=1}^{M}\frac{1+\tilde\beta_i (q^{s_j}w_j)^{-1}}{1+\tilde\beta_i (w_j)^{-1}}\right| \leq C_0\big(|w| q\big)^{-N x/2} q^{-M x}
$$
where we recall the notation $s=x+\I y$ and $M=M_{\beta}$ is the number of non-zero elements of $\tilde\beta$. Note that the constant $C_0$ depends on $\tilde a,\tilde\beta$ and the exact choice of contours. We may bound $\Gamma(-s)\Gamma(1+s) (-\zeta)^{s}$ as in Step 1e. For $s$ on the vertical portion of $\CsPre{w}$ (with real part $R$), there exists a constant $C_1>0$ such that
$$
\Big\vert \Gamma(-s)\Gamma(1+s) (-\zeta)^{s}\Big\vert \leq C_1 y^{-1} r^R e^{-y\sigma -\pi |y|}
$$
where we write $\zeta = r e^{i \sigma}$ with $\sigma\in (-\pi,\pi)$.
For $s$ on the rest of $\CsPre{w}$, there exists a constant $C_2>0$ such that
$$
\Big\vert \Gamma(-s)\Gamma(1+s) (-\zeta)^{s}\Big\vert \leq C_2 d^{-1} r^x,
$$
where $d$ comes from Definition~\ref{CwPredef}.

Finally, from Hadamard's inequality and the conditions we have imposed on $\CsPre{w_j}$ there exists a constant $C_3>0$ such that
\begin{equation*}
\left| \det\left[\frac{1}{q^{s_i}w_i -w_j}\right]_{i,j=1}^{L} \right| \leq C_3^L L^{L/2}.
\end{equation*}

Let us see how these bounds imply the analyticity of the fixed $L$ summand in \eqref{eqnfredexpansionseries} as well as the uniform absolute convergence of the series. The integrand in \eqref{eqnfredexpansionseries} is clearly analytic in $\zeta\in \C\setminus\Rplus$. Likewise holds true for the integral in \eqref{eqnfredexpansionseries} when $w_1,\ldots, w_L$ and $s_1,\ldots, s_L$ are restricted to compact portions of their respective contours. To show that the entire integral in \eqref{eqnfredexpansionseries} is analytic over $\zeta\in D$ (for some compact domain $D$ bounded away from $\R_{+}$) it suffices to show uniform integrability of the integrand as $\zeta\in D$ varies.
See also Lemma~\ref{lem:series_anal}.

Writing $\zeta = r e^{\I \sigma}$, let $r^{*}$ represent the maximal $r$ over $\zeta\in D$ and $\sigma^{*}$ represent the $\sigma$ which is closest to $\pm \pi$ over $\zeta \in D$. Then (with possibly modified values for $C_1,C_2,C_3$ to account for the approximations that for $|w|$ large, $R \approx \ln|w|$ and $d\approx |w|^{-1}$) we find that for $s$ on the vertical portion of $\CsPre{w}$ (with real part $R$) there is a constant $C_4>0$ such that
\begin{multline*}
\left|\Gamma(-s)\Gamma(1+s) (-\zeta)^{s} \prod_{i=1}^{N}\frac{(q^{s}w \tilde a_i;q)_{\infty}}{(w \tilde a_i;q)_{\infty}}\prod_{i=1}^{M}\frac{1+\tilde\beta_i (q^{s}w)^{-1}}{1+\tilde\beta_i (w)^{-1}} \right| \\
\leq C_0 C_1 y^{-1} e^{-|y|(\pi -\sigma^{*})} \big( r^{*} |w|^{-N/2} q^{-N/2 -M}\big)^{C_4 \ln |w|},
\end{multline*}
whereas for $s$ on the rest of $\CsPre{w}$
\begin{equation*}
\left|\Gamma(-s)\Gamma(1+s) (-\zeta)^{s} \prod_{i=1}^{N}\frac{(q^{s}w \tilde a_i;q)_{\infty}}{(w \tilde a_i;q)_{\infty}}\prod_{i=1}^{M}\frac{1+\tilde\beta_i (q^{s}w)^{-1}}{1+\tilde\beta_i (w)^{-1}} \right| \leq C_0 C_2 |w| \big(r^{*} |w|^{-N/2} q^{-N/2 -M}\big)^{x}.
\end{equation*}
Performing the integrals over $s_1,\ldots, s_L$ we find that some constants $C_5,C_6>0$,
\begin{equation*}\begin{aligned}
&\int_{\CsPre{w_1}} \frac{\d s_1}{2\pi \I} \cdots \int_{\CsPre{w_n}} \frac{\d s_L}{2\pi \I} \Bigg|\det\left[\frac{1}{q^{s_i}w_i -w_j}\right]_{i,j=1}^{L}\\
&\times \prod_{j=1}^{L} \left(\Gamma(-s_j)\Gamma(1+s_j) (-\zeta)^{s_j} \prod_{i=1}^{N}\frac{(q^{s_j}w_j \tilde a_i;q)_{\infty}}{(w_j \tilde a_i;q)_{\infty}}\prod_{i=1}^{M}\frac{1+\tilde\beta_i (q^{s_j}w_j)^{-1}}{1+\tilde\beta_i (w_j)^{-1}} \right)\Bigg|\\
&\leq C_3^L L^{L/2} \prod_{i=1}^L \left( C_5 \ln|w| \big(r^{*} |w_i|^{-N/2} q^{-N/2 -M}\big)^{C_4 \ln |w|} + C_6 |w| \frac{\big(r^{*} |w_i|^{-N/2} q^{-N/2 -M}\big)^{1/2}}{\ln(r^{*} |w_i|^{-N/2} q^{-N/2 -M}\big)}\right).
\end{aligned}\end{equation*}
The right-hand side above decays in large $|w|$ like $|w|^{1-N/4}$ (up to logarithmic corrections). Thus, given that $N\geq 9$, we find that for some constant $C_7>0$,
\begin{multline*}
\frac{1}{L!} \int_{\CwPre{\tilde a;0,\tilde\beta;\varphi}}\frac{\d w_1}{2\pi \I}\cdots \int_{\CwPre{\tilde a;0,\tilde\beta;\varphi}}\frac{\d w_L}{2\pi \I} \int_{\CsPre{w_1}} \frac{\d s_1}{2\pi \I} \cdots \int_{\CsPre{w_n}} \frac{\d s_L}{2\pi \I} \,\Bigg|\det\left[\frac{1}{q^{s_i}w_i -w_j}\right]_{i,j=1}^{L}\\
\times \prod_{j=1}^{L} \left(\Gamma(-s_j)\Gamma(1+s_j) (-\zeta)^{s_j} \prod_{i=1}^{N}\frac{(q^{s_j}w_j \tilde a_i;q)_{\infty}}{(w_j \tilde a_i;q)_{\infty}}\prod_{i=1}^{M}\frac{1+\tilde\beta_i (q^{s_j}w_j)^{-1}}{1+\tilde\beta_i (w_j)^{-1}} \right)\Bigg| \leq \frac{C_3^L L^{L/2}}{L!} C_7.
\end{multline*}

This implies $F_L(\zeta)$ is analytic, and it also shows that $|F_{L}(\zeta)| \leq \frac{C_3^L L^{L/2}}{L!} C_7$ uniformly over $\zeta \in D$. Hence follows the uniform absolute convergence and analyticity of the full series $\det(\Id+\tilde K_{\zeta})_{L^2(\CwPre{\tilde a;0,\tilde\beta;\varphi})}$ as well by Lemma~\ref{lem:series_anal}.
\end{proof}

\subsection{Step 2: Formal series identity}

In this step, we will complete the proof of Theorem~\ref{qFredDetThm} for general $\tilde\alpha,\tilde\beta,\tilde\gamma$ specializations. Recall that the algebra of symmetric functions $\Sym$ in formal variables $X=(x_1,x_2,\ldots)$ is algebraically generated by Newton power sums (for more background, see~\cite[Chapter~1]{Mac79})
$$
p_k(X) = \sum_i (x_i)^{k}, \qquad k=1,2,\ldots.
$$
For any partition $\lambda$, set $p_{\lambda}(X) = \prod_i p_{\lambda_i}(X)$. These form a linear basis for $\Sym$.

In order to prove the theorem, we must show that the identity in \eqref{thmlaplaceeqn} holds for general specializations satisfying the conditions of Definition~\ref{defParams}. We do this in three lemmas. In Lemma~\ref{lemstep2a} we establish that the left-hand side of \eqref{thmlaplaceeqn} can be expanded into a series in the $\rho(\tilde \alpha;\tilde \beta;\tilde \gamma)$ specialization of $p_{\lambda}(X)$ functions with coefficients $\ell_{\lambda}(\zeta,q,\tilde{a})$ which are independent of said specialization. In Lemma~\ref{lemstep2b} we do the same for the right-hand side of \eqref{thmlaplaceeqn} with coefficients $r_{\lambda}(\zeta,q,\tilde{a})$. In Lemma~\ref{lemstep2c} we observe that since Proposition~\ref{qFredDetThmbeta} amounts to the identity \eqref{thmlaplaceeqn} for all pure $\tilde\beta$ specializations, this implies the equality of all coefficients $\ell_{\lambda}(\zeta,q,\tilde{a})=r_{\lambda}(\zeta,q,\tilde{a})$ in the $p_k$ expansions. This along with the two previous lemmas, however,
implies that \eqref{thmlaplaceeqn} holds for all specializations for which the $p_k$ expansions are absolutely convergent -- in particular for the general $\rho(\tilde \alpha;\tilde \beta;\tilde \gamma)$ specialization satisfying Definition~\ref{defParams}. This completes the proof of Theorem~\ref{qFredDetThm}.

What remains, therefore, is to state and prove the three lemmas.

\begin{lemma}\label{lemstep2a}
There exist coefficients $\ell_{\lambda}=\ell_{\lambda}(\zeta,q,\tilde{a})$, depending on $\zeta\in \C\setminus\Rplus$, $q\in(0,1)$ and $\tilde{a}$ (but not $\rho(\tilde \alpha;\tilde \beta;\tilde \gamma)$) such that, for all specializations $\rho(\tilde \alpha;\tilde \beta;\tilde \gamma)$ satisfying Definition~\ref{defParams},
$$
\EE_{\M_{\tilde a,\rho(\tilde \alpha;\tilde \beta;\tilde \gamma)}}\Bigg[ \frac{1}{\big(\zeta q^{-\lambda^{(N)}_1};q\big)_{\infty}}\Bigg] = \sum_{\lambda} \ell_{\lambda}\, p_{\lambda}\big(\rho(\tilde \alpha;\tilde \beta;\tilde \gamma)\big).
$$
Moreover, the right-hand side is an absolutely convergent series for all such $\rho(\tilde \alpha;\tilde \beta;\tilde \gamma)$.
\end{lemma}
\begin{proof}
Let us first work in terms of the formal variables $X$ and the algebra $\Sym$. Since the $p_{\lambda}$ form a linear basis of $\Sym$, there are coefficients $c_{\lambda,\mu}$ with $|\lambda|=|\mu|$ such that
$$
Q_{\lambda}(X) = \sum_{\mu:|\mu|=|\lambda|} c_{\lambda,\mu} p_{\lambda}(X).
$$
Similarly, we can express
$
\Pi(\tilde a;X) = \sum_{\lambda} P_{\lambda}(\tilde a) Q_{\lambda}(X)
$
via the $p_{\lambda}(X)$ basis as
\begin{equation}\label{pieq}
\Pi(\tilde a;X) = \exp\left(\sum_{k=1}^{\infty} \frac{p_{k}(\tilde{a}) p_{k}(X)}{(1-q^k)k}\right).
\end{equation}
Using these expansions in $p_{\lambda}(X)$ functions we can write
\begin{equation}\label{eqnretg}
\sum_{\lambda^{(N)}} \frac{1}{\big(\zeta q^{-\lambda^{(N)}_1};q\big)_{\infty}} \frac{P_{\lambda}(\tilde a) Q_{\lambda}(X)}{\Pi(\tilde a;X)} = \sum_{\lambda} \ell_{\lambda}(\zeta,q,\tilde{a})\, p_{\lambda}(X).
\end{equation}
To establish the equality, we have used the above expansions of $Q$ and $\Pi$ into the $p_{\lambda}$ functions. It is easy to see that for each $p_{\lambda}(X)$ there are only a finite number of terms from these expansions which combine to form the coefficient $\ell_{\lambda}(\zeta,q,\tilde{a})$.

This determines the value of the coefficients $\ell_{\lambda}(\zeta,q,\tilde{a})$. It is, moreover, evident that the expansions are absolutely convergent for specializations $\rho(\tilde \alpha;\tilde \beta;\tilde \gamma)$ satisfying Definition~\ref{defParams}. The specialization of the right-hand side of \eqref{eqnretg} is identified with
$$\EE_{\M_{\tilde a,\rho(\tilde \alpha;\tilde \beta;\tilde \gamma)}}\Bigg[ \frac{1}{\big(\zeta q^{-\lambda^{(N)}_1};q\big)_{\infty}}\Bigg],$$ hence completing the proof of the lemma.
\end{proof}

\begin{lemma}\label{lemstep2b}
There exist coefficients $r_{\lambda}=r_{\lambda}(\zeta,q,\tilde{a})$, depending on $\zeta\in \C\setminus \Rplus$, $q\in (0,1)$ and $\tilde{a}$ (but not $\rho(\tilde \alpha;\tilde \beta;\tilde \gamma)$) such that, for all specializations $\rho(\tilde \alpha;\tilde \beta;\tilde \gamma)$ satisfying Definition~\ref{defParams},
$$
\det(\Id+\tilde K_{\zeta})_{L^2(\CwPre{\tilde a; \tilde \alpha,\tilde\beta;\varphi})} = \sum_{\lambda} r_{\lambda} \,p_{\lambda}\big(\rho(\tilde \alpha;\tilde \beta;\tilde \gamma)\big).
$$
Moreover, the right-hand side is an absolutely convergent series for all such $\rho(\tilde \alpha;\tilde \beta;\tilde \gamma)$.
\end{lemma}
\begin{proof}
Recall that the Fredholm determinant means the expansion
$$
\det(\Id+\tilde K_{\zeta})_{L^2(\CwPre{\tilde a; \tilde \alpha,\tilde\beta;\varphi})} = \sum_{L\geq 0}\frac{1}{L!} \int_{\CwPre{\tilde a; \tilde \alpha,\tilde\beta;\varphi}}\frac{\d w_1}{2\pi \I}\cdots \int_{\CwPre{\tilde a; \tilde \alpha,\tilde\beta;\varphi}}\frac{\d w_L}{2\pi \I} \det\left[ \tilde K_{\zeta}(w_i,w_j)\right]_{i,j=1}^L.
$$
The kernel is defined in \eqref{eqnkzetakernel} as
$$
\tilde K_{\zeta}(w,w') = \frac{1}{2\pi \I}\int_{\CsPre{w}} \Gamma(-s)\Gamma(1+s)(-\zeta)^s g_{w,w'}(q^s)\,\d s
$$
where, as in \eqref{gwwprimeeqn},
$$
g_{w,w'}(q^s) = \frac{1}{q^s w - w'}\, \frac{\Pi(w;a)}{\Pi(q^s w;a)}\, \frac{\Pi\big((q^s w)^{-1};\rho(\tilde\alpha;\tilde\beta;\tilde\gamma)\big)}{\Pi\big((w)^{-1};\rho(\tilde\alpha;\tilde\beta;\tilde\gamma)\big)}.
$$
For specializations $\rho(\tilde \alpha;\tilde \beta;\tilde \gamma)$ satisfying Definition~\ref{defParams}, we can choose the contours $\CwPre{\tilde a; \tilde \alpha,\tilde\beta;\varphi}$ and $\CsPre{w}$ as in Definition~\ref{CwPredef} in such as way that
$$\bigg|\frac{\tilde\alpha}{q^s w}\bigg|, \bigg|\frac{\tilde\alpha}{w}\bigg|, \bigg|\frac{\tilde\beta}{q^s w}\bigg|, \bigg|\frac{\tilde\beta}{w}\bigg|<1
$$ for all $w\in \CwPre{\tilde a; \tilde \alpha,\tilde\beta;\varphi}$ and $s\in \CsPre{w}$.
These conditions imply that the $\rho(\tilde \alpha;\tilde \beta;\tilde \gamma)$ specialization of the identities in \eqref{pieq} remain valid (with convergent right-hand sides). In particular, we find that the term in the formula for $g_{w,w'}(q^s)$ is
$$
\frac{\Pi\big((q^s w)^{-1};\rho(\tilde\alpha;\tilde\beta;\tilde\gamma)\big)}{\Pi\big((w)^{-1};\rho(\tilde\alpha;\tilde\beta;\tilde\gamma)\big)}
= \exp\left(\sum_{k=1}^{\infty} \frac{p_{k}\big((q^s w)^{-1}\big) p_{k}\big(\rho(\tilde\alpha;\tilde\beta;\tilde\gamma)\big)}{(1-q^k)k} - \frac{p_{k}\big((w)^{-1}\big) p_{k}\big(\rho(\tilde\alpha;\tilde\beta;\tilde\gamma)\big)}{(1-q^k)k} \right).
$$
We can substitute this convergent expansion into $g_{w,w'}(q^s)$ and thus express the Fredholm determinant expansion in terms of a convergent series in the $p_{\lambda}\big(\rho(\tilde\alpha;\tilde\beta;\tilde\gamma)\big)$:
$$
\det(\Id+\tilde K_{\zeta})_{L^2(\CwPre{\tilde a; \tilde \alpha,\tilde\beta;\varphi})} = \sum_{\lambda} r_{\lambda}(\zeta,q,\tilde{a})\, p_{\lambda}\big(\rho(\tilde\alpha;\tilde\beta;\tilde\gamma)\big).
$$
It is easily checked that the $L^{th}$ term in the Fredholm determinant expansion contributes to $p_{\lambda}$'s with $|\lambda|\geq L$, and hence each coefficient is well-defined and finite. The convergence of this sum follows from the convergence of the expansion into the $p_{k}\big(\rho(\tilde\alpha;\tilde\beta;\tilde\gamma)\big)$ as well as the convergence of the Fredholm determinant expansion.
\end{proof}

\begin{lemma}\label{lemstep2c}
For any $\zeta\in \C\setminus\Rplus$, $q\in(0,1)$ and $\tilde{a}$, we have
$$
\ell_{\lambda}(\zeta,q,\tilde{a})=r_{\lambda}(\zeta,q,\tilde{a})
$$
for all partitions $\lambda$.
\end{lemma}
\begin{proof}
Proposition~\ref{qFredDetThmbeta} implies that the left-hand sides of the identities in Lemmas~\ref{lemstep2a} and~\ref{lemstep2b} are equal for all pure $\tilde\beta$ specializations. The right-hand sides of the identities in Lemmas~\ref{lemstep2a} and~\ref{lemstep2b} are therefore also equal under these specializations:
$$
\sum_{\lambda} \ell_{\lambda}(\zeta,q,\tilde a)\, p_{\lambda}\big(\rho(0;\tilde \beta;0)\big)=\sum_{\lambda} r_{\lambda}(\zeta,q,\tilde a) \, p_{\lambda}\big(\rho(0;\tilde \beta;0)\big).
$$
View both sides as power series in individual $\tilde{\beta}_j$'s. The equality implies equality of parts of fixed degree. Moreover, $p_{\lambda}\big(\rho(0;\tilde \beta;0)\big)$ for fixed $|\lambda|$ are linearly independent for sufficiently many nonzero $\tilde{\beta}_j$'s. Therefore, since this equality holds for general $\tilde\beta$ specializations, it implies equality of the expansion coefficients.
\end{proof}

\section{Whittaker processes}\label{SectConvWhitt}

\subsection{Defining the processes}
Let us introduce some notation. Write $T$ for the triangular array $\big(T^{(k)}_j\big)_{1\leq j\leq k\leq N}$ with entries in $\R$. Alternatively, write $T=\big(T^{(1)},\ldots, T^{(N)}\big)$ with \mbox{$T^{(k)} = \big(T^{(k)}_1,\ldots, T^{(k)}_k\big)$}. Also, write $\nu= (\nu_1,\ldots,\nu_N)\in \R^N$.

\begin{definition}\label{defWhitfunc}
As shown by Givental~\cite{Giv97}, the class-one $\mathfrak{gl}_{N}$-Whittaker functions admit the following integral representation:
$$
\psi_{\nu}(T^{(N)}) = \int_{\R^{\frac{N(N-1)}{2}}} e^{\mathcal{F}_{\nu}(T)} \prod_{1\leq j\leq k\leq N-1} \d T^{(k)}_j
$$
where the integral is over all $T$ with fixed $T^{(N)}$, and where
$$
\mathcal{F}_{\nu}(T)=\I\sum_{n=1}^{N} \nu_n\left(\sum_{i=1}^n T^{(n)}_i-\sum_{i=1}^{n-1} T^{(n-1)}_i\right)-\sum_{n=1}^{N-1}\sum_{i=1}^n \left(e^{T^{(n)}_i-T^{(n+1)}_i}+e^{T^{(n+1)}_{i+1}-T^{(n)}_i}\right).
$$
\end{definition}

We now define a class of Whittaker processes which are composites of those which arose in~\cite{OCon09,COSZ11}.
\begin{definition}
Fix integers $N\geq 1$, $M\geq 0$, vectors $a=(a_1,\ldots, a_N)$, $\alpha=(\alpha_1,\ldots, \alpha_M)$, and $\tau\geq 0$, such that $\alpha_m>0$ for all $1\leq m\leq M$ and $\alpha_m+a_n>0$ for all $1\leq n\leq N$ and $1\leq m\leq M$. The {\it Whittaker process} corresponding to these parameters is a probability measure on $\R^{\frac{N(N+1)}{2}}$ with density function (with respect to Lebesgue measure) given by
$$
\W{a;\alpha,\tau}(T) = e^{-\tau \sum_{n=1}^{N} \frac{a_n^2}{2}} \prod_{n=1}^{N}\prod_{m=1}^{M} \frac{1}{\Gamma(\alpha_m+a_n)} e^{\mathcal{F}_{\I a}(T)} \theta_{\alpha,\tau}(T^{(N)})
$$
with
$$
\theta_{\alpha,\tau}(T^{(N)}) = \int_{\R^N}\psi_{\nu}(T^{(N)}) e^{-\tau \sum_{n=1}^{N} \frac{\nu_n^2}{2}} \prod_{n=1}^{N}\prod_{m=1}^{M} \Gamma(\alpha_m - \I \nu_n) m_N(\nu)\,\d\nu_1\ldots\d\nu_N,
$$
and the Skylanin measure $m_N$ defined as
$$
m_N(\nu) = \frac{1}{(2\pi)^N N!} \prod_{1\leq j\neq k\leq N} \frac{1}{\Gamma(\I \nu_k - \I \nu_j)}.
$$

The Whittaker measure $\WM{a;\alpha,\tau}(T^{(N)})$ is the marginal of the Whittaker process $\W{a;\alpha,\tau}(T)$ on $T^{(N)}$ as defined in~\cite[Definition 4.1.16]{BC11}.
\end{definition}

\subsection{Whittaker processes and the semi-discrete directed random polymer}

The following result connects the developments of Sections~\ref{SectMacdonald} and~\ref{SectConvWhitt} with the study of the partition function for the semi-discrete directed random polymer with log-gamma boundary sources.

\begin{theorem}\label{ThmConnectWhitpoly}
Fix integers $N\geq 1$, $M\geq 0$ and $\tau\geq 0$. Let $a=(a_1,\ldots,a_N)\in \R^N$ and \mbox{$\alpha = (\alpha_1,\ldots,\alpha_M)\in \big(\R_{>0}\big)^M$} be such that $\alpha_m-a_n>0$ for all $1\leq n\leq N$ and $1\leq m\leq M$. Recall $\mathbf{F}_j^{k,M}(\tau)$ defined in \eqref{eqfreeenergy}. Then $\big\{\mathbf{F}_j^{k,M}(\tau)\big\}_{1\leq j\leq k\leq N}$ is distributed according to the Whittaker process $\W{-a;\alpha,\tau}$, where $-a = (-a_1,\ldots, -a_N)$.
\end{theorem}
\begin{proof}
This result relies on a combination of the work of~\cite{OCon09} on the semi-discrete directed random polymer and of~\cite{COSZ11} on the log-gamma discrete directed polymer. Those papers use geometric liftings of the Robinson--Schensted--Knuth correspondence to relate the polymer partition functions to pure $\tilde\gamma$ and pure $\tilde\alpha$ specialized Whittaker processes. The present result follows by combining~\cite[Theorems~3.7 and 3.9]{COSZ11} with~\cite[Theorem~3.1]{OCon09}. See also~\cite[Section~5.2.1]{BC11} for the $M=0$ case.
\end{proof}

\begin{remark}\label{remint}
It follows from Theorem~\ref{ThmConnectWhitpoly} that the Whittaker process is positive and integrates to 1. It should also be possible to show this directly in the manner of~\cite[Proposition~4.1.18]{BC11}.
\end{remark}

\subsection{Convergence of $q$-Whittaker processes to Whittaker processes}
We start by recording how $q$-Whittaker polynomials limit to Whittaker functions. Note that in the scalings which we describe below, it is understood that when it is necessary to work with integers, we take the integer part of $\e$ dependent expressions.

\begin{proposition}[Theorem~4.1.7 of~\cite{BC11}]\label{qwhitconvTHM}
For $N\geq 1$, consider the scalings
$$q=e^{-\e},\qquad \mathcal{A}(\e) = -\e^{-1}\frac{\pi^2}{6} - \frac{1}{2}\ln\frac{\e}{2\pi}$$
and for $1\leq n\leq N$
$$z_n =e^{\I \e \nu_n},\qquad \lambda^{(N)}_n = (N-2n)\e^{-1}\log\e^{-1} + \e^{-1} T^{(N)}_n.$$
Define rescaled (and index, variable flipped) $q$-Whittaker functions by
$$
\psi^{\e}_{\nu}(T^{(N)}) = \e^{\frac{N(N-1)}{2}} e^{\frac{N(N+2)}{2} \mathcal{A}(\e)} P_{\lambda^{(N)}}(z).
$$
Then, for all $\nu\in \R^N$, we have the following:
\begin{enumerate}
\item For each $\sigma\subset \{1,\ldots,N-1\}$, there exists a polynomial $R_{N,\sigma}$ of $N$ variables (chosen independently of $\nu_1,\ldots,\nu_N$ and $\e$) such that for all $T^{(N)}\in \R^N$ with
 $$\sigma = \sigma(T^{(N)}) := \big\{n\in \{1,\ldots,N-1\}: T^{(N)}_n-T^{(N)}_{n+1}\leq 0\big\},$$
 we have the following estimate: for some $c^*>0$
\begin{equation*}
\big|\psi^{\e}_{\nu}(T^{(N)})\big| \leq R_{N,\sigma(T^{(N)})}(T^{(N)}) \prod_{n\in \sigma(T^{(N)})} \exp\{-c^* e^{-(T^{(N)}_n-T^{(N)}_{n+1})/2}\}.
\end{equation*}
\item For $(T^{(N)})$ varying in a compact domain of $\R^{N}$, $\psi^{\e}_{\nu}(T^{(N)})$ converges (as $\e$ goes to zero) uniformly to $\psi_{\nu}(T^{(N)})$.
\end{enumerate}
\end{proposition}

\begin{theorem}\label{thmweakconv}
Fix integers $N\geq 1$, $M\geq 0$, vectors $a=(a_1,\ldots, a_N)$, $\alpha=(\alpha_1,\ldots, \alpha_M)$ and $\tau> 0$ such that $\alpha_m>0$ for all $1\leq m\leq M$ and $\alpha_m+a_n>0$ for all $1\leq n\leq N$ and $1\leq m\leq M$. Introduce the following $\e>0$ dependent scalings:
\begin{equation}\label{pgfourfour}\begin{gathered}
q= e^{-\e},\qquad \tilde \gamma= \tau \e^{-2},\qquad \tilde{a}_n= e^{-a_n\e},\, 1\leq n\leq N,\qquad \tilde{\alpha}_m= e^{-\alpha_m\e},\, 1\leq m\leq M,\\
\lambda^{(k)}_{j} = \tau \e^{-2} + M \e^{-1}\ln\e^{-1} + (k+1-2j)\e^{-1}\ln\e^{-1} + T^{(k)}_{j} \e^{-1},\, 1\leq j\leq k\leq N.
\end{gathered}\end{equation}
The $q$-Whittaker process $\M_{\tilde a,\rho(\tilde \alpha;0;\tilde \gamma)}\big(\la^{(1)},\ldots,\la^{(N)}\big)$ induces an $\e$-indexed measure on $T$ which converges weakly, as $\e\to 0$, to the Whittaker process
$\W{a;\alpha,\tau}(T)$.
\end{theorem}
\begin{remark}
The above theorem only deals with convergence of the $\tilde\alpha,\tilde\gamma$ specialized \mbox{$q$-Whittaker} process. It is presently unclear whether the $\tilde\beta$ specialized process admits a non-trivial limit as $q\to 1$.
\end{remark}
\begin{proof}
This proof is quite similar to that of~\cite[Theorems~4.1.12 and 4.2.4]{BC11} which work with (respectively) the pure $\tilde \gamma$ and pure $\tilde \alpha$ cases. It should be noted, however, that the pure $\tilde \alpha$ case~\cite[Theorem~4.2.4]{BC11} was stated modulo a decay estimate which was not checked. By combining the $\tilde \gamma$ with the $\tilde \alpha$ specialization, the necessary decay is easily shown to hold. On account of the similarities to those theorems, we only include the steps of the proof and refer to the proofs from~\cite{BC11} for the justification of the estimates.

The $q$-Whittaker process which we seek to asymptotically analyze is given as
$$
\M_{\tilde a,\rho(\tilde \alpha;0;\tilde \gamma)}\big(\la^{(1)},\ldots,\la^{(N)}\big) = \frac{P_{\la^{(1)}}(\tilde a_1)P_{\la^{(2)}/\la^{(1)}}(\tilde a_2)\cdots P_{\la^{(N)}/\la^{(N-1)}}(\tilde a_N) Q_{\la^{(N)}}\big(\rho(\tilde \alpha;0;\tilde \gamma)\big)}{\Pi\big(\tilde a;\rho(\tilde \alpha;0;\tilde \gamma)\big)}.
$$

Through the association of the $\la^{(k)}_{j}$ with the $T^{(k)}_j$ given in \eqref{pgfourfour}, this measure is pushed forward to one on $T$. It suffices to show that for any compact set $D\in \R^{\frac{N(N+1)}{2}}$, the convergence (as $\e$ goes to zero) is uniform as $T$ varies in $D$. This is due to the positivity of the measure and our independent knowledge (see Remark~\ref{remint}) that the limiting density integrates to 1. In order to estimate the behavior of the $q$-Whittaker process, we split it up into three lemmas, the combination of which proves the theorem.
\begin{lemma}
Fix any compact subset $D\in \R^{\frac{N(N+1)}{2}}$. Then
$$
P_{\la^{(1)}}(\tilde a_1)P_{\la^{(2)}/\la^{(1)}}(\tilde a_2)\cdots P_{\la^{(N)}/\la^{(N-1)}}(\tilde a_N) = \Big(e^{-\frac{(N-1)(N-2)}{2} \mathcal{A}(\e)} e^{-\e^{-1} \tau \sum_{n=1}^{N} a_n} \e^{M\sum_{n=1}^{N} a_n}\Big) \mathcal{F}_{\I a}(T) e^{o(1)}
$$
where the $o(1)$ error goes uniformly (with respect to $T\in D$) to zero as $\e$ goes to zero.
\end{lemma}
\begin{proof}
This is proved by combining the computations of~\cite[Lemmas~4.1.23 and 4.2.5]{BC11}.
\end{proof}

\begin{lemma}
We have
$$
\Pi\big(\tilde a;\rho(\tilde \alpha;0;\tilde \gamma)\big) = \Big(e^{\tau N \e^{-2}} e^{-\e^{-1}\tau\sum_{n=1}^{N} a_n} \prod_{n=1}^{N}\prod_{m=1}^{M}\frac{1}{e^{\mathcal{A}(\e)}\e^{1-\alpha_m-a_n}}\Big) e^{\tau \sum_{n=1}^{N}a_n^2/2} \prod_{n=1}^{N}\prod_{m=1}^{M} \Gamma(\alpha_m+a_n) e^{o(1)}
$$
where the $o(1)$ error goes to zero as $\e$ goes to zero.
\end{lemma}
\begin{proof}
This is proved by combining the computations of~\cite[Lemmas~4.1.24 and 4.2.6]{BC11}.
\end{proof}
\begin{lemma}
Fix any compact subset $D\in \R^{\frac{N(N+1)}{2}}$. Then
$$
Q_{\la^{(N)}}\big(\rho(\tilde \alpha;0;\tilde \gamma)\big) = \Big(e^{\frac{(N-1)(N-2)}{2}\mathcal{A}(\e)} e^{\tau N \e^{-2}} \e^{\frac{N(N+1)}{2}} \prod_{m=1}^{M} \e^{N(\alpha_m-1)}\Big) \theta_{\alpha,\tau}(T^{(N)}) e^{o(1)}
$$
where the $o(1)$ error goes uniformly (with respect to $T\in D$) to zero as $\e$ goes to zero.
\end{lemma}
\begin{proof}
This is proved by combining the computations of~\cite[Lemmas~4.1.25 and 4.2.7]{BC11}. However, since the result of~\cite[Lemma~4.2.7]{BC11} was stated modulo a decay estimate, we will provide the steps to prove the above result. That decay estimate is readily confirmed in the present case because of the presence of the $\tilde \gamma$ specialization, which provides ample decay. In~\cite{BC11}, the proof of these analogous lemmas split into four steps. It is only in the fourth step where a bit more justification is needed, which we give.

We employ the torus scalar product~\cite[Section~2.1.5]{BC11} with respect to which the Macdonald polynomials are orthogonal (we keep $t=0$ and use the notation $\T^N$ to represent the torus $\{z:|z_1|,\ldots, |z_N|=1\}$):
\begin{equation*}
\langle f,g\rangle'_N = \int_{\T^N} f(z) \overline{g(z)} m_N^{q}(z) \prod_{n=1}^{N} \frac{\d z_n}{z_n}, \qquad m_N^q(z) = \frac{1}{(2\pi \I)^{N} N!}\prod_{1\leq j\neq k\leq N} (z_jz_k^{-1};q)_{\infty}.
\end{equation*}
Note that taking $t=0$~\cite[equation~(2.8)]{BC11} yields
\begin{equation*}
\langle P_{\lambda^{(N)}},P_{\lambda^{(N)}}\rangle'_N = \prod_{n=1}^{N-1} (q^{\lambda^{(N)}_n-\lambda^{(N)}_{n+1}+1};q)_{\infty}^{-1}.
\end{equation*}
Recalling the definition of $\Pi$, we may write
\begin{equation*}
Q_{\lambda^{(N)}}(\rho) = \frac{1}{\langle P_{\lambda^{(N)}},P_{\lambda^{(N)}}\rangle'_N} \big\langle \Pi(z_1,\ldots, z_N;\rho),P_{\lambda^{(N)}}(z_{1},\ldots, z_{N})\big\rangle'_N.
\end{equation*}
Therefore, in order to study the asymptotic behavior of $Q_{\la^{(N)}}\big(\rho(\tilde \alpha;0;\tilde \gamma)\big)$, we will study the torus scalar product above. Let us introduce one additional scaling to those in \eqref{pgfourfour} that for $1\leq n\leq N$, $z_n = e^{\e \I\nu_n}$.

In Step 1 we show that $\langle P_\lambda,P_\lambda\rangle'_N = e^{o(1)}$ where the $o(1)$ error goes uniformly (with respect to $T\in D$) to zero as $\e$ goes to zero. The proof from~\cite[Lemma~4.1.25]{BC11} applies just as well here.

In Step 2 we find that for any compact subset $V\subset \R^{N}$,
\begin{align*}
\Pi\big(z_1,\ldots, z_N;\rho(\tilde \alpha;0;\tilde \gamma)\big) &= E_{\Pi} e^{-\tau \sum_{n=1}^{N} \nu_n^2/2} \prod_{n=1}^{N}\prod_{m=1}^{M}\Gamma(\alpha_m-\I\nu_n) e^{o(1)},\\
E_{\Pi}&= e^{\tau N \e^{-2}} e^{\tau \e^{-1} \I \sum_{n=1}^{N} \nu_n} \prod_{n=1}^{N}\prod_{m=1}^{M}\frac{1}{e^{\mathcal{A}(\e)}\e^{1-\alpha_m+\I \nu_n}},
\end{align*}
and, using Proposition~\ref{qwhitconvTHM}, we find
\begin{align*}
P_{\lambda^{(N)}}(z_1,\ldots, z_N) &= E_{P}\,\psi_{\nu}(T^{(N)}) e^{o(1)},\\
E_{P}&=\e^{-\frac{N(N-1)}{2}} e^{-\frac{(N-1)(N+2)}{2}\mathcal{A}(\e)}e^{\tau \e^{-1}\I\sum_{n=1}^{N} \nu_n} \e^{-M\sum_{n=1}^{N} \I \nu_n}
\end{align*}
where the $o(1)$ error goes uniformly (with respect to $T\in D$ and $\nu\in V$) to zero as $\e$ goes to zero. The proof from~\cite[Lemma~4.1.25]{BC11} applies just as well here.

In Step 3 we find that for any compact subset $V\subset \R^{N}$,
\begin{equation*}
m_N^{q}(z)\prod_{n=1}^{N} \frac{\d z_n}{z_n} = E_m m_N(\nu) \prod_{n=1}^{N} \d\nu_n e^{o(1)}, \qquad E_{m}=\e^{N^2}e^{N(N-1)\mathcal{A}(\e)}
\end{equation*}
where the $o(1)$ error goes uniformly (with respect to $\nu\in V$) to zero as $\e$ goes to zero. The proof from~\cite[Lemma~4.1.25]{BC11} applies just as well here.

In Step 4 we find that
\begin{multline*}
\big\langle \Pi(z_1,\ldots, z_N;\rho),P_{\lambda^{(N)}}(z_{1},\ldots, z_{N})\big\rangle'_N \\
= \left(e^{\frac{(N-1)(N-2)}{2}\mathcal{A}(\e)} e^{\tau N \e^{-2}} \e^{\frac{N(N+1)}{2}}\prod_{m=1}^{M} \frac{1}{\e^{N(1-\alpha_m)}} \right)\theta_{\tau}(T^{(N)}) e^{o(1)}
\end{multline*}
where the $o(1)$ error goes uniformly (with respect to $T\in D$) to zero as $\e$ goes to zero.
The proof from~\cite[Lemma~4.1.25]{BC11} applies just as well here, though we need to check that the following inequality still holds: for all $\nu_n\in [-\e^{-1}\pi,\e^{-1}\pi]$, $1\leq n\leq N$,
\begin{equation}\label{decaynueqn}
\left|\frac{\Pi(z_1,\ldots, z_N;\rho)}{E_{\Pi}}\right| \leq e^{-\tau \sum_{n=1}^{N} \nu_n^2/6}.
\end{equation}
This was checked in Step 4 of the proof of~\cite[Lemma~4.1.25]{BC11} for $\rho =\rho(0;0;\tilde \gamma)$. It is, however, easily confirmed that including the $\tilde\alpha$ specialization as well as the $\tilde \gamma$ one does not increase the left-hand side of \eqref{decaynueqn}. In particular, this amounts to showing that for all $\nu\in [-\e^{-1}\pi,\e^{-1}\pi]$ (and $\alpha>0$ fixed),
$$
\frac{e^{\mathcal{A}(\e)} \e^{1-\alpha+\I \nu}}{(e^{-\e \alpha}e^{\e \I \nu};e^{-\e})_{\infty}} = \Gamma_q(\alpha- \I \nu)
$$
is bounded by a constant (cf.\ Appendix~\ref{qSec} for more on the $q$-Gamma function). This is easily checked, hence Step 4 proceeds and the lemma is proved as in~\cite[Lemma~4.1.25]{BC11}.
\end{proof}

As in the proof of~\cite[Theorems~4.1.12 and 4.2.4]{BC11}, the above three lemmas (along with the Jacobian factor of $\e^{\frac{N(N+1)}{2}}$ coming from the rescaling of the $q$-Whittaker process) implies Theorem~\ref{thmweakconv}.
\end{proof}

\section{Semi-discrete polymer with boundary sources -- Proof of Theorem~\ref{ThmFormulaSemiDiscrete}}\label{SectSemiDirected}

\begin{theorem}\label{ThmFormulaWhit}
Fix $N\geq 9$, $M\geq 0$, $\tau>0$ and vectors $a=(a_1,\ldots, a_N)\in \R^N$ and \mbox{$\alpha=(\alpha_1,\dots,\alpha_M)\in\R^M$} so that $\alpha_m-a_n>0$ for all $1\leq m\leq M$, $1\leq n\leq N$. Then for all $u\in \C$ with positive real part
\begin{equation*}
\EE_{\W{-a;\alpha,\tau}}\Big[e^{-u e^{T^{(N)}_{1}}} \Big] = \det(\Id+ K_{u})_{L^2(\Cv{a;\alpha;\varphi})}\, ,
\end{equation*}
where the operator $K_u$ is defined as in Theorem~\ref{ThmFormulaSemiDiscrete} and the contour $\Cv{a;\alpha;\varphi}$ is given in Definition~\ref{DefCaCsdefBis} with any $\varphi\in(0,\pi/4)$.
\end{theorem}

Before proving this theorem, let us see how, combined with the earlier result of Theorem~\ref{ThmConnectWhitpoly}, the above theorem yields Theorem~\ref{ThmFormulaSemiDiscrete}.

\begin{proof}[Proof of Theorem~\ref{ThmFormulaSemiDiscrete}]
Theorem~\ref{ThmConnectWhitpoly} implies that $\mathbf{Z}_1^{N,M}(\tau)$ is equal in distribution to $e^{T^{(N)}_1}$ where $T$ is distributed according to the Whittaker process $\W{-a;\alpha,\tau}$. Theorem~\ref{ThmFormulaWhit} provides a Fredholm determinant formula for the Laplace transform of $e^{T^{(N)}_1}$ which implies Theorem~\ref{ThmFormulaSemiDiscrete}.
\end{proof}

\begin{proof}[Proof of Theorem~\ref{ThmFormulaWhit}]
The proof of Theorem~\ref{ThmFormulaWhit} follows a similar line as that of~\cite[Theorem~4.5]{BCF12}. We proceed in two steps. In the first step we prove:
\begin{lemma}\label{lemstep1}
Under the scalings from \eqref{pgfourfour} and with $\zeta = -\e^{M+N}e^{-\e^{-1} \tau} u$, for all
 $u\in \C$ with positive real part
\begin{equation}\label{thmlaplaceeqnlemma}
\lim_{\e\to 0}\EE_{\M_{\tilde a^{-1},\rho(\tilde \alpha;0;\tilde \gamma)}}\Bigg[ \frac{1}{\big(\zeta q^{-\lambda^{(N)}_1};q\big)_{\infty}}\Bigg] = \EE_{\W{-a;\alpha,\tau}}\Big[e^{-u e^{T^{(N)}_{1}}} \Big]
\end{equation}
\end{lemma}

In the second step we prove:
\begin{proposition}\label{propstep2}
Under the scalings from \eqref{pgfourfour} and with $\zeta = -\e^{M+N}e^{-\e^{-1} \tau} u$, for all
 $u\in \C$ with positive real part
$$\lim_{\e\to 0} \det(\Id+\tilde K_{\zeta})_{L^2(\CwPre{\tilde a^{-1}; \tilde \alpha,0;\varphi})} = \det(\Id+ K_{u})_{L^2(\Cv{a;\alpha;\varphi})}.$$
\end{proposition}

Combining these two results along with Theorem~\ref{qFredDetThm} (which shows the equivalence of the left-hand sides of these two results) immediately yields Theorem~\ref{ThmFormulaWhit}.
\end{proof}

\subsection{Step 1: Proof of Lemma~\ref{lemstep1}}

Rewrite the left-hand side of equation \eqref{thmlaplaceeqnlemma} as
\begin{equation*}
\lim_{\e\to 0}\EE_{\M_{\tilde a^{-1},\rho(\tilde \alpha;0;\tilde \gamma)}}\Bigg[ \frac{1}{\big(\zeta q^{-\lambda^{(N)}_1};q\big)_{\infty}}\Bigg] = \lim_{\e\to 0}\EE_{\M_{\tilde a^{-1},\rho(\tilde \alpha;0;\tilde \gamma)}}\big[e_q(x_q)\big]
\end{equation*}
where
\begin{equation*}
x_q=(1-q)^{-1}\zeta q^{\lambda^{(N)}_1} =-ue^{-T^{(N)}_1} \e/(1-q)
\end{equation*}
and
$$
e_q(x) = \frac{1}{\big((1-q)x;q\big)_{\infty}}
$$
is a $q$-exponential (cf.\ Appendix~\ref{qSec}). Combine this with the fact that $e_q(x)\to e^{x}$ uniformly on $x\in~(-\infty,0)$ to show that, considered as a function of $T^{(N)}_1$, $e_q(x_q)\to e^{-u e^{-T^{(N)}_1}}$ uniformly over $T^{(N)}_1\in \R$. By Theorem~\ref{thmweakconv}, the measure on $T$ induced from the $q$-Whittaker process on $\la^{(1)},\ldots,\la^{(N)}$ converges weakly in distribution to the Whittaker process $\W{-a;\alpha,\tau}$. Combining this weak convergence with the uniform convergence of $e_q(x_q)$ and Lemma~\ref{problemma2} completes the proof of Lemma~\ref{lemstep1}.

\subsection{Step 2: Proof of Proposition~\ref{propstep2}}
Employing the change of variables $w=q^v$ and $w'=q^{v'}$, the kernel in the left-hand side of Proposition~\ref{propstep2} can be rewritten as
\begin{equation*}
\det(\Id+\tilde K_{\zeta})_{L^2(\CwPre{\tilde a^{-1}; \tilde \alpha,0;\varphi})} = \det(\Id+K_u^{\e})_{L^2(\CvEps{a;\alpha;\varphi})}.
\end{equation*}
Here, the kernel $K_u^{\e}$ is given by
\begin{equation}\label{445}
K_u^{\e}(v,v') = \frac{1}{2\pi \I}\int_{\CsPre{q^v}}h^q(s)\,\d s
\end{equation}
where (cf.\ Appendix~\ref{qSec} for the definition of $\Gamma_{q}$)
\begin{multline}\label{446}
h^q(s)=\Gamma(-s)\Gamma(1+s)\left(\frac{-\zeta}{(1-q)^{M+N}}\right)^s \frac{q^v \ln q}{q^{s+v} - q^{v'}} e^{\tilde \gamma q^{-v}(q^{-s}-1)}\\
\times \prod_{n=1}^{N} \frac{\Gamma_q(v- a_n)}{\Gamma_q(s+v- a_n)} \prod_{m=1}^M \frac{\Gamma_q(\alpha_m-s-v)}{\Gamma_q(\alpha_m-v)}.
\end{multline}

The contour on which this kernel $K_u^{\e}$ acts is the image of the contour $\CwPre{\tilde a^{-1}; \tilde \alpha;\varphi}$ under the map $x\mapsto \ln_q x$ and the contour $\CsPre{q^v}$ is as in Definition~\ref{CwPredef}.
There was some freedom in specifying the contour $\CwPre{\tilde a^{-1}; \tilde \alpha;\varphi}$. It will be convenient for us to fix a particular contour in performing asymptotics. Let $\mu = \tfrac{1}{2}\max(a)+\tfrac{1}{2}\min(\alpha)$.
Then we define the contour $\CvEps{a;\alpha;\varphi}$ as the image of \mbox{$q^\mu + e^{\pm \varphi\I}\Rplus $} under the map $x\mapsto \ln_{q} x$.
This contour is illustrated in Figure~\ref{whitasymcontours}.
Note that as $\e\to 0$ this contour converges locally uniformly to $\Cv{a;\alpha;\varphi}$ from Definition~\ref{DefCaCsdefBis}, as can readily be seen by Taylor expanding the map $x\mapsto \ln_{q} x$.

It follows from the above observation that the contour on which the kernel $K_u^{\e}$ is defined converges as $\e\to 0$ to the contour $\Cv{a;\alpha;\varphi}$
on which the kernel in Theorem~\ref{ThmFormulaWhit} is defined.
Let us now likewise demonstrate the pointwise convergence of the integrand in the integral \eqref{445} defining kernel $K_u^{\e}$ to that of the kernel $K_u$.

Consider the behavior of each term as $q\to 1$ (or equivalently as $\e\to 0$ as $q=e^{-\e}$):
\begin{align}
e^{\tau s \e^{-1}}\left(\frac{-\zeta}{(1-q)^{N+M}}\right)^s & \to u^s\,,\label{pwlimits1}\\
\frac{q^v \ln q}{q^{s+v} - q^{v'}} &\to \frac{1}{v+s-v'}\,, \label{pwlimits2}\\
\frac{\Gamma_q(v-a_m)}{\Gamma_q(v+s-a_m)} & \to \frac{\Gamma(v-a_m)}{\Gamma(s+v-a_m)}\,,\label{pwlimits3}\\
\frac{\Gamma_q(\alpha_m-s-v)}{\Gamma_q(\alpha_m-v)} & \to \frac{\Gamma(\alpha_m-s-v)}{\Gamma(\alpha_m-v)}\,,\label{pwlimits3.5}\\
e^{-\tau s \e^{-1}} \exp\left(\tilde \gamma q^{-v}(q^{-s}-1)\right) & \to e^{v \tau s + \tau s^2/2}\,.\label{pwlimits4}
\end{align}
Combining these pointwise limits together gives the integrand of the kernel $K_u$ given in Theorem~\ref{ThmFormulaSemiDiscrete}.
In order to prove convergence of the Fredholm determinant, one needs more than just pointwise convergence.

There are four things we must do to complete \emph{Step 2} and prove convergence of the determinants.
In proving convergence of Fredholm determinants it is convenient to have the contour on which the operators act be fixed as $\e$ varies.

\smallskip
In \emph{Step 2a} we deform $\CvEps{a;\alpha;\varphi}$ to a contour $\CvEpsR{a;\alpha;\varphi;r}$ with a portion $\CvRLeq{a;\alpha;\varphi;<r}$ (of distance $<r$ to the origin)
which coincides with the limiting contour $\Cv{a;\alpha;\varphi}$.

\smallskip
Then in \emph{Step 2b} we show that for any fixed $\kappa>0$, by choosing $\e_0$ small enough and $r_0$ large enough,
for all $\e<\e_0$ and $r>r_0$ the determinant restricted to $L^2(\CvRLeq{a;\alpha;\varphi;<r})$ differs from the entire determinant on $L^2(\CvEpsR{a;\alpha;\varphi;r})$ by less than $\kappa$.
Thus, at an arbitrarily small cost of $\kappa$, we can restrict to a sufficiently large radius on which the contour is independent of $\e$.

\smallskip
In \emph{Step 2c} we show that for any $\kappa>0$, for $\e$ small, the Fredholm determinant of $K_u^{\e}$ restricted to $L^2(\CvRLeq{a;\alpha;\varphi;<r})$
differs by at most $\kappa$ from the Fredholm determinant of $K_u$ restricted to the same space.

\smallskip
Finally, \emph{Step 2d} shows that for $r_0$ large enough, for all $r>r_0$ the Fredholm determinant of $K_u$ restricted to $L^2(\CvRLeq{a;\alpha;\varphi;<r})$
differs from the Fredholm determinant of $K_u$ on $L^2(\Cv{a;\alpha;\varphi})$ by at most $\kappa$.
Summing up the steps, we deform the contour, cut the contour to be finite, take the $\e\to 0$ limit, and then repair the contour to its final form
-- all with error at most $3\kappa$ for $\kappa$ arbitrarily small.

\subsubsection*{Step 2a:} We must define the contour to which we want to deform $\CvEps{a;\alpha;\varphi}$, and then justify that this deformation does not change the value of the Fredholm determinant.

\begin{definition}\label{cpctcontdef}
Fix $\varphi\in (0,\pi/4)$, $r>0$, and real numbers $a=\{a_1,\dots,a_N\}$ and \mbox{$\alpha=\{\alpha_1,\dots,\alpha_M\}$} such that $\alpha_m-a_n>0$ for all $1\leq n\leq N$ and $1\leq m\leq M$.
Define the finite contour $\CvRLeq{a;\alpha;\varphi;<r}$ to be $\{\mu+te^{(\pi+ \varphi)\I}:0\leq t\leq r\}\cup \{\mu+te^{(\pi- \varphi)\I}:0\leq t\leq r\}$ where we have set $\mu = \tfrac{1}{2}\max(a)+\tfrac{1}{2}\min(\alpha)$.
The maximal imaginary part along $\CvRLeq{a;\alpha;\varphi;<r}$ is $r\sin(\varphi)$.
Define the infinite contour $\CvEpsR{a;\alpha;\varphi;r}$ (oriented with increasing imaginary part) to be the union of $\CvRLeq{a;\alpha;\varphi;<r}$ with $\CvEpsRGeq{a;\alpha;\varphi;>r}$ and $\CvEpsReq{a;\alpha;\varphi;=r}$.
Here, the contour $\CvEpsRGeq{a;\alpha;\varphi;>r}$ is the portion of the contour $\CvEps{a;\alpha;\varphi}$ which has imaginary part exceeding $r\sin(\varphi)$ in absolute value;
and the contour $\CvEpsReq{a;\alpha;\varphi;=r}$ is composed of the two horizontal line segments which join $\CvRLeq{a;\alpha;\varphi;<r}$ with $\CvEpsRGeq{a;\alpha;\varphi;>r}$.
These contours are illustrated in Figure~\ref{whitasymcontours}.
\end{definition}

\begin{figure}
\begin{center}
\psfrag{alpha}[cb]{$\mu$}
\psfrag{C}[lb]{$\Cv{a;\alpha;\varphi}$}
\psfrag{Ce}[lb]{$\!\!\CvEps{a;\alpha;\varphi}$}
\psfrag{C1}[lb]{$\CvEpsR{a;\alpha;\varphi;>r}$}
\psfrag{C2}[lb]{$\CvEpsR{a;\alpha;\varphi;=r}$}
\psfrag{C3}[lb]{$\Cv{a;\alpha;\varphi;<r}$}
\includegraphics[height=5cm]{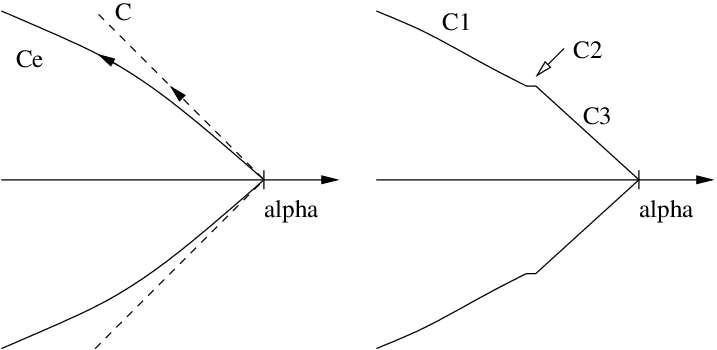}
\caption{Left: The infinite contour $\CvEps{a;\alpha;\varphi}$ and the limiting contour $\Cv{a;\alpha;\varphi}$.
Right: The infinite contour $\CvEpsR{a;\alpha;\varphi;r}$ (which we deform from $\CvEps{a;\alpha;\varphi}$).}
\label{whitasymcontours}
\end{center}
\end{figure}

Now we justify replacing the contour $\CvEps{a;\alpha;\varphi}$ by $\CvEpsR{a;\alpha;\varphi;r}$.

\begin{lemma}
For any $r>0$ there exists $\e_0>0$ such that for all $\e<\e_0$,
\begin{equation*}
\det(\Id+K_u^{\e})_{L^2(\CvEps{a;\alpha;\varphi})} = \det(\Id+K_u^{\e})_{L^2(\CvEpsR{a;\alpha;\varphi;r})}.
\end{equation*}
\end{lemma}
\begin{proof}
The two contours differ only by a finite length modification. We can continuously deform between the two contours.
We will employ Lemma~\ref{TWprop1} which says that as long as the kernel is analytic in a neighborhood of the contour as we continuously deform
then the Fredholm determinant remains unchanged throughout the deformation.
The only things which could threaten the analyticity of the kernel are the poles coming from the left-hand side terms of \eqref{pwlimits2}, \eqref{pwlimits3} and \eqref{pwlimits3.5}.
On account of the condition satisfied by the contour $\CsPre{q^v}$ (see Definition~\ref{CwPredef}), it follows that these poles are avoided.
By choosing $\e$ small enough, the two contours we are deforming between can be made as close as desired.
Taking them close enough ensures it is possible then to deform between them while avoiding poles of the kernel in $v$ or $v'$ -- hence proving the lemma.
\end{proof}

\subsubsection*{Step 2b:}
We must now show that we can, with small error, restrict our Fredholm determinant to acting on the finite, fixed contour $\CvRLeq{a;\alpha;\varphi;<r}$.
This requires us choosing $r>r_0$ for $r_0$ large enough, and also choosing $\e<\e_0$ for $\e_0$ small enough.

\begin{proposition}\label{compactifyprop}
Fix $\varphi\in (0,\pi/4)$. For any $\kappa>0$ there exist $r_0>0$ and $\e_0>0$ such that for all $r>r_0$ and $\e<\e_0$
\begin{equation*}
\left|\det(\Id+ K_u^{\e})_{L^2(\CvEpsR{a;\alpha;\varphi;r})} - \det(\Id+ K_u^{\e})_{L^2(\CvRLeq{a;\alpha;\varphi;<r})}\right| \leq \kappa.
\end{equation*}
\end{proposition}
The proof of this proposition is fairly technical and is given in Section~\ref{prop2sec}.

\subsubsection*{Step 2c:}
Having restricted our attention to the finite contour $\CvRLeq{a;\alpha;\varphi;<r}$ which does not change with $\e$, we may now take the limit of Fredholm determinants on the restricted $L^2$ space as $\e\to 0$.

\begin{proposition}\label{finiteSprop}
Fix $\varphi\in(0,\pi/4)$. For any $\kappa>0$ and any $r>0$ there exists $\e_0>0$ such that for all $\e<\e_0$
\begin{equation*}
\left|\det(\Id+ K_u^{\e})_{L^2(\CvRLeq{a;\alpha;\varphi;<r})} - \det(\Id+ K_u)_{L^2(\CvRLeq{a;\alpha;\varphi;<r})}\right|\leq \kappa
\end{equation*}
where $K_u(v,v')$ is given in Theorem~\ref{ThmFormulaSemiDiscrete}.
\end{proposition}
The proof of this proposition is also fairly technical and is given in Section~\ref{prop1sec}.

\subsubsection*{Step 2d:} Finally, we show that post-asymptotics we can return to the simple infinite contour $\Cv{a;\alpha;\varphi}$.

\begin{proposition}\label{postasymlemma}
Fix $\varphi\in(0,\pi/4)$. For any $\kappa>0$ there exists $r_0>0$ such that for all $r>r_0$
\begin{equation*}
\left| \det(\Id+ K_u)_{L^2(\CvRLeq{a;\alpha;\varphi;<r})} - \det(\Id+ K_u)_{L^2(\Cv{a;\alpha;\varphi})}\right|\leq \kappa.
\end{equation*}
\end{proposition}
The proof of this proposition is given in Section~\ref{prop3sec}. It is a fair amount more straightforward than the previous two proofs and hence is given first.

Having completed the four substeps, we may combine Propositions~\ref{compactifyprop},~\ref{finiteSprop} and~\ref{postasymlemma} to show that for any $\kappa>0$,
there exists $\e_0>0$ such that for all $\e<\e_0$,
\begin{equation*}
\left|\det(\Id+\tilde K_{\zeta})_{L^2(\CwPre{\tilde a^{-1};\tilde\alpha,0;\varphi})} - \det(\Id+ K_u)_{L^2(\Cv{a;\alpha;\varphi})} \right| \leq 3\kappa
\end{equation*}
where $\det(\Id+\tilde K_{\zeta})$ is as in the right-hand side of Proposition~\ref{propstep2}.
Since $\kappa$ is arbitrary this shows that
\begin{equation*}
\lim_{\e \to 0} \det(\Id+\tilde K_{\zeta})_{L^2(\CwPre{\tilde a^{-1};\tilde\alpha,0;\varphi})} = \det(\Id+ K_u)_{L^2(\Cv{a;\alpha;\varphi})}.
\end{equation*}

The above result completes the proof of Proposition~\ref{propstep2} modulo proving Propositions~\ref{compactifyprop},~\ref{finiteSprop} and~\ref{postasymlemma}.

\subsection{Proof of Proposition~\ref{postasymlemma}}\label{prop3sec}

By virtue of Lemma~\ref{exponentialdecaycutoff}, it suffices to show that for some $c,C>0$,
\begin{equation}\label{kuineq1}
\big|K_u(v,v')\big|\leq Ce^{-c|v|}
\end{equation}
as $v,v'$ varies along $\Cv{a;\alpha;\varphi}$.

Before proving this let us recall the contours with which we are dealing.
The variable $v$ lies on $\Cv{a;\alpha;\varphi}$ and hence can be written as $v=\mu -\kappa \cos(\varphi) \pm \I \kappa \sin(\varphi)$, for $\kappa\in \R_+$
where the $\pm$ represents the two rays of the contour.
The $s$ variables lie on $\Cs{v}$ which depends on $v$ and has two parts:
The portion (which we denote by $\Cs{v;\sqsubset}$) with real part bounded between $1/2$ and $R$ and imaginary part between $\pm d$ for $d$ sufficiently small,
and the vertical portion (which we denote by $\Cs{v;\vert}$) with real part $R$.
Recall that $R=-\Re(v)+\eta$ where $\eta =\tfrac{1}{4}\max(a)+\tfrac{3}{4}\min(\alpha)$.

Let us denote by $h(s)$ the integrand through which $K_u(v,v')$ is defined.
We split the proof into two steps. \emph{Step 1:} We show that the integral of $h(s)$ over $s\in \Cs{v;\sqsubset}$ is bounded by an expression with exponential decay in $|v|$, uniformly over $v'$.
\emph{Step 2:} We show the integral of $h(s)$ over $s\in \Cs{v;\vert}$ is bounded by an expression with exponential decay in $|v|$, uniformly over $v'$. The combination of these two steps imply the inequality \eqref{kuineq1} and hence complete the proof.

\subsubsection*{Step 1:}
We deal with the various terms in $h(s)$ separately and develop bounds for each.
Write $s=x+\I y$ and note that along the contour $\Cs{v;\sqsubset}$, $y\in [-d,d]$ for $d$ small, and $x\in [1/2,R]$.

Let us start with $e^{v\tau s + \tau s^2/2}$.
The norm of this is bounded by the exponential of the real part of the exponent. For $s$ along $\Cs{v;\sqsubset}$
\begin{equation*}
\Re(v s + s^2/2) = x\Re(v)+\frac{x^2}{2}-y\Im(v)-\frac{y^2}{2} .
\end{equation*}
Given our choice of $R=-\Re(v)+\eta$, by taking $d$ sufficiently small and using the bound $\Re(v)\leq \tilde c'-c' |v|$ for some constants $c',\tilde c'$ (depending on $\varphi$), we may deduce that
\begin{equation*}
\Re(v s + s^2/2) \leq \tilde c -c |v| x
\end{equation*}
for some constants $c,\tilde c>0$. From this it follows that
\begin{equation*}
|e^{v\tau s + \tau s^2/2}| \leq C e^{-c\tau |v|x}.
\end{equation*}
Turning to the other terms in $h(s)$, we have that
\begin{equation*}
|u^s|\leq e^{x\ln |u| - y\arg(u)}
\end{equation*}
and we may also bound
\begin{equation}\label{threebounds}
\left|\frac{\Gamma(v-a_m)}{\Gamma(s+v-a_m)}\right|, \quad \left|\frac{\Gamma(\alpha_m-v-s)}{\Gamma(\alpha_m-v)}\right|,
\quad \left|\frac{1}{v+s-v'}\right|,\quad |\Gamma(-s)\Gamma(1+s)|\quad \leq \textrm{const}
\end{equation}
for some constant $\textrm{const}>0$.
The first two bounds come from the functional equation for the Gamma function, and the last from the fact that $s$ is bounded away from $\Z$. Let us explain in some further detail the first bound (the second follows in a similar manner). Just for this argument, call $\tilde{v} = v+a_m$. It follows that $\tilde{v} = \tilde{\mu} - \kappa\cos(\varphi) \pm \I \kappa \sin(\varphi)$ with $\tilde\mu$ real and strictly positive. We can write $s = t + r$ where $t\in \Z_{\geq 0}$ and $r$ has real part bounded in $[0,1)$ and imaginary part bounded between $\pm d$. The functional equation for the Gamma function then implies that
$$
\frac{\Gamma(\tilde{v})}{\Gamma(s+\tilde{v})} = \frac{1}{s-1+\tilde v}\frac{1}{s-2+\tilde v}\cdots\frac{1}{r+\tilde v} \frac{\Gamma(\tilde v)}{\Gamma(r+\tilde v)}.
$$
As $\tilde v$ varies along its contour, all of the factors $\frac{1}{s-j+\tilde{v}}$, $j\in \Z_{\geq 1}$, are bounded in norm by a constant (uniform as $\tilde{v}$ and $s$ vary along their contours), and, in fact, all but two of those factors are bounded in norm by 1. This implies that the product of these factors is bounded by a constant (uniform as $\tilde{v}$ and $s$ vary along their contours). As for the remaining factor $\frac{\Gamma(\tilde v)}{\Gamma(r+\tilde v)}$, as $r$ varies with real part in $[0,1)$ and imaginary part in $[-d,d]$, and as $\tilde v$ varies along its contour, this ratio remains uniformly bounded by a constant. This implies the first bound in \eqref{threebounds}. The second follows in a similar manner.

Combining the bounds from \eqref{threebounds} together shows that for $|v|$ large, the portion of the integral of $h(s)$ for $s$ in $\Cs{v;\sqsubset}$ is bounded by (recall $s=x+\I y$)
\begin{equation*}
\int_{\Cs{v;\sqsubset}}|\d s| \textrm{const}\cdot e^{-c\tau |v|x + x\ln |u| -\arg(u)y} \leq C e^{-c|v|}
\end{equation*}
for some constants $c,C>0$.

\subsubsection*{Step 2:}
As above, we consider the various terms in $h(s)$ separately and develop bounds for each.
Let us write $s=R+\I y$ and note that $s\in \Cs{v;\vert}$ corresponds to $y$ varying over all $|y|\geq d$.
As in \emph{Step 1}, the most important bound will be that of $e^{v\tau s + \tau s^2/2}$.

Observe that
\begin{equation*}
\Re(v s + s^2/2) = \Re(v)R - \Im(v) y + \frac{R^2}{2} -\frac{y^2}{2} = -\frac{(y+\Im(v))^2}{2} + \frac{\Im(v)^2}{2} +\frac{R^2}{2} +\Re(v)R.
\end{equation*}
Observe that because $\varphi\in (0,\pi/4)$ and $R=-\Re(v)+\eta$,
\begin{equation*}
\frac{\Im(v)^2}{2} +\frac{R^2}{2} +\Re(v)R \leq \tilde c -c |v|^2
\end{equation*}
for some constants $c,\tilde c>0$. Thus
\begin{equation}\label{step2bdd}
\Re(v s + s^2/2) \leq -\frac{(y+\Im(v))^2}{2} +\tilde c-c |v|^2.
\end{equation}

Let us now turn to the other terms in $h(s)$. We bound
\begin{equation*}
|u^s|\leq e^{R\ln |u| - y\arg(u)}.
\end{equation*}
By standard bounds for the large imaginary part behavior, we can show
\begin{equation*}
\left|\frac{\Gamma(v-a_m)}{\Gamma(s+v-a_m)}\right| \leq C e^{\frac{\pi}{2} |y|}, \qquad \left|\frac{\Gamma(\alpha_m-v-s)}{\Gamma(\alpha_m-v)}\right| \leq C e^{-(\frac{\pi}{2}-\e)|y|}\le C
\end{equation*}
for some constant $C>0$ sufficiently large and $\e>0$ small enough. Also, $|1/(v+s-v')|\leq C$ for a fixed constant. Finally, the term
\begin{equation*}
|\Gamma(-s)\Gamma(1+s)|\leq Ce^{-\pi |y|}
\end{equation*}
for some constant $C>0$.

Combining these together shows that the integral of $h(s)$ over $s$ in $\Cs{v;\vert}$ is bounded by a constant times
\begin{equation}\label{eqnest}
\int_{\R} \exp\left(-\tau \frac{(y+\Im(v))^2}{2} -\tau c |v|^2 + R\ln|u| -y\arg(u) -\pi |y| + N\frac{\pi}{2}|y|\right)\d y.
\end{equation}

We can factor out the terms above which do not depend on $y$, giving
\begin{equation*}
\exp\left(-\tau c |v|^2 + R\ln|u| \right) \int_{\R} \exp\left(-\tau \frac{(y+\Im(v))^2}{2} -y \arg(u) + N\frac{\pi}{2}|y|\right)\d y.
\end{equation*}
Notice that the prefactors on $y$ and $|y|$ in the integrand's exponential are fixed constants.
We can therefore use the following bound that for $a$ fixed and $b\in \R$, there exists a constant $C$ such that
\begin{equation}\label{usethefollowingbound}
\int_{\R} e^{-\nu(y+b)^2 + a|y|}\d y \leq C e^{|ab|},\quad \nu>0.
\end{equation}
For $a<0$ this inequality is obvious, so let us assume $a>0$ and consider which $y$ maximizes the exponential in the integrand on the left-hand side of the inequality. Without loss of generality, we may take $b>0$ as well. It is clear that the maximizing $y$ will be negative, so we are looking for the maximum over $y<0$ of $-\nu(y+b)^2 - ay$. This is achieved when $y+b = -\frac{a}{2\nu}$ which means that the maximal argument of the exponential is $\frac{a^2}{4\nu} + ab$. It is easy to see that there is rapid decay away from this maximal value and hence the integral is bounded by a constant time $e^{\frac{a^2}{4\nu}+ab}$. Since $a$ is fixed, this is itself like a constant time $e^{ab}$. The argument for $b<0$ likewise produces a bound by a constant times $e^{-ab}$, hence inequality \eqref{usethefollowingbound} follows.

Using inequality \eqref{usethefollowingbound}, we find that we can upper-bound \eqref{eqnest} by
\begin{equation*}
\exp\left(-\tau c |v|^2 + R\ln|u|+c'|v|\right).
\end{equation*}
For $|v|$ large enough, the Gaussian decay in the above bound dominates, and hence integral of $h(s)$ over $s$ in $\Cs{v;\vert}$ is bounded by
\begin{equation*}
C e^{-c|v|}
\end{equation*}
for some constants $c,C>0$.

\subsection{Proof of Proposition~\ref{finiteSprop}}\label{prop1sec}
Fix $\kappa,r>0$. We are presently considering the Fredholm determinant of the kernels $K_u^{\e}$ and $K_u$ restricted to the fixed finite contour $\CvRLeq{a;\alpha;\varphi;<r}$.
By Lemma~\ref{uniformptconvergence}, we only need to show convergence of the kernel $K_u^{\e}(v,v')$ to $K_u(v,v')$ as $\e\to0$, uniformly over $v,v'\in \CvRLeq{a;\alpha;\varphi;<r}$.
This is achieved via showing that for all $\kappa'>0$ there exists $\e_0>0$ such that for all $\e<\e_0$ and for all $v,v'\in \CvRLeq{a;\alpha;\varphi;<r}$,
\begin{equation}\label{etapeqn}
\left|K_u^{\e}(v,v')-K_u(v,v')\right|\leq \kappa'.
\end{equation}

The kernels $K_u^{\e}$ and $K_u$ are both defined via integrals over $s$.
The contour on which $s$ is integrated can be fixed for ($\e<\e_0$) to equal $\Cs{v}$, which is the $s$ contour used to define $K_u$. The fact that the $s$ contours are the same for $K_u^{\e}$ and $K_u$ is convenient. The proof of \eqref{etapeqn} will follow from three lemmas. The first deals with the uniformity of convergence of the integrand defining $K_u^{\e}$ to the integrand defining $K_u$ for $s$ restricted to any fixed compact set.

Before stating this lemma, let us define some notation.

\begin{definition}\label{CsMdef}
Let $\Cs{v;>L}= \{s\in \Cs{v}: |s|\geq L\}$ be the portion of $\Cs{v}$ of magnitude greater than $L$ and similarly let $\Cs{v;<L}= \{s\in \Cs{v}: |s|<L\}$.
Let us assume $L$ is large enough so that $\Cs{v;>L}$ is the union of two vertical rays with fixed real part \mbox{$R=-\Re(v)+\eta$} (recall $\eta = \tfrac{1}{4}\max(a)+\tfrac{3}{4}\min(\alpha)$).
Assuming this, we will write $s=R+\I y$. Then for $y_L=(L^2- R^2)^{1/2}$, the contour \mbox{$\Cs{v;>L}=\{R+\I y: |y|\geq y_L\}$}.
\end{definition}

\begin{lemma}\label{etappclaim}
For all $\kappa''>0$ and $L>0$ there exists $\e_0>0$ such that for all $\e<\e_0$, for all $v,v'\in \CvRLeq{a;\alpha;\varphi;<r}$, and for all $s\in \Cs{v;<L}$,
\begin{equation}\label{uniflemmaeqn}
\left| h^q(s) - \Gamma(-s)\Gamma(1+s) \prod_{n=1}^{N}\frac{\Gamma(v-a_n)}{\Gamma(s+v-a_n)} \prod_{m=1}^M\frac{\Gamma(\alpha_m-v-s)}{\Gamma(\alpha_m-v)}
\frac{ u^s e^{v\tau s+\tau s^2/2}}{v+s-v'}\right|\leq \kappa''
\end{equation}
where $h^q$ is given in \eqref{446}.
\end{lemma}
\begin{proof}
This is a strengthened version of the pointwise convergence in \eqref{pwlimits1} through \eqref{pwlimits4}.
It follows from the uniform convergence of the $\Gamma_q$ function to the $\Gamma$ function on compact regions away from the poles (cf.\ Appendix~\ref{qSec}, as well as standard Taylor series estimates.
The choice of contours is such that the pole arising from $1/(v+s-v')$ is uniformly avoided in the limiting procedure as well.
\end{proof}

It remains to show that for $L$ large enough, the integrals defining $K_u^{\e}(v,v')$ and $K_u(v,v')$ restricted to $s$ in $\Cs{v;>L}$,
have negligible contribution to the kernel, uniformly over $v,v'$ and $\e$. This must be done separately for each of the kernels and hence requires two lemmas.

\begin{lemma}
For all $\kappa'>0$ there exist $L_0>0$ and $\e_0>0$ such that for all $\e<\e_0$, for all $v,v'\in \CvRLeq{a;\alpha;\varphi;<r}$, and for all $L>L_0$,
\begin{equation*}
\bigg|\int_{\Cs{v;>L}}\d s h^q(s)\bigg|\leq \kappa'.
\end{equation*}
\end{lemma}
\begin{proof}
We will use the notation introduced in Definition~\ref{CsMdef} and assume $L_0$ is large enough so that $\Cs{v;>L}$ is only comprised of two vertical rays.

Let us first consider the behavior of the left-hand side of \eqref{pwlimits4}. The magnitude of this term is bounded by the exponential of
\begin{equation*}
\Re(-\tau \e^{-1} s + \e^{-2} \tau q^{-v} (q^{-s}-1)).
\end{equation*}
This quantity is periodic in $y$ (recall $s=R+\I y$) with a fundamental domain $y\in[-\pi\e^{-1},\pi \e^{-1}]$.
For $\e^{-1}\pi>|y|>y_0$ for some $y_0$ which can be chosen uniformly in $v$ and $\e$, the following inequality holds
\begin{equation*}
\Re(-\tau \e^{-1} s + \e^{-2} \tau q^{-v} (q^{-s}-1)) \leq -\tau y^2/6.
\end{equation*}
This can is proved by careful Taylor series estimation and the inequality that for $x\in [-\pi,\pi]$, $\cos(x)-1\leq -x^2/6$.
This provides Gaussian decay in the fundamental domain of $y$.

Turning to the ratio of $q$-Gamma functions in \eqref{pwlimits3}, observe that away from its poles, the denominator
\begin{equation}\label{festatement}
\left|\frac{1}{\Gamma_q(s+v-a_m)}\right|\leq c e^{c' {\rm dist}(\Im(s),2\pi \e^{-1} \Z)}
\end{equation}
where $c,c'$ are positive constants independent of $\e$ and $v$ (as it varies in its compact contour).
This establishes a periodic bound on this denominator, which grows at most exponentially in the fundamental domain.
The numerator $\Gamma_q(v-a_m)$ in \eqref{pwlimits3} is bounded uniformly by a constant.
This is because the $v$ contour was chosen to avoid the poles of the Gamma function, and the convergence of the $q$-Gamma function to the Gamma function is uniform on compact sets away from those poles.

Similarly,
\[|\Gamma_q(\alpha_m-s-v)|\le c e^{-c''{\rm dist}(\Im(s),2\pi \e^{-1} \Z)}\le c\]
where $c, c''$ are positive constants.
This is from the uniform convergence of the $q$-Gamma function to the Gamma function which implies that $\Gamma_q(\alpha_m-v)$ remains uniformly bounded from below as $v\in\CvRLeq{a;\alpha;\varphi;<r}$ varies.

Finally, the magnitude of \eqref{pwlimits1} corresponds to $|u^s|$ and behaves like $e^{-R\ln{|u|} + y\arg(u)}$.
Thus, we have established the following inequality which is uniform in $v,v'$ and $\e$ as $y$ varies:
\begin{multline}\label{perbdd}
\left|\left(\frac{-\zeta}{(1-q)^N}\right)^s \frac{q^v \ln q}{q^{s+v} - q^{v'}} e^{\tilde \gamma q^v(q^{s}-1)} \prod_{i=1}^{N} \frac{\Gamma_q(v+\ln_q(\tilde a_i^{-1}))}{\Gamma_q(s+v+\ln_q(\tilde a_i^{-1}))}
\prod_{j=1}^M\frac{\Gamma_q(\ln_q(\tilde\alpha_j)-s-v)}{\Gamma_q(\ln_q(\tilde\alpha_j)-v)}\right|\\
\leq \tilde c\, e^{-\big({\rm dist}(\Im(s),2\pi \e^{-1} \Z)\big)^2/6+c'N\big|{\rm dist}(\Im(s),2\pi \e^{-1} \Z)\big|}
\end{multline}
for some constant $\tilde c>0$. Notice that this inequality is periodic with respect to the fundamental domain for \mbox{$y\in[-\pi\e^{-1},\pi \e^{-1}]$}.

The last term to consider is $\Gamma(-s)\Gamma(1+s)=\frac{-\pi}{\sin(\pi s)}$ which is not periodic in $y$ and decays like $e^{-\pi |y|}$ for $y\in \R$.
Since $\Cs{v;>L}$ is only comprised of two vertical rays, we must control the integral of $h^q(s)$ for $s=R+\I y$ and $|y|>y_L$.
By making sure $L$ is large enough, we can use the periodic bound \eqref{perbdd} to show that the integral over $y_L<|y|<\e^{-1} \pi$ is less than $\kappa$ (with the desired uniformity in $v,v'$ and $\e$).
For the integral over $|y|>\e^{-1}\pi$, we can use the above exponential decay of $\Gamma(-s)\Gamma(1+s)$.
On shifts by $2\pi \e^{-1}\Z$ of the fundamental domain, the exponential decay of $\Gamma(-s)\Gamma(1+s)$ can be compared to the boundedness of the other terms
(which is certainly true considering the bounds we established above). The integral of each shift can be bounded by a term in a convergent geometric series.
Taking $\e_0$ small then implies that the sum can be bounded by $\kappa'$ as well.
\end{proof}

\begin{lemma}
For all $\kappa'>0$ there exists $L_0>0$ such that for all $v,v'\in \CvRLeq{a;\alpha;\varphi;<r}$, and for all $L>L_0$,
\begin{equation*}
\left|\int_{\Cs{v;>L}} \d s \Gamma(-s)\Gamma(1+s) \prod_{n=1}^{N}\frac{\Gamma(v-a_n)}{\Gamma(s+v-a_n)} \prod_{m=1}^M\frac{\Gamma(\alpha_m-v-s)}{\Gamma(\alpha_m-v)}
\frac{u^s e^{v\tau s+\tau s^2/2}}{v+s-v'}\right|\leq\kappa'.
\end{equation*}
\end{lemma}
\begin{proof}
The desired decay here comes easily from the behavior of $vs+s^2/2$ as $s$ varies along $\Cs{v;>L}$.
As before, assume that $L_0$ is large enough so that this contour is only comprised of two vertical rays and set $s=R+ \I y$ for $y\in \R$ for $|y|>y_L$.
As in the proof of Proposition~\ref{postasymlemma} given in Section~\ref{prop3sec}, one shows that
\begin{equation*}
|e^{v\tau s+\tau s^2/2}|\leq C e^{-cy^2}
\end{equation*}
uniformly over $v,v'\in \CvRLeq{a;\alpha;\varphi;<R}$, and for all $L>L_0$.
This behavior should be compared to that of the other terms: $|\Gamma(-s)\Gamma(1+s)|\approx e^{-\pi |y|}$; $|u^s|= e^{-R\ln{|u|} + y\arg(u)}$;
$\left|\frac{\Gamma(v-a_n)}{\Gamma(s+v-a_n)}\right|\leq C e^{|y| \pi/2}$; $\left|\frac{\Gamma(\alpha_m-v-s)}{\Gamma(\alpha_m-v)}\right|\leq C e^{|y|(\pi/2-\e)}$; and $|1/(v+s+v')|\leq C$ as well.
Combining these observations we see that the integral decays in $|y|$ at worst like $C e^{-cy^2+c' |y|}$.
Thus, by choosing $L$ large enough so that $y_L\gg 1$, we can be assured that the integral over $|y|>y_L$ is as small as desired, proving the lemma.
\end{proof}

Let us now combine the above three lemmas to finish the proof of the Proposition~\ref{finiteSprop}.
Choose $\kappa'=\kappa/3$ and fix $L_0$ and $\e_0'$ as specified by the second and third of the above lemmas.
Fix some $L>L_0$ and let $\ell$ equal the length of the finite contour $\Cs{v;<L}$. Set $\kappa''=\frac{\kappa'}{3\ell}$ and apply Lemma~\ref{etappclaim}.
This yields an $\e_0$ (which we can assume is less than $\e_0'$) so that \eqref{uniflemmaeqn} holds. This implies that for $\e<\e_0$, and for all $v,v'\in \CvRLeq{\alpha,\varphi;<r}$,
\begin{multline*}
\bigg| \int_{\Cs{v;<L}} h^q(s)\,\d s\\
-\int_{\Cs{v;<L}} \Gamma(-s)\Gamma(1+s) \prod_{n=1}^{N}\frac{\Gamma(v-a_n)}{\Gamma(s+v-a_n)} \prod_{m=1}^M\frac{\Gamma(\alpha_m-v-s)}{\Gamma(\alpha_j-v)}\frac{ u^s e^{v\tau s+\tau s^2/2}}{v+s-v'}\d s \bigg|\leq \kappa'/3.
\end{multline*}
From the triangle inequality and the three factors of $\kappa'/3$ we arrive at the claimed result of \eqref{etapeqn} and thus complete the proof of Proposition~\ref{finiteSprop}.

\subsection{Proof of Proposition~\ref{compactifyprop}}\label{prop2sec}
The proof of this proposition is essentially a finite $\e$ (recall $q=e^{-\e}$) perturbation of the proof of Proposition~\ref{postasymlemma} given in Section~\ref{prop3sec}.
The estimates presently are a little more involved since the functions involved are $q$-deformations of classic functions.
However, by careful Taylor approximation with remainder estimates, all estimates can be carefully shown. By virtue of Lemma~\ref{exponentialdecaycutoff}, it suffices to show that for some $c,C>0$,
\begin{equation}\label{kueineq}
|K_u^{\e}(v,v')|\leq C e^{-c|v|}.
\end{equation}

Before proving this, let us recall from Definition~\ref{cpctcontdef} the contours with which we are dealing.
The variable $v$ lies on $\CvEpsR{a;\alpha;\varphi;r}$ for $\varphi\in (0,\pi/4)$. The $s$ variables lies on $\Cs{v}$ from Definition~\ref{DefCaCsdefBis} which depends on $v$ and can be divided into two parts:
The portion (which we denote by $\Cs{v,\sqsubset}$) with real part bounded between $1/2$ and $R$ and imaginary part between $-d$ and $d$ for $d$ sufficiently small;
and the vertical portion (which we denote by $\Cs{v,\vert}$) with real part $R$ where $R=-\Re(v)+\eta$ and $\eta=\tfrac{1}{4}\max(a)+\tfrac{3}{4}\min(\alpha)$.

Let us recall that the integrand in \eqref{445}, through which $K_u^{\e}(v,v')$ is defined, is denoted by $h^q(s)$. We split the proof into two steps.
\emph{Step 1:}~We show that the integral of $h^q(s)$ over $s\in \Cs{v,\sqsubset}$ is bounded for all $\e<\e_0$ by an expression with exponential decay in $|v|$, uniformly over $v'$.
\emph{Step 2:}~We show that the integral of $h^q(s)$ over $s\in \Cs{v,\vert}$ is bounded for all $\e<\e_0$ by an expression with exponential decay in $|v|$, uniformly over $v'$. The combination of these two steps implies the inequality \eqref{kueineq} and hence completes the proof.

\subsubsection*{Step 1:}
We consider the various terms in $h^q(s)$ separately (in particular we consider the left-hand sides of \eqref{pwlimits1} through \eqref{pwlimits4})
and develop bounds for each which are valid uniformly for $\e<\e_0$ and $\e_0$ small enough.
Let us write $s=x+\I y$ and note that along the contour $\Cs{v,\sqsubset}$, $y\in [-d,d]$ for $d$ small, and $x\in [1/2,R]$.

Let us start with the left-hand side of \eqref{pwlimits4} which can be rewritten as
\begin{equation*}
\exp\left(\tau (-\e^{-1}s + \e^{-2} q^{-v}(q^{-s}-1))\right).
\end{equation*}
The norm of the above expression is bounded by the exponential of the real part of the exponent.
For $\varphi\in (0,\pi/4)$, one shows (as a perturbation of the analogous estimate in \emph{Step 1} of the Proof of Proposition~\ref{postasymlemma}) via Taylor expansion with remainder estimates that
\begin{equation*}
\tau \Re(-\e^{-1}s + \e^{-2} q^{-v}(q^{-s}-1))\leq \tilde c- \tau c|v| x
\end{equation*}
for some constants $c,\tilde c$.
The above bound implies
\begin{equation*}
\left|\exp\left(\tau(-\e^{-1}s + \e^{-2} q^{-v}(q^{-s}-1))\right)\right|\leq C e^{-\tau c|v|x}.
\end{equation*}

Let us now turn to the other terms in $h^q(s)$. We bound the left-hand side of \eqref{pwlimits1} as
\begin{equation*}
\left|e^{\tau s \e^{-1}}\left(\frac{-\zeta}{(1-q)^{M+N}}\right)^s\right| \leq C |u^s| \leq C e^{x\ln|u| - y\arg(u)}.
\end{equation*}
We may also bound the left-hand sides of \eqref{pwlimits2}, \eqref{pwlimits3} and \eqref{pwlimits3.5}, as well as the remaining product of Gamma functions by constants:
\begin{equation*}
\left|\frac{\Gamma_q(v+\ln_q(\tilde a_m^{-1}))}{\Gamma_q(s+v+\ln_q(\tilde a_m^{-1}))} \right|, \quad
\left|\frac{\Gamma_q(\ln_q(\tilde\alpha_m)-s-v)}{\Gamma_q(\ln_q(\tilde\alpha_m)-v)} \right|, \quad
\left|\frac{q^v \ln q}{q^{s+v} - q^{v'}}\right|, \quad
|\Gamma(-s)\Gamma(1+s)|\quad\leq \textrm{const}
\end{equation*}
for some constant $\textrm{const}>0$ (which may be different in each case).
The first two bounds come from the functional equation for the $q$-Gamma function (cf.\ Appendix~\ref{qSec}), and the last from the fact that $s$ is bounded away from $\Z$.

Combining these together shows that for $|v|$ large,
\begin{equation*}
\left|\int_{\Cs{v,\sqsubset}}h^q(s)\,\d s\right| \leq \int_{\Cs{v,\sqsubset}} C e^{-\tau c|v| \Re(s) + x\ln|u|-y \arg(u)} |\d s| \leq C' e^{-c'|v|}
\end{equation*}
for some constants $c',C'>0$, while for bounded $|v|$ the integral is just bounded as well.

\subsubsection*{Step 2:}
As above, we consider the various terms in $h^q(s)$ separately and develop bounds for each.
Let us write $s=R+\I y$ and note that $s\in \Cs{v,\vert}$ corresponds to $y$ varying over all $|y|\geq d$.
Four of the terms we consider (corresponding to the left-hand sides of \eqref{pwlimits2}, \eqref{pwlimits3}, \eqref{pwlimits3.5} and \eqref{pwlimits4})
are periodic functions in $y$ with fundamental domain $y\in [-\pi\e^{-1},\pi\e^{-1}]$.
We will first develop bounds on these four terms in this fundamental domain, and then turn to the non-periodic terms.

We start by controlling the behavior of the left-hand side of \eqref{pwlimits4} as $y$ varies in its fundamental domain.
For each $\varphi<\pi/4$ there exists a sufficiently small (yet positive) constant $c'$ such that as $y$ varies in its fundamental domain
\begin{equation*}
\tau \Re(-\e^{-1}s + \e^{-2} q^{-v}(q^{-s}-1)) \leq c' \tau \Re(vs+s^2/2).
\end{equation*}
On account of this, we can use the bound \eqref{step2bdd} from the proof of Proposition~\ref{postasymlemma}. This implies that
\begin{equation*}
\tau \Re(-\e^{-1}s + \e^{-2} q^{-v}(q^{-s}-1)) \leq c' \tau \left(-\frac{(y+\Im(v))^2}{2} -c |v|^2\right).
\end{equation*}

Let us now turn to the other $y$-periodic terms in $h^q(s)$. By bounds for the large imaginary part behavior of the $q$-Gamma function, we can show
\begin{equation*}
\left|\frac{\Gamma_q(v+\ln_q(\tilde a_m^{-1}))}{\Gamma_q(s+v+\ln_q(\tilde a_m^{-1}))} \right| \leq C e^{c\cdot{\rm dist}(\Im(s+v),2\pi \e^{-1} \Z)}
\end{equation*}
for some constants $c,C>0$.
Note that as opposed to \eqref{festatement} when $|v|$ was bounded, in the above inequality, we write ${\rm dist}(\Im(s+v),2\pi \e^{-1} \Z)$ in the exponential on the right-hand side.
This is because we are presently considering unbounded ranges for $v$.

One has similarly the bound
\begin{equation*}
\left|\frac{\Gamma_q(\ln_q(\tilde\alpha_m)-v-s)}{\Gamma_q(\ln_q(\tilde\alpha_m)-v)} \right| \leq C e^{-c' {\rm dist}(\Im(s+v),2\pi \e^{-1} \Z)}
\end{equation*}
for other positive constants $C$ and $c'$.

Also, we can bound
\begin{equation*}
\left|\frac{q^v \ln q}{q^{s+v} - q^{v'}}\right|\leq C
\end{equation*}
for some constant $C>0$.

The parts of $h^q(s)$ which are not periodic in $y$ can easily be bounded. We bound the left-hand side of \eqref{pwlimits1} as in \emph{Step 1} by
\begin{equation*}
\left|e^{-\tau s \e^{-1}}\left(\frac{-\zeta}{(1-q)^N}\right)^s\right| \leq C |u^s| \leq C e^{x\ln|u| - y\arg(u)}.
\end{equation*}

Finally, the term
\begin{equation*}
|\Gamma(-s)\Gamma(1+s)|\leq Ce^{-\pi |y|}
\end{equation*}
for some constant $C>0$.

We may now combine the estimates above.
The idea is to first prove that the integral on the fundamental domain $y\in [-\pi\e^{-1},\pi\e^{-1}]$ is exponentially small in $|v|$.
Then, by using the decay of the two non-periodic terms above, we can get a similar bound for the integral as $y$ varies over all of $\R$.
For $j\in \Z$, define the $j$ shifted fundamental domain as $D_j=j\e^{-1}2\pi + [-\e^{-1}\pi,\e^{-1}\pi]$. Let
\begin{equation*}
I_j:= \int_{D_j} h^q(R+\I y)\,\d y
\end{equation*}
and observe that combining all of the bounds developed above, we have that
\begin{equation*}
|I_j|\leq C \int_{-\e^{-1}\pi}^{\e^{-1}\pi} F_1(y) F_2(y)\,\d y
\end{equation*}
where
\begin{equation*}
\begin{aligned}
F_1(y) &= \exp\left(c' \tau \left(-\frac{(y+\Im(v))^2}{2} -c |v|^2\right) +c'' {\rm dist}(\Im(s+v),2\pi \e^{-1} \Z) +x\ln|u|\right),\\
F_2(y) &= \exp\left(- (y+j\e^{-1}2\pi)\arg(u) -\pi |y+j\e^{-1}2\pi| \right).
\end{aligned}
\end{equation*}
The term $F_1(y)$ is from the periodic bounds while $F_2(y)$ from the non-periodic terms (hence explaining the $j\e^{-1}2\pi$ shift in $y$).
By assumption on $u$, we have $-\arg(u)-\pi=\delta\leq c$ for some $\delta$. Therefore
$
F_2(y) \leq C e^{-c\e^{-1} |j|}
$
for some constants $c,C>0$. Thus
\begin{equation*}
|I_j|\leq C e^{-c\e^{-1} |j|} \int_{-\e^{-1}\pi}^{\e^{-1}\pi} F_1(y)\,\d y.
\end{equation*}
Just as in the end of \emph{Step 2} in the proof of Proposition~\ref{postasymlemma}, we can estimate the integral
\begin{equation*}
\int_{-\e^{-1}\pi}^{\e^{-1}\pi} F_1(y)\,\d y \leq \hat C e^{-\hat c|v|}
\end{equation*}
for some constants $\hat C,\hat c>0$. This implies
$|I_j|\leq \hat C C e^{-c\e^{-1} |j|} e^{-\hat c|v|}$.
Finally, observe that
\begin{equation*}
\left|\int_{\CsPre{v,\vert}} h^{q}(s)\,\d s\right| \leq \sum_{j\in \Z} |I_j| \leq \hat C C e^{-\hat c|v|} \sum_{j\in \Z}e^{-c\e^{-1} |j|} \leq C' e^{-\hat c|v|}
\end{equation*}
where $C'$ is independent of $\e$ as long as $\e<\e_0$ for some fixed $\e_0$. This is the bound desired to complete this step.

\section{SHE/KPZ equation with two-sided Brownian initial data -- Proof of Theorem~\ref{ThmFormulaContinuous}}\label{SectCDRP}
\subsection{Convergence of the Laplace transforms}

Recall the special semi-discrete directed random polymer considered in Definition~\ref{defnearstat} in which $M=1$ and $a_1=a$, $a_n\equiv 0$ for $n>1$, and $\alpha_1=\alpha>0$. We denoted by $\Zsd(\tau,N)$ the semi-discrete directed random polymer partition function in which the weight $\omega_{-1,1}$ is replaced by zero. We will now observe how, by scaling $\tau, N, a,\alpha$ accordingly, it is possible to show convergence of this partition function to the solution to the SHE with initial data related to the scalings of $a,\alpha$. Towards this end, let $T>0$ and $X\in \R$ represent the limiting time and space variables for the SHE and define the following $N$ dependent scalings
\begin{align}
\kappa&=\sqrt{\frac TN}+\frac XN,\label{defkappa}\\
\tau&=\kappa N=\sqrt{TN}+X\label{deftau}.
\end{align}

\begin{definition}\label{digammadef}
Let $\Psi(z) = \frac{d}{dz}\ln \Gamma(z)$ be the digamma function. For a given
$\theta\in \R_+$, define
\begin{equation*}
\kappa(\theta):=\Psi'(\theta),\quad f(\theta):=\theta\Psi'(\theta)-\Psi(\theta),\quad c(\theta):=(-\Psi''(\theta)/2)^{1/3}.
\end{equation*}
We may alternatively parameterize $\theta\in \R_+$ in terms of $\kappa \in\R_+$ as
\begin{equation*}
\theta_\kappa:=(\Psi')^{-1}(\kappa)\in\R_+,\quad f_\kappa:=\inf_{t>0} (\kappa t - \Psi(t))\equiv f(\theta_\kappa),\quad c_\kappa:=c(\theta_\kappa).
\end{equation*}
As given at the beginning of Section 6 in~\cite{BCF12}, the large $\theta$ asymptotics of $\kappa$ and $f$ are
\begin{align}
\kappa(\theta)&=\frac{1}{\theta}+\frac{1}{2 \theta^2}+\frac{1}{6\theta^3}+\Or(\theta^{-5}),\label{kappaseries}\\
f(\theta)&=1-\ln(\theta)+\frac{1}{\theta}+\frac{1}{4\theta^2}+\Or(\theta^{-4}).\label{fseries}
\end{align}
\end{definition}

\begin{theorem}[\cite{QMR12}]
\label{ThmDiscreteToContinuous}
Fix $T>0$, $X\in\R$ and real numbers $b<\beta$.
With Definition~\ref{digammadef}, let \mbox{$\vartheta=\theta_{\sqrt{T/N}}\simeq\sqrt{N/T}+\frac12$}.
Consider the semi-discrete directed random polymer in Definition~\ref{defnearstat} with partition function $\Zsd(\tau,N)$. Let the $a$ and $\alpha$ parameters of the polymer be defined as
\begin{equation}
a=\vartheta+b,\quad \alpha=\vartheta+\beta.\label{awt}
\end{equation}
Define the scaling factor
\begin{equation*}
C(N,T,X)=\exp\left(N+\frac12(N-1)\ln(T/N)+\frac12\left(\sqrt{TN}+X\right)+X\sqrt{N/T}\right).
\end{equation*}
Then, as $N$ goes to infinity,
\[\frac{\Zsd(\sqrt{TN}+X,N)}{C(N,T,X)}\Rightarrow\mathcal Z_{b,\beta}(T,X).\]
The convergence is in distribution and $\mathcal Z_{b,\beta}(T,X)$ is the solution to the SHE given in the statement of Theorem~\ref{ThmFormulaContinuous}.
\end{theorem}

Instead of $\vartheta$ in Theorem~\ref{ThmDiscreteToContinuous}, we choose our scaling parameter for the analysis and for \eqref{awt} to be
\begin{equation}\label{deftheta}
\theta=\theta_\kappa\simeq \sqrt{\frac NT}-\frac XT+\frac12
\end{equation}
which is the $\theta_\kappa$ given in Definition~\ref{digammadef} that corresponds to $\kappa$ given in \eqref{defkappa}.
We rewrite \eqref{awt} as
\begin{equation*}
a=\theta+b+X/T,\quad \alpha=\theta+\beta+X/T.
\end{equation*}

The scaling factor that appears in Theorem~\ref{ThmKbbeta} below is
\begin{equation}\label{defu}
u=Se^{-N-\frac12(N-1)\ln\frac TN-\frac12\sqrt{TN}-X\sqrt{\frac NT}+\frac T{24}-\frac X2+\frac{X^2}{2T}}.
\end{equation}
By comparing the exponents of $C(N,T,X)$ and $u$ and by Theorem~\ref{ThmDiscreteToContinuous},
\begin{equation}\label{uZconvdistr}
u\Zsd(\tau,N)\to Se^{\frac{X^2}{2T}+\frac T{24}}\mathcal Z_{b,\beta}(T,X)
\end{equation}
in distribution as $N\to\infty$ where $\mathcal Z_{b,\beta}(T,X)$ is the partition function of the continuous directed random polymer (CDRP) with boundary drift $b$ and $\beta$.
The convergence of the Fredholm determinant is the following.
\begin{theorem}\label{ThmKbbeta}
Fix $S$ with positive real part, $T>0$, $b<\beta$ real numbers and assume that $X=0$.
Set $\kappa$ and $\tau$ as in \eqref{defkappa} and \eqref{deftau}, $\sigma$ as in \eqref{defsigma} and $\theta$ as in \eqref{deftheta}.
Define the parameters
\begin{equation}
a=\theta+b,\quad \alpha=\theta+\beta,\label{a2nd}
\end{equation}
and use $u$ given in \eqref{defu}. Then
\begin{equation}\label{eq6.15}
\lim_{N\to\infty} \det(\Id+ K_{u})_{L^2(\Cv{a_+;\alpha;\pi/4})}=\det(\Id-K_{b,\beta})_{L^2(\R_+)}
\end{equation}
where $a_+=\max\{a,0\}$ and $K_{b,\beta}$ is defined in \eqref{defKbbeta}.
\end{theorem}
\begin{remark}
The Fredholm determinant in the left-hand side of \eqref{eq6.15}
is a special case of the one in Theorem~\ref{ThmFormulaSemiDiscrete} where we specialized to $a_1=a$, $a_2=a_3=\ldots=0$, $\alpha_1=\alpha$, $M=1$.
The condition $\varphi\in (0,\pi/4)$ in Theorem~\ref{ThmFormulaSemiDiscrete} is to ensure that the Fredholm series converges.
$\varphi=\pi/4$ is the borderline and depending on where the line crosses the real axis, the series might converge or not.
In Theorem~\ref{ThmKbbeta}, we use $\varphi=\pi/4$ and the crossing at the axis is chosen to be the critical point.
For this case, as one can see from the estimates in the proof (see e.g.\ Lemma~\ref{lemma:wtg}), the Fredholm series converges.
\end{remark}

In order to keep the notation simpler, we prove the theorem above for $X=0$; in the $X\neq0$ case, one can simply substitute $b$ by $b+\frac XT$.
The condition on the parameter $S$ comes from its appearance in the argument of the logarithm and as base of powers with complex exponent.
In order to avoid the different branches, we restrict it to the halfplane with positive real part.

In order to prove Theorem~\ref{ThmFormulaContinuous} and~\ref{ThmFormulaStationary}, we need some bounds on the modified Bessel function which are contained in the following lemma.
\begin{lemma}\label{LemmaModifiedBessel}
For $\nu>0$ and $x\in\R_+$, it holds:
\begin{enumerate}
\item[(a)] $x\mapsto x^\nu \BesselK_{-\nu}(x)$ is positive, continuous and decreasing in $x\in\R_+$,
\item[(b)] $0\leq x^\nu \BesselK_{-\nu}(x)\leq C(\nu)$ with $C(\nu)=2^\nu \Gamma(\nu)\sim \nu^{-1}$ as $\nu\to 0$,
\item[(c)] $0\leq -\frac{\d}{\d x}x^\nu \BesselK_{-\nu}(x) = x^\nu \BesselK_{1-\nu}(x) \leq C|1-\nu| x^\beta$ with $\beta=\max\{1,2\nu-1\}$,
\item[(d)] $\BesselK_{-\nu}(x)\simeq C e^{-x} x^{-1/2}$ as $x\to\infty$ where $C$ is independent of $\nu$.
\end{enumerate}
\end{lemma}
\begin{proof}
By the integral representation 9.6.24 of~\cite{AS84},
\begin{equation*}
\BesselK_{\nu}(x)=\int_{\R_+}\d t \cosh(\nu t) e^{-x \cosh(t)},
\end{equation*}
properties of (a) are trivially verified. To get (b), we bound the $\cosh$ by simple exponential and obtain
\begin{equation*}
x^\nu \BesselK_{-\nu}(x)\leq \int_{\R_+} \d t x^\nu e^{\nu t} e^{-x e^{t}/2}=2^\nu \Gamma(\nu,x/2)
\end{equation*}
where $\Gamma(a,z)$ is the incomplete Gamma function (where the last equality is obtained by the change of variable $\tau=xe^t/2$).
In particular, $\Gamma(\nu,0)=\Gamma(\nu)$. As $x^\nu \BesselK_{-\nu}(x)$ is monotone, (b) is shown.
The bound in (c) is obtained using (b), subdividing the cases $\nu\in (0,1]$ and $\nu>1$ taking into account that $\BesselK_\nu(x)=\BesselK_{-\nu}(x)$.
Finally, the bound (d) is formula 9.7.2 of~\cite{AS84}.
\end{proof}

\begin{proof}[Proof of Theorem~\ref{ThmFormulaContinuous}]
We start with \eqref{Besseltransform}. By Lemma~\ref{LemmaModifiedBessel}, $x^\nu \BesselK_{-\nu}(x)$ is a continuous and bounded function. Then, the convergence in distribution \eqref{uZconvdistr} implies
that the left-hand side of \eqref{Besseltransform} converges to that of \eqref{Bessel=Fredholm}.
The convergence of the right-hand side of \eqref{Besseltransform} to that of \eqref{Bessel=Fredholm} is exactly Theorem~\ref{ThmKbbeta} which is proved below.
\end{proof}

The rest of this section is devoted to the proof of Theorem~\ref{ThmKbbeta}.
\subsection{Formal critical point asymptotics}

By using \eqref{deftheta} and comparing \eqref{defu} with \eqref{fseries}, we have
\[u=\frac S\theta e^{-Nf_\kappa+\Or(N^{-1/2})},\]
that is, we can work with
\begin{equation}\label{usimple}
u=\frac S\theta e^{-Nf_\kappa}
\end{equation}
instead to get the same limit.

To rewrite the kernel $K_u$, first we apply the identity
$\Gamma(-s)\Gamma(1-s)=-\pi/\sin(\pi s)$. Then, we do a change of variable
$\tilde z=s+v$, to get
\begin{equation}\label{kvvztilde}
K_u(v,v')
=\frac{1}{2\pi \I}\int_{v+\Cs{v}}\!\! \d\tilde z \frac\pi{\sin(\pi(\tilde z-v))} \frac{\Gamma(v)^{N-1}}{\Gamma(\tilde z)^{N-1}}
\frac{\exp\left(-\frac12\tau v^2-v\ln u\right)}{\exp\left(-\frac12\tau \tilde z^2-\tilde z\ln u\right)}
\frac1{\tilde z-v'} \frac{\Gamma(v-a)}{\Gamma(\tilde z-a)} \frac{\Gamma(\alpha-\tilde z)}{\Gamma(\alpha-v)}.
\end{equation}

Let
\begin{equation}\label{defG}
G(z)=\ln\Gamma(z)-\kappa\frac{z^2}2+f_\kappa z.
\end{equation}
We are looking for the critical point of $G$, hence we are to solve the equation
\[G'(z)=\Psi(z)-\kappa z+f_\kappa=0.\]
It follows from Definition~\ref{digammadef} that $\theta_\kappa$ is a double critical point, i.e.\ $G'(\theta_\kappa)=G''(\theta_\kappa)=0$ and the Taylor series is
\[G(z)=G(\theta_\kappa)-\frac{(c_\kappa)^3}3\left(z-\theta_\kappa\right)+\Or\left(\left(z-\theta_\kappa\right)^4\right).\]
With the present choice of $\kappa$, we have $c_\kappa=\sigma^{-1}N^{-1/3}$, hence
\begin{equation}\label{Gtaylor}
NG(\theta+\sigma w)=NG(\theta)-\frac13 w^3+\Or\left(\frac{w^4}\theta\right).
\end{equation}

We can rewrite the kernel in \eqref{kvvztilde} using \eqref{usimple} and \eqref{defG} as
\[K_u(v,v') = -\frac{1}{2\pi \I}\int_{v+\Cs{v}}\d\tilde z \frac{\pi S^{\tilde z-v}}{\sin(\pi(\tilde z-v))} \frac{e^{NG(v)-NG(\tilde z)}}{\tilde z-v'}
\frac{\Gamma(v-a)}{\Gamma(\tilde z-a)} \frac{\Gamma(\alpha-\tilde z)}{\Gamma(\alpha-v)} \frac{\Gamma(\tilde z)}{\Gamma(v)}\frac{\theta^v}{\theta^{\tilde z}}.\]
We do the change of variables $v=\theta+\sigma w$, $v'=\theta+\sigma w'$ and $\tilde z=\theta+\sigma z$ and substitute \eqref{a2nd} to get
\begin{multline*}
K_\theta(w,w') = -\frac{1}{2\pi \I}\int_{\Cv{z}}\d z \frac{\sigma\pi S^{\sigma(z-w)}}{\sin(\sigma\pi(z-w))} \frac{e^{NG(\theta+\sigma w)-NG(\theta+\sigma z)}}{z-w'}\\
\times\frac{\Gamma(\sigma w-b)}{\Gamma(\sigma z-b)} \frac{\Gamma(\beta-\sigma z)}{\Gamma(\beta-\sigma w)}
\frac{\Gamma(\theta+\sigma z)}{\Gamma(\theta+\sigma w)} \frac{\theta^{\sigma w}}{\theta^{\sigma z}}.
\end{multline*}

As $\theta$ goes to infinity with $N$, for the last two factors,
\[\frac{\Gamma(\theta+\sigma z)}{\Gamma(\theta+\sigma w)} \frac{\theta^{\sigma w}}{\theta^{\sigma z}}\to1.\]
Along with the Taylor expansion in \eqref{Gtaylor}, we get that
\[K_\theta(w,w')\to-\wt K_{b,\beta}(w,w')\]
where
\begin{equation}\label{defwtKbbeta}
\wt K_{b,\beta}(w,w')=\frac1{2\pi \I}\int_{\Cv{z}}\d z \frac{\sigma\pi S^{\sigma(z-w)}}{\sin(\sigma\pi(z-w))} \frac{e^{z^3/3-w^3/3}}{z-w'}
\frac{\Gamma(\sigma w-b)}{\Gamma(\sigma z-b)}\frac{\Gamma(\beta-\sigma z)}{\Gamma(\beta-\sigma w)}.
\end{equation}
The Fredholm determinant of this kernel is rewritten in terms of a Fredholm determinant on $L^2(\R_+)$ in Lemma~\ref{lem:PFLemKPZreformuation}.

\subsection{Rigorous asymptotic analysis}
Let us first assume that
\begin{equation}\label{bbetaassumption}
b<-\frac12\quad\mbox{and}\quad\beta>\frac12.
\end{equation}
We will relax this assumption at the end of Section~\ref{ss:conv_bounds}.

We follow the lines of the proof of Theorem~3.3 in~\cite{BCF12}.
We have to determine \mbox{$\lim_{N\to\infty} \det(\Id+K_u)_{L^2(\mathcal{C}_v)}$} where the contour $\mathcal C_v$ is defined below and it is different from the contour given in Theorem~\ref{ThmFormulaContinuous}.
This change of notation only applies for this section, so it will not cause difficulties.
The contour $\mathcal{C}_v$ is chosen as
\[\mathcal{C}_v=\{\theta-1/4+\I r, |r|\leq r^*\}\cup\{\theta e^{\I t}, t^*\leq |t| \leq \pi/2\}\cup\{\theta-|y|+\I y,|y|\geq\theta\}\]
where
\begin{equation}\label{defrt*}
r^*=\sqrt{\frac\theta2-\frac1{16}},\qquad t^*=\arcsin\left(\sqrt{\frac1{2\theta}-\frac1{16\theta^2}}\right).
\end{equation}
The contour $\mathcal{C}_{\tilde z}$ is set as
\begin{equation}\label{eq5.11}
\mathcal{C}_{\tilde z}=\{\theta+p/4+\I \tilde y, \tilde y\in\R\}\cup\bigcup_{k=1}^{\ell} B_{v+k}
\end{equation}
where $B_{z}$ is a small circle around $z$ clockwise oriented and $p\in\{1,2\}$ depending on the value of $v$, see Figure~\ref{PFFigPathsKPZ}.
More precisely, for given $v$, we consider the sequence of points $S=\{\Re(v)+1,\Re(v)+2,\ldots\}$ and we choose $p=p(v)$ and $\ell=\ell(v)$ as follows:
\begin{itemize}
\item[(1)] If the sequence $S$ does not contain points in $[\theta,\theta+1/2]$, then let $\ell\in \N_0$ be such that $\Re(v)+\ell \in [\theta-1,\theta]$ and we set $p=1$.
\item[(2)] If the sequence $S$ contains a point in $[\theta,\theta+3/8]$, then let $\ell\in \N$ such that \mbox{$\Re(v)+\ell \in [\theta,\theta+3/8]$} and set $p=2$.
\item[(3)] If the sequence $S$ contains a point in $[\theta+3/8,\theta+1/2]$, then let $\ell\in \N$ such that $\Re(v)+\ell \in [\theta-5/8,\theta-1/2]$ and set $p=1$.
\end{itemize}
With this choice, the singularity of the sine along the line $\theta+p/4+\I\R$ is not present, since the poles are at a distance at least $1/8$ from it.
Also, the leading contribution of the kernel will come from situation (1) with $\ell=0$ and $p=1$.

This choice of the contours is identical to the one made in the unperturbed case in~\cite{BCF12}.
If condition \eqref{bbetaassumption} holds, these contours can be used, since the extra singularities coming from
$\Gamma(v-a)$ are on the left-hand side of $\Cv{v}$ and the poles coming from $\Gamma(\alpha-\tilde z)$ are on the right-hand side of $\Cv{\tilde z}$.
Otherwise the contours have to be locally modified. This is made precise later.

\begin{figure}
\begin{center}
\psfrag{Phi1}[cb]{$\Phi^{-1}$}
\psfrag{0}[cb]{$0$}
\psfrag{theta}[cb]{$\theta$}
\psfrag{-thetat}[ct]{$-\frac{\theta}{\sigma}$}
\psfrag{+p4}[lt]{$\frac{p}{4\sigma}$}
\psfrag{-14}[ct]{$\frac{1}{4\sigma}$}
\psfrag{theta+p4}[lt]{$\theta+\frac p4$}
\psfrag{theta-14}[rt]{$\theta-\frac14$}
\psfrag{p}[lb]{$p$}
\psfrag{w}[cb]{$w$}
\psfrag{Cw}[lb]{$\mathcal{C}_w$}
\psfrag{v}[cb]{$v$}
\psfrag{Cv}[lb]{$\mathcal{C}_v$}
\psfrag{Cztilde}[lb]{$\mathcal{C}_{\tilde z}$}
\psfrag{Cz}[lb]{$\mathcal{C}_z$}
\psfrag{a}[cb]{$a$}
\psfrag{b}[cb]{$\frac b\sigma$}
\psfrag{alpha}[cb]{$\alpha$}
\psfrag{beta}[cb]{$\frac\beta\sigma$}
\includegraphics[height=7cm]{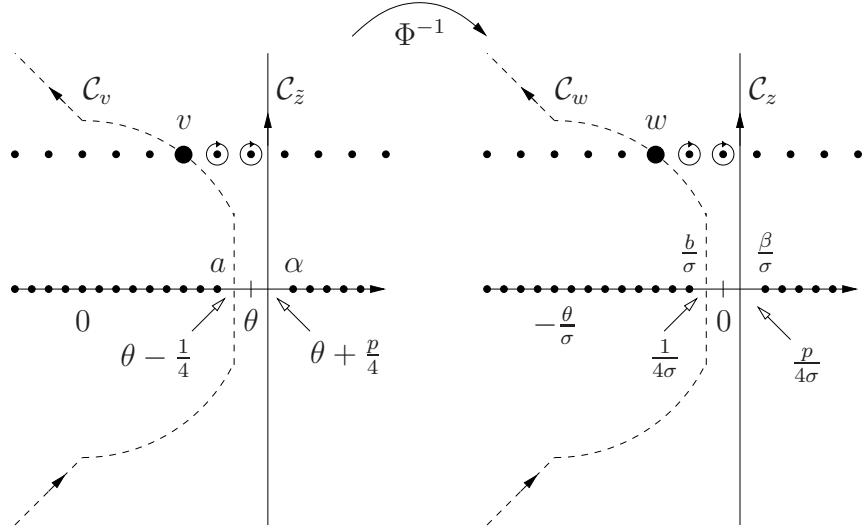}
\caption{Left: Integration contours $\mathcal{C}_v$ (dashed) and $\mathcal{C}_{\tilde z}$ (the solid line plus circles at $v+1,\ldots,v+\ell$)
where the small black dots are poles either of the sine or of the gamma functions.
Right: Integration contours after the change of variables $\mathcal{C}_w$ (dashed) and $\mathcal{C}_z$ (the solid line plus circles at $w+1,\ldots,w+\ell$), with $p=p(w)\in\{1,2\}$.}
\label{PFFigPathsKPZ}
\end{center}
\end{figure}

With $\sigma$ defined in \eqref{defsigma}, we do the change of variables
\begin{equation*}
\{v,v',\tilde z\}=\{\Phi(w),\Phi(w'),\Phi(z)\}\quad\textrm{with}\quad\Phi(z):=\theta+z\sigma
\end{equation*}
and
\begin{multline}\label{defKtheta}
K_\theta(w,w'):=\sigma K_u(\Phi(w),\Phi(w'))=\\
-\frac{1}{2\pi \I}\int_{\Cv{z}}\d z \frac{\sigma\pi S^{\sigma(z-w)}}{\sin(\sigma\pi(z-w))}
\frac{e^{NG(\theta+\sigma w)-NG(\theta+\sigma z)}}{z-w'}
\times\frac{\Gamma(\sigma w-b)}{\Gamma(\sigma z-b)}\frac{\Gamma(\beta-\sigma z)}{\Gamma(\beta-\sigma w)}
\frac{\Gamma(\theta+\sigma z)}{\Gamma(\theta+\sigma w)} \frac{\theta^{\sigma w}}{\theta^{\sigma z}}.
\end{multline}
After this change of variables, the contours $\mathcal{C}_w=\Phi^{-1}(\mathcal{C}_v)$ and $\mathcal{C}_z=\Phi^{-1}(\mathcal{C}_{\tilde z})$ are given by
\begin{equation}\label{PFeqCwKPZ}
\mathcal{C}_w=\{-1/(4\sigma)+\I r/\sigma, |r|\leq r^*\}\cup\{(e^{\I t}-1)\theta/\sigma, t^*\leq |t|\leq \pi/2\}\cup\{-|y|+\I y,|y|\geq \theta/\sigma\}
\end{equation}
and
\begin{equation*}
\mathcal{C}_z=\{p/(4\sigma)+\I y, y\in\R\}\cup \bigcup_{k=1}^{\ell} B_{w+k/\sigma}
\end{equation*}
with $r^*$ and $t^*$ defined in \eqref{defrt*}, and $B_{z}$ is a small circle around $z$ clockwise oriented. Then we have
\begin{equation*}
\det(\Id+K_u)_{L^2(\mathcal{C}_v)} = \det(\Id+K_\theta)_{L^2(\mathcal{C}_w)}.
\end{equation*}
Thus, we need to prove that
\begin{equation*}
\lim_{N\to\infty}\det(\Id+K_\theta)_{L^2(\mathcal{C}_w)}=\det(\Id-K_{b,\beta})_{L^2(\R_+)}
\end{equation*}
with $K_{b,\beta}$ given in \eqref{defKbbeta}. The convergence of the kernel follows by Proposition~\ref{PFPropPtConvKPZ} and the exponential bound by Proposition~\ref{PFPropBoundKPZ}.
We then obtain
\begin{equation}\label{KthetatoK}
\lim_{N\to\infty} \det(\Id+K_\theta)_{L^2(\mathcal{C}_w)} = \det(\Id-\wt K_{b,\beta})_{L^2(\mathcal{C}_w)}
\end{equation}
with $\wt K_{b,\beta}$ given in \eqref{defwtKbbeta}.
Note that by definition \eqref{PFeqCwKPZ}, the contour $\Cv{w}$ itself depends on $\theta$.
With a slight abuse of notation, we will denote by $\Cv{w}$ also the contour on the right-hand side of \eqref{KthetatoK} that appears in the $\theta\to\infty$ limit,
which is $-\frac14+\I\R$ with a possible local perturbation close to $0$ that will be given later.
Lemma~\ref{lem:PFLemKPZreformuation} shows that the limiting Fredholm determinant is equal to $\det(\Id-K_{b,\beta})_{L^2(\R_+)}$
and thus completes the proof of Theorem~\ref{ThmKbbeta} for the case of \eqref{bbetaassumption}.

\subsection{Pointwise convergence and bounds}\label{ss:conv_bounds}

Proposition~\ref{PFPropPtConvKPZ} and~\ref{PFPropBoundKPZ} in this section are the analogues of Propositions~6.1 and~6.2 in~\cite{BCF12}.
For the sake of completeness, we give the proof of them putting emphasis on the new factors that appear in the kernel.
These are the following gamma ratios in the definition \eqref{defKtheta} of $K_\theta$:
\[\frac{\Gamma(\sigma w-b)}{\Gamma(\sigma z-b)}
\frac{\Gamma(\beta-\sigma z)}{\Gamma(\beta-\sigma w)}
\frac{\Gamma(\theta+\sigma z)}{\Gamma(\theta+\sigma w)} \frac{\theta^{\sigma w}}{\theta^{\sigma z}}.\]

As in~\cite{BCF12}, the scale of the steep descent analysis is $N\theta=\Or(N^{3/2})$.
The main contribution of the Fredholm determinant $\det(\Id+K_\theta)_{L^2(\Cv{w})}$ comes from the regime
when the variables $v,v'$ and $\tilde z$ are in the neighbourhood of $\theta$, i.e.\ $w,w'$ and $z$ are in the neighbourhood of $0$.
The function that gives the leading contribution to the integral in the steep descent analysis is
\[\wt G(z)=\frac{G(\theta+\theta z)}\theta.\]
It has a double critical point at $0$, and for further derivatives, it holds
\begin{equation}\label{wtG'}
\begin{aligned}
\wt G^{(3)}(0)&=-1+\Or(\theta^{-1}),\\
\wt G^{(n)}(0)&=\Or(1),\qquad n\ge4.
\end{aligned}\end{equation}
We will denote the real part of $\wt G$ by
\begin{equation}\label{defwtg}
\wt g(x,y):=\Re(\wt G(x+\I y)).
\end{equation}
The statements of the following lemma are completely taken from Lemma~6.3, 6.4 and~6.5 of~\cite{BCF12}.
\begin{lemma}\label{lemma:wtg}\mbox{}
\begin{enumerate}
\item For any fixed $X\ge0$, the function $Y\mapsto\wt g(X,Y)$ is strictly increasing for $Y>0$ with $\partial_Y\wt g(X,Y)\ge\partial_Y\wt g(0,Y)$.
\item For $X\ge0$,
\[\wt g(X,Y)\ge\wt g(X,0)+Y^4/12+\Or(Y^4/\theta,Y^6).\]
\item The function $t\mapsto\wt g\big(\cos (t)-1,\sin (t)\big)$ is strictly decreasing for $t\in(0,\pi/2]$. For \mbox{$t\in[0,\pi/2]$} and $\theta$ large enough,
\[\wt g\big(\cos(t)-1,\sin(t)\big)- \wt g(0,0)\le - \sin(t)^4/16.\]
\item The function $Y\mapsto\wt g(-Y,Y)$ is strictly decreasing for $Y>0$. For $Y\to\infty$, we have
\[\partial_Y\wt g(-Y,Y)\simeq-\ln Y.\]
\end{enumerate}
\end{lemma}
Notational remark: $\Or(Y^4/\theta,Y^6)$ in Lemma~\ref{lemma:wtg} is the error term coming from Taylor expansion around $Y=0$.

We will also use the following properties of the gamma function.

\begin{lemma}\label{gammalemma}$ $
\begin{enumerate}
\item For any $u,v\in\R$,
\begin{equation}\label{gammalemma1}
\left|\frac{\Gamma(u+\I v)}{\Gamma(u)}\right|^2=\frac{\Gamma(u+\I v)\Gamma(u-\I v)}{\Gamma(u)^2} =\prod_{n=0}^\infty\left(1+\frac{v^2}{(u+n)^2}\right)^{-1}.
\end{equation}

\item For any $u,v,w\in\R$,
\begin{equation}\label{gammalemma2}
\left|\frac{\Gamma(u+\I w)}{\Gamma(v\pm\I w)}\right|\simeq|w|^{u-v}\quad\mbox{as}\quad |w|\to\infty
\end{equation}
where $\simeq$ means that the ratio of the two sides converges to $1$.
\end{enumerate}
\end{lemma}

\begin{proof}
Part (1) is Formula~6.1.25 in~\cite{AS84}.
To get (2), we use Formula~6.1.45 in~\cite{AS84}
\[\lim_{|y|\to\infty}(2\pi)^{-1/2}|\Gamma(x+\I y)|e^{|y|\pi/2}|y|^{1/2-x}=1\]
for $x$ and $y$ real. \eqref{gammalemma2} is a straightforward consequence.
\end{proof}

\begin{proposition}\label{PFPropPtConvKPZ}
Uniformly for $w,w'$ in a bounded set of $\mathcal{C}_w$,
\begin{equation*}
\lim_{N\to\infty} K_\theta(w,w') = - \wt K_{b,\beta}(w,w')
\end{equation*}
where $\wt K_{b,\beta}$ is given by \eqref{defwtKbbeta}.
\end{proposition}

\begin{proof}
The dependence on $N$ of the kernel $K_\theta$ in \eqref{defKtheta} appears in the factors
\begin{equation}\label{expNG}\begin{aligned}
e^{NG(\theta+\sigma w)-NG(\theta+\sigma z)}\frac{\Gamma(\theta+\sigma z)}{\Gamma(\theta+\sigma w)}\frac{\theta^{\sigma w}}{\theta^{\sigma z}}&=\frac{e^{(N-1)G(\theta+\sigma w)-\frac\kappa2(\theta+\sigma w)^2+f_\kappa(\theta+\sigma w)+\sigma w\ln\theta}}{e^{(N-1)G(\theta+\sigma z)
-\frac\kappa2(\theta+\sigma z)^2+f_\kappa(\theta+\sigma z)+\sigma z\ln\theta}}\\
&=\frac{e^{(N-1)\theta \wt G(\frac{w\sigma}\theta)-\frac{\kappa\theta^2}2\left(1+\frac{w\sigma}\theta\right)^2+f_\kappa\theta\left(1+\frac{w\sigma}\theta\right)
+\theta \frac{w\sigma}\theta\ln\theta}}{e^{(N-1)\theta \wt G(\frac{z\sigma}\theta)-\frac{\kappa\theta^2}2\left(1+\frac{z\sigma}\theta\right)^2+f_\kappa\theta\left(1+\frac{z\sigma}\theta\right)
+\theta \frac{z\sigma}\theta\ln\theta}}.
\end{aligned}\end{equation}
One can already see that the scale of the steep descent analysis is $N^{3/2}$.
By \eqref{kappaseries} and \eqref{fseries}, we have $\kappa\theta^2/2=\Or(\theta)$ and $f_\kappa\theta=\Or(\theta)$,
which shows that we have to investigate the real part of $\wt G$ along the contour $\Cv{z}$.

For $N$ large enough and for $w$ in a fixed bounded subset of $\Cv{w}$, $\Re(w\sigma+1)>1/2$ and $\Re((z-w)\sigma)\in (0,1)$ so that we have $\ell=0$ and $p=1$,
i.e.\ in this case $\Cv{z}=\{\frac{1}{4\sigma}+\I y,y\in\R\}$. Taylor expansion around $w=0$ give us
\begin{equation}\label{NGTaylor}
\begin{aligned}
N G(\theta+\sigma w)=N\theta\wt G\left(\frac{w\sigma}\theta\right)&=N\theta\wt G(0)+\frac{N\theta}{6}\wt G^{(3)}(0)\left(\frac{\sigma w}\theta\right)^3+\Or(N\theta w^4/\theta^4)\\
&=N\theta\wt G(0)-\frac{N\theta \sigma^3}{2\theta^3}\frac{w^3}3+\Or(w^4/\theta,N\theta w^3/\theta^4)\\
&=N\theta\wt G(0)-\frac{w^3}{3}+\Or(w^4/\theta)
\end{aligned}
\end{equation}
where we used \eqref{wtG'}.

We divide the integral over $z$ into two parts: (a) $|\Im(z)|>\theta^{1/3}$ and (b) $|\Im(z)|\leq \theta^{1/3}$.

\smallskip
(a) \emph{Contribution of the integration over $|\Im(z)|>\theta^{1/3}$.}
We will show that the integral can be bounded as
\begin{equation}\label{largezintegral}
\int_{|z|>\theta^{1/3}}\d ze^{N\theta\left(\wt g(0,0)-\wt g\left(\frac1{4\sigma\theta},\frac{\Im(z)}\theta\right)\right)+\Or(1/\theta)}
=\Or(\theta)\int_{\theta^{-2/3}}^\infty\d y e^{N\theta\left(\wt g(0,0)-\wt g\left(\frac1{4\sigma\theta},\frac y\theta\right)\right)+\Or(1/\theta)}.
\end{equation}
This can be seen as follows. We have to work with the $z$-dependent part of the left-hand side of \eqref{expNG}. Therefore,
\[e^{NG(\theta)-NG(\theta+\sigma z)}=e^{N\theta\left(\wt G(0)-\wt G\left(\frac{z\sigma}\theta\right)\right)}
=e^{N\theta\left(\wt g(0,0)-\wt g\left(\frac1{4\sigma\theta},\frac{\Im(z)}\theta\right)\right)}\]
by the definition \eqref{defwtg}. Then
\[\frac{\Gamma(\theta+\sigma z)}{\Gamma(\theta+\sigma w)}\frac{\theta^{\sigma w}}{\theta^{\sigma z}}
=\frac{\Gamma(\theta+\sigma z)}{\Gamma(\theta+1/4)}\frac{\theta^{1/4}}{\theta^{\sigma z}}
\frac{\Gamma(\theta+1/4)}{\Gamma(\theta+\sigma w)}\frac{\theta^{\sigma w}}{\theta^{1/4}}.\]
By \eqref{gammalemma1},
\[\left|\frac{\Gamma(\theta+\sigma z)}{\Gamma(\theta+p/4)}\frac{\theta^{p/4}}{\theta^{\sigma z}}\right|\le1,\]
and from \eqref{gammalemma2},
\[\frac{\Gamma(\theta+1/4)}{\Gamma(\theta+\sigma w)}\frac{\theta^{\sigma w}}{\theta^{1/4}}=1+\Or\left(\frac{(1/4-w)^2}\theta\right)\]
as $\theta\to\infty$ which can be controlled by $\Or(1/\theta)$ in the exponent in \eqref{largezintegral}.

The remaining $z$-dependent factor $\Gamma(\beta-\sigma z)/\Gamma(\sigma z-b)$ is only polynomial in $\Im(z)$ along $\Cv{z}$ by \eqref{gammalemma2}.
Hence using the first part of Lemma~\ref{lemma:wtg} about the decay of $y\mapsto\wt g(1/(4\sigma\theta),y)$,
we see that the integral in \eqref{largezintegral} can be bounded by
\begin{equation}\label{largezbound}
e^{N\theta\left(\wt g(0,0)-\wt g\left(\frac1{4\sigma\theta},\theta^{-2/3}\right)\right)+\Or(\theta^{-1})}\leq e^{N\theta\left(\wt g(0,0)-\wt g\left(\frac1{4\sigma\theta},0\right)-\frac{\theta^{-8/3}}{12}+\Or(\theta^{-11/3})\right)+\Or(\theta^{-1})}.
\end{equation}
Exploiting the relation $\wt g(1/(4\sigma\theta),0)=\wt g(0,0)+\Or(\theta^{-3})$, we get the exponential decay
\[\eqref{largezbound}\le\Or(1)\exp(-cN^{1/6})\]
with some $c>0$.

\smallskip
(b) \emph{Contribution of the integration over $|\Im(z)|\leq \theta^{1/3}$.}
As in~\cite{BCF12}, one can see that in the expansion
\[-N\theta\wt G\left(\frac{z\sigma}\theta\right)=-N\theta\wt G(0)+\frac{z^3}3+\Or(z^4/\theta)\]
for $z=1/(4\sigma)+\I y$, the real part
\[\Re\left(\frac{z^3}{3}\right)=-\frac{y^2}{4\sigma^2}+\frac{1}{192\sigma^3}\]
dominates the error term $\Or(z^4/\theta)$ for large $\theta$.

\smallskip
(b.1) $\theta^{1/6}\leq |\Im(z)|\leq \theta^{1/3}$.
From the previous observation,
\begin{multline*}
\left|\frac{1}{2\pi \I}\int_{\theta^{1/6}\le|\Im(z)|\le\theta^{1/3}}\d z
\frac{\sigma\pi S^{\sigma(z-w)}}{\sin(\sigma\pi(z-w))} \frac{e^{z^3/3-w^3/3+\Or(w^4/\theta,z^4/\theta,1/\theta)}}{z-w'}
\frac{\Gamma(\sigma w-b)}{\Gamma(\sigma z-b)}\frac{\Gamma(\beta-\sigma z)}{\Gamma(\beta-\sigma w)}\right|\\
\le\Or\left(e^{-c\theta^{1/3}}\right)=\Or\left(e^{-c'N^{1/6}}\right)
\end{multline*}
for some $c,c'>0$.

\smallskip
(b.2) $|\Im(z)|\le\theta^{1/6}$.
This contribution of the integral is
\begin{equation}\label{verysmallzintegral}
-\frac{1}{2\pi \I}\int_{\frac1{4\sigma}+\I y,|y|\le\theta^{1/6}}\d z \frac{\sigma\pi S^{\sigma(z-w)}}{\sin(\sigma\pi(z-w))}
\frac{e^{z^3/3-w^3/3+\Or(w^4/\theta,z^4/\theta)}}{z-w'}\frac{\Gamma(\sigma w-b)}{\Gamma(\sigma z-b)}\frac{\Gamma(\beta-\sigma z)}{\Gamma(\beta-\sigma w)}.
\end{equation}
For $|y|\leq \theta^{1/6}$, $\Or(z^4/\theta)=\Or(\theta^{-1/3})$.
Using $|e^x-1|\leq |x| e^{|x|}$ for $x=\Or(z^4/\theta)$ and then for $x=\Or(w^4/\theta)$,
we can delete the error term by making an error of order $\Or(\theta^{-1/3})=\Or(N^{-1/6})$. Thus,
\[\eqref{verysmallzintegral}=\Or(N^{-1/6})-\frac{1}{2\pi \I}\int_{\frac{1}{4\sigma}+\I y,|y|\leq\theta^{1/6}}\d z \frac{\sigma \pi S^{\sigma(z-w)}}{\sin(\sigma\pi(z-w))}
\frac{e^{z^3/3-w^3/3}}{z-w'}\frac{\Gamma(\sigma w-b)}{\Gamma(\sigma z-b)}\frac{\Gamma(\beta-\sigma z)}{\Gamma(\beta-\sigma w)}.\]
Finally, extending the last integral to $\frac{1}{4\sigma}+\I\R$, we make an error of order $\Or(e^{-c\theta^{1/3}})$ for some constant $c>0$.

Putting all the above estimates together we obtain that, for $w,w'\in \mathcal{C}_w$ in a bounded set around $0$,
\[K_\theta(w,w')=\Or(N^{-1/6})-\frac{1}{2\pi \I}\int_{\frac{1}{4\sigma}+\I \R}\d z \frac{\sigma \pi S^{\sigma(z-w)}}{\sin(\sigma\pi(z-w)} \frac{e^{z^3/3-w^3/3}}{z-w'}
\frac{\Gamma(\sigma w-b)}{\Gamma(\sigma z-b)}\frac{\Gamma(\beta-\sigma z)}{\Gamma(\beta-\sigma w)}\]
which completes the proof.
\end{proof}

\begin{proposition}\label{PFPropBoundKPZ}
For any $w,w'$ in $\mathcal{C}_w$, uniformly for all $N$ large enough,
\begin{equation*}
|K_\theta(w,w')|\leq C e^{-|\Im(w)|}
\end{equation*}
for some positive constant $C$.
\end{proposition}

\begin{proof}
We follow the lines of the proof of Proposition~6.2 in~\cite{BCF12}. First, we
can rewrite the kernel as
\begin{multline}\label{Ktheta_separated}
K_\theta(w,w')=S^{-\sigma w}e^{NG(\theta+\sigma w)-NG(\theta)}\frac{\Gamma(\sigma w-b)}{\Gamma(\beta-\sigma w)}
\frac{\Gamma(\theta+p/4)}{\Gamma(\theta+\sigma w)} \frac{\theta^{\sigma w}}{\theta^{p/4}}\\
\times\frac{-1}{2\pi \I}\int_{\Cs{v}}\d z \frac{\sigma\pi S^{\sigma z}}{\sin(\sigma\pi(z-w))} \frac{e^{NG(\theta)-NG(\theta+\sigma z)}}{z-w'}
\frac{\Gamma(\sigma z-b)}{\Gamma(\beta-\sigma z)} \frac{\Gamma(\theta+\sigma z)}{\Gamma(\theta+p/4)}\frac{\theta^{p/4}}{\theta^{\sigma z}}
\end{multline}
where we have separated the dependence on $w$ and $z$.

The dependence on $w'$ is marginal because we can choose the integration variable $z$ such that $|z-w'|\geq 1/(4\sigma)$ and
because we will get the bound through evaluating the absolute value of the integrand of \eqref{Ktheta_separated}.

\smallskip
\emph{Case 1: $w\in \{-1/(4\sigma)+\I y, |y|\leq r^*/\sigma\}$ with $r^*$ given in \eqref{defrt*}.}
The integration contour for $z$ is $1/(4\sigma)+\I\R$ ($p=1$) and no extra contributions from poles of the sine are present.
The factor $1/\sin(\sigma\pi(z-w))$ is uniformly bounded from above. By taking $z=\frac1{4\sigma}+\I\frac{Y\theta}\sigma$, we get
\begin{equation}\label{smallwestimate}\begin{aligned}
|K_\theta(w,w')|&\le\Or(1)\left|e^{N\theta\left(\wt G\left(\frac{w\sigma}\theta\right)-\wt G(0)\right)}
\frac{\Gamma(\theta+1/4)}{\Gamma(\theta+\sigma w)}\frac{\theta^{\sigma w}}{\theta^{1/4}}\right|
\int_{\R}\d Y \frac{e^{N\theta \left(\widetilde g(0,0)-\widetilde g(\tilde \e,Y)\right)} \theta}{(1+|\Im(w)|)^{(b+\beta)+\frac12}}\\
&\le\Or(1)\left|e^{N\theta\left(\wt G\left(\frac{w\sigma}\theta\right)-\wt G(0)\right)}
\frac{\Gamma(\theta+1/4)}{\Gamma(\theta+\sigma w)}\frac{\theta^{\sigma w}}{\theta^{1/4}}\right|(1+|\Im(w)|)^{-(b+\beta)-\frac12}
\end{aligned}\end{equation}
where $\tilde\e=1/(4\sigma\theta)$. The integral over $Y$ is finite by Proposition~\ref{PFPropPtConvKPZ}.
The last factor above $(1+|\Im(w)|)^{-(b+\beta)-1/2}$ is due to
\[\left|\frac{\Gamma(\sigma w-b)}{\Gamma(\beta-\sigma w)}\right|\simeq(\sigma\Im(w))^{-(b+\beta)-\frac12}\]
as $|\Im(w)|\to\infty$ which follows from \eqref{gammalemma2}.
In order to avoid the possible divergence of this bound around $\Im(w)=0$, we use $(1+|\Im(w)|)$ instead of $|\Im(w)|$ in \eqref{smallwestimate}.
The factor $(1+|\Im(w)|)^{-(b+\beta)-1/2}$ will be negligible since we prove exponential decay in $|\Im(w)|$.

We rewrite the estimate \eqref{smallwestimate} as in \eqref{expNG}:
\begin{equation}\label{exp(N-1)G}
\left|e^{N\theta\left(\wt G\left(\frac{w\sigma}\theta\right)-\wt G(0)\right)}
\frac{\Gamma(\theta+1/4)}{\Gamma(\theta+\sigma w)}\frac{\theta^{\sigma w}}{\theta^{1/4}}\right|=\frac{e^{(N-1)\theta\Re \wt G(\frac{w\sigma}\theta)-\frac{\kappa\theta^2}2\left(1+\frac{w\sigma}\theta\right)^2+f_\kappa\theta\left(1+\frac{w\sigma}\theta\right)}}
{e^{(N-1)\theta \wt G(0) -\frac{\kappa\theta^2}2+f_\kappa\theta}}
\left|\frac{\Gamma(\theta+1/4)}{\Gamma(\theta)}\theta^{-1/4}\right|.
\end{equation}
Since $|w\sigma/\theta|=\Or(\theta^{-1/2})$, we use the Taylor expansion of \eqref{NGTaylor} to get that
\[(N-1)\theta\wt G\left(\frac{w\sigma}\theta\right)=(N-1)\theta\wt G(0)-\frac{w^3}3+\Or(w^4/\theta).\]
Substituting $w=-1/(4\sigma)+\I y$ and taking real part, we get
\[(N-1)\theta\Re\left(\wt G\left(\frac{w\sigma}\theta\right)-\wt G(0)\right)=-\frac1{4\sigma} y^2+\Or(1)+\Or(y^4/\theta).\]
For $|y|\le r^*/\sigma=\Or(\theta^{1/2})$, the error term $\Or(y^4/\theta)$ is dominated by the $y^2$ term for $\theta$ large enough. Hence we can write
\[(N-1)\theta\Re\left(\wt G\left(\frac{w\sigma}\theta\right)-\wt G(0)\right)\le-\frac1{8\sigma}y^2+\Or(1).\]

For the rest of the terms in the exponent of \eqref{exp(N-1)G} after substituting $w=-1/(4\sigma)+\I y$, we find that
\begin{multline*}
\Re\left(-\frac{\kappa\theta^2}2\left(1+\frac{w\sigma}\theta\right)^2+\frac{\kappa\theta^2}2+f_\kappa\theta\left(1+\frac{w\sigma}\theta\right)-f_\kappa\theta\right)\\
=\frac{\kappa\theta^2}2\left(-\left(1-\frac1{4\theta}\right)^2+1+\frac{y^2}{\theta^2}\right)-\frac{f_\kappa}4=\frac\kappa2y^2+\Or(1)=\Or(y^2/\theta,1).
\end{multline*}
Putting these bounds together yields
\[|K_\theta(w,w')|\le\Or(1) e^{-\frac{1}{8\sigma} |\Im(w)|^2}\leq C e^{-|\Im(w)|}.\]

\smallskip
\emph{Case 2: $w\in\{(e^{\I t}-1)\theta/\sigma, t^*\leq |t|\leq \pi/2\}\cup\{-|y|+\I y,|y|\geq \theta/\sigma\}$}.
We divide the estimation of the bound by separating into the contributions from
(a) integration over $\frac{p}{4\sigma}+\I\R$ with $p\in\{1,2\}$ depending on $w$ (see the definitions after \eqref{eq5.11}) and
(b) integration over the circles $B_{w+k/\sigma}$, $k=1,\ldots,\ell$.

\smallskip
\emph{Case 2(a).} First notice that the estimate \eqref{smallwestimate} of \emph{Case 1} still holds with the minor difference that
$\tilde\e=p/(4\theta)$ where $p\in\{1,2\}$ depending on the value of $w$, so that we only need to estimate the exponent.

For $w\in\{(e^{\I t}-1)\theta/\sigma, t^*\leq |t|\leq \pi/2\}$, the third part of Lemma~\ref{lemma:wtg} shows that $\widetilde g(\cos(t)-1,\sin(t))-\widetilde g(0,0)\leq -\sin(t)^4/16$.
Replacing $\Im(w)=\sin(t) \theta/\sigma$ and using $|\Im(w)|\geq \sqrt{\theta/2-1/16}$ we obtain
\[(N-1)\theta\Re\left(\wt G\left(\frac{w\sigma}\theta\right)-\wt G(0)\right)\le-c_1|\Im(w)|^4/\theta\le-c_2|\Im(w)|^2\]
if $\theta$ is large enough and for $c_1,c_2>0$.

Then we take $w=(e^{\I t}-1)\theta/\sigma$ in \eqref{exp(N-1)G} for the other terms of the exponent to get
\begin{multline*}
\Re\left(\frac\kappa2(-(\theta+\sigma w)^2+\theta^2)+f_\kappa(\theta+\sigma w-\theta)\right)
=\Re\left(\frac\kappa2(-\theta^2e^{2\I t}+\theta^2)+f_\kappa(\theta e^{\I t}-\theta)\right)\\
=\frac{\kappa\theta^2}2(1-\cos(2t))+f_\kappa\theta(\cos t-1)\le\kappa\theta^2\sin^2 t.
\end{multline*}
Since $\kappa\simeq1/\theta$, this term becomes small compared to $|\Im(w)|^2=\theta^2\sin^2 t/\sigma^2$ as $\theta$ gets large.
It follows from Lemma~\ref{lemma:perturbedgamma} below that the ratio $\Gamma(\sigma w-b)/\Gamma(\beta-\sigma w)$ decays along the semicircle
$\{(e^{\I t}-1)\theta/\sigma, t^*\leq |t|\leq \pi/2\}$.

For $w\in\{-|y|+\I y,|y|\geq \theta/\sigma\}$, it follows from the last statement of Lemma~\ref{lemma:wtg} that
$\partial_Y\wt g(-Y,Y)\sim-\ln Y$ meaning that $\wt g(-Y,Y)\simeq-Y\ln Y$ for $Y$ large.
What we show is that the rest in the exponent is of smaller order.
For $w=-y+\I y$, we have
\begin{multline*}
\Re\left(\frac\kappa2(-(\theta+\sigma w)^2+\theta^2)+f_\kappa(\theta+\sigma w-\theta)\right)\\
=\Re\left(\frac\kappa2(-(\theta-\sigma y+\I\sigma y)^2+\theta^2)+f_\kappa(\theta-\sigma y+\I\sigma y-\theta)\right)
=\kappa\theta\sigma y-f_\kappa\sigma y
\end{multline*}
which simplifies in the leading order by \eqref{kappaseries}--\eqref{fseries}, but it is certainly controlled by the decay $-Y\ln Y$ in the exponent.
The factor $\Gamma(\sigma w-b)/\Gamma(\beta-\sigma w)$ decreases as $|\Im(w)|$ increases along $\{-|y|+\I y,|y|\geq \theta/\sigma\}$ by Lemma~\ref{lemma:perturbedgamma}.
This shows that for $\theta$ large enough, we have the bound
$$|K_\theta(w,w')|\le Ce^{-|\Im(w)|}.$$

\smallskip
\emph{Case 2(b).} The contribution of the integration over $B_{w+k/\sigma}$ is
(up to a $\pm$ sign depending on $k$) given by
\[\frac{S^k e^{N G(\Phi(w))-N G(\Phi(w+k/\sigma))}}{w+k/\sigma-w'}
\frac{\Gamma(\beta-\sigma w-k)}{\Gamma(\sigma w-b+k)}
\frac{\Gamma(\sigma w-b)}{\Gamma(\beta-\sigma w)} \frac{\Gamma(\theta+\sigma w+k)}{\Gamma(\theta+\sigma w)}
\frac{\theta^{\sigma w}}{\theta^{\sigma w+k}}.\]
It is shown in the last part of the proof of Proposition~6.2 in~\cite{BCF12} that the first ratio can be bounded by
\begin{equation}\label{firstresidueratio}
\left|\frac{S^k e^{N G(\Phi(w))-N G(\Phi(w+k/\sigma))}}{w+k/\sigma-w'}\right|\le e^{-c|\Im(w)|}
\end{equation}
for an arbitrary $c>0$ if $N$ is large enough uniformly in $k$. For the rest of the factors, we have
\begin{equation}\label{uparrow1}
\frac{\Gamma(\beta-\sigma w-k)}{\Gamma(\sigma w-b+k)}\frac{\Gamma(\sigma w-b)}{\Gamma(\beta-\sigma w)}
=\frac1{(\sigma w-b)_{k\uparrow}(\beta-\sigma w-k)_{k\uparrow}}
\end{equation}
and
\begin{equation}\label{uparrow2}
\frac{\Gamma(\theta+\sigma w+k)}{\Gamma(\theta+\sigma w)} \frac{\theta^{\sigma w}}{\theta^{\sigma w+k}}
=\frac{(\theta+\sigma w)_{k\uparrow}}{\theta^k}
\end{equation}
where $(x)_{k\uparrow}=x(x+1)\dots(x+k-1)$ is the rising factorial.

For a fixed $w$, \eqref{uparrow2} goes to $1$ as $\theta\to\infty$, but the error is not uniform in $\theta$.
Hence we regard \eqref{uparrow2} as a degree $k$ polynomial in $w$.
Since along the contour \mbox{$\{(e^{\I t}-1)\theta/\sigma, t^*\leq |t|\leq \pi/2\}\cup\{-|y|+\I y,|y|\geq \theta/\sigma\}$}, $|\Re(w)|\le|\Im(w)|$,
the absolute value of \eqref{uparrow2} is also at most a degree $k$ polynomial in $|\Im(w)|$.
The leading coefficient is uniformly small for large $\theta$.

On the other hand, the denominator of \eqref{uparrow1} is independent of $\theta$, and the imaginary part of each of the factors of the products is $\sigma\Im(w)$, hence
\[\left|\frac1{(\sigma w-b)_{k\uparrow}(\beta-\sigma w-k)_{k\uparrow}}\right|\le\frac1{\sigma^{2k}|\Im(w)|^{2k}}.\]
This cancels the polynomial coming from \eqref{uparrow2}, and since $k\le\sigma|\Im(w)|$, by choosing $c$ in \eqref{firstresidueratio} large enough,
the product is still exponentially small. The sum of the $k$ residues is also bounded by $e^{-|\Im(w)|}$ as required.
\end{proof}

\begin{lemma}\label{lemma:perturbedgamma}
For the function
\[w\mapsto\left|\frac{\Gamma(\sigma w-b)}{\Gamma(\beta-\sigma w)}\right|,\]
the following holds:
\begin{enumerate}
\item Along the semicircle $w(t)=(e^{\I t}-1)\theta/\sigma$, it decreases for $t^*\le t\le\pi/2$
and increases for $-\pi/2\le t\le-t^*$ if $\theta$ is large enough.

\item Along the halflines $w(y)=-y\pm\I y$, it decreases for $y\ge\theta/\sigma$ if $\theta$ is large enough.
\end{enumerate}
\end{lemma}

\begin{proof}
Let us call
\[f(x,y):=\Re(\ln\Gamma(x+\I y))=\sum_{n=0}^\infty \left(\frac{x}{n+1}-\frac12 \ln\left((x+n)^2+y^2\right)+\ln(n)\Id_{n\geq 1}\right)-\gamma_{\rm E} x\]
where the second equation appears at the beginning of Section 5.2 in~\cite{BCF12}. It follows that
\begin{align*}
\frac{\partial f(x,y)}{\partial x}&=\sum_{n=0}^\infty \left(\frac1{n+1}-\frac{x+n}{(x+n)^2+y^2}\right)-\gamma_{\rm E}\\
\frac{\partial f(x,y)}{\partial y}&=\sum_{n=0}^\infty -\frac{y}{(x+n)^2+y^2}.
\end{align*}

\begin{enumerate}
\item Let $w(t)=(e^{\I t}-1)\theta/\sigma$. It is elementary to see that
\begin{equation}\label{gammaratio'1}\begin{aligned}
&\frac\partial{\partial t}\Re\left(\ln\frac{\Gamma(\sigma w(t)-b)}{\Gamma(\beta-\sigma w(t))}\right)\\
&\qquad=\theta\sin t\bigg(\sum_{n=0}^\infty\bigg(\frac{-\theta-b+n}{(\theta(\cos t-1)-b+n)^2+\theta^2\sin^2 t}\\
&\qquad\qquad+\frac{\theta+\beta+n}{(\theta(1-\cos t)+\beta+n)^2+\theta^2\sin^2 t}-\frac2{n+1}\bigg)+2\gamma_{\rm E}\bigg).
\end{aligned}\end{equation}
If we consider the above sum for $n\ge\theta$, then it is not hard to show by dominated convergence that
\begin{multline*}
\sum_{n=\lfloor\theta\rfloor}^\infty\left(\frac{-\theta-b+n}{(\theta(\cos t-1)-b+n)^2+\theta^2\sin^2 t}
+\frac{\theta+\beta+n}{(\theta(1-\cos t)+\beta+n)^2+\theta^2\sin^2 t}-\frac2{n+1}\right)\\
\to\int_1^\infty\d x\left(\frac{x-1}{(\cos t-1+x)^2+\sin^2 t}+\frac{x+1}{(1-\cos t+x)^2+\sin^2 t}-\frac2{x+1}\right)
\end{multline*}
as $\theta\to\infty$ for a fixed $t\in(0,\pi/2]$. The integrand on the right-hand side is $\Or(x^{-2})$ as $x\to\infty$, hence the integral is finite.
On the other hand, the sum
\[\sum_{n=0}^{\lfloor\theta\rfloor}\left(\frac{-\theta-b+n}{(\theta(\cos t-1)-b+n)^2+\theta^2\sin^2 t}
+\frac{\theta+\beta+n}{(\theta(1-\cos t)+\beta+n)^2+\theta^2\sin^2 t}\right)\]
remains bounded as $\theta\to\infty$, since it converge to the corresponding integral on $[0,1]$.
But $\sum_{n=0}^{\lfloor\theta\rfloor}\frac2{n+1}\simeq2\ln\theta$ which goes to infinity.
This shows that the derivative in \eqref{gammaratio'1} is negative for $\theta$ large enough if $t\in(0,\pi/2]$.
For negative $t$, the argument is identical. The factor $\sin t$ on the right-hand side of \eqref{gammaratio'1} makes the derivative positive.
This is sufficient for the first assertion of the lemma.

\item We set $w(y)=-y+\I y$. A straightforward computation yields
\begin{equation}\label{gammaratio'2}\begin{aligned}
&\frac1\sigma\frac\partial{\partial y}\Re\left(\ln\frac{\Gamma(\sigma w(y)-b)}{\Gamma(\beta-\sigma w(y))}\right)\\
&\qquad=\sum_{n=0}^\infty\left(\frac{-2\sigma y-b+n}{(-\sigma y-b+n)^2+(\sigma y)^2}
+\frac{2\sigma y+\beta+n}{(\sigma y+\beta+n)^2+(\sigma y)^2}-\frac2{n+1}\right)+2\gamma_{\rm E}.
\end{aligned}\end{equation}
As in the previous part of the proof, we have
\begin{multline*}
\sum_{n=\lfloor y\rfloor}^\infty\left(\frac{-2\sigma y-b+n}{(-\sigma y-b+n)^2+(\sigma y)^2}
+\frac{2\sigma y+\beta+n}{(\sigma y+\beta+n)^2+(\sigma y)^2}-\frac2{n+1}\right)\\
\to\int_1^\infty\d x\left(\frac{-2\sigma+x}{(-\sigma+x)^2+\sigma^2}+\frac{2\sigma+x}{(\sigma+x)^2+\sigma^2}-\frac2{x+1}\right)
\end{multline*}
as $y\to\infty$ and the integral is finite, because the integrand is $\Or(x^{-2})$.
Similarly to the first part of this proof, the first two summands on the right-hand side of \eqref{gammaratio'2} summed over $n\in[0,\lfloor y\rfloor]$ are finite,
because the corresponding integral on $[0,1]$ is finite.
Hence the sum $\sum_{n=0}^{\lfloor y\rfloor}\frac2{n+1}$ makes the derivative negative for $y$ large enough.
The statement for the other branch of the contour can be proved in the same way.
\end{enumerate}
\end{proof}

Proposition~\ref{PFPropPtConvKPZ} and \ref{PFPropBoundKPZ} together imply the convergence of the Fredholm determinants
\begin{equation}\label{restrictedKuconv}
\lim_{N\to\infty}\det(\Id+K_u)_{L^2(\Cv{(2a+\alpha)/3,\varphi})}=\det(\Id-\wt K_{b,\beta})_{L^2(\Cv{w})}
\end{equation}
if assumption \eqref{bbetaassumption} holds.
If \eqref{bbetaassumption} does not hold,
then we modify locally the contours $\Cv{v}$ and $\Cv{\tilde z}$ around the critical point $\theta$ such that they cross the real axis strictly between the poles $a$ and $\alpha$.
The contours $\Cv{w}$ and $\Cv{z}$ are similarly modified.
From now on, we focus on the $\theta\to\infty$ limit of these contours, that is,
we explain how to modify the contour $\Cv{w}$ starting from $-\frac1{4\sigma}+\I\R$ and $\Cv{z}$ starting from $\frac 1{4\sigma}+\I\R$ (since $p=1$).
In order to keep the factor $1/\sin(\sigma\pi(z-w))$ bounded in the limiting kernel $\wt K_{b,\beta}$,
the contour $\Cv{z}$ has to be confined between $\e+\Cv{w}$ and $1/\sigma-\e+\Cv{w}$ for an arbitrarily small but fixed $\e>0$ which might depend on $b$ and $\beta$.

If $\beta-b\ge1$, then the distance of the poles at $\frac b\sigma$ and $\frac\beta\sigma$ is enough to let the two contours run parallelly between them.
In this case, for $b>-\frac12$, replace the $|\Im(w)|\le\frac{2b+1}{2\sigma}$ and $|\Im(z)|\le\frac{2b+1}{2\sigma}$ part of $\Cv{w}$ and $\Cv{z}$ by the parallel semicircles
$\{\frac{2b+1}{2\sigma}e^{\I t}-\frac1{4\sigma},-\frac\pi2\le t\le\frac\pi2\}$ and $\{\frac{2b+1}{2\sigma}e^{\I t}+\frac1{4\sigma},-\frac\pi2\le t\le\frac\pi2\}$ respectively.
The case $\beta<\frac12$ is handled symmetrically.
If $\beta-b<1$, then the contours will come closer together between the poles, we choose them so that they intersect the real axis at $(2b+\beta)/(3\sigma)$ and at $(b+2\beta)/(3\sigma)$.
This choice of the modified contours is shown on Figure~\ref{PFFigPathsKPZproof}.

\begin{figure}
\begin{center}
\psfrag{-thetat}[cb]{$-\theta c_\theta N^{1/3}$}
\psfrag{w}[cb]{$w$}
\psfrag{Cw}[lb]{$\mathcal{C}_w$}
\psfrag{Cz}[lb]{$\mathcal{C}_z$}
\psfrag{bbar}[cb]{$\bar b$}
\psfrag{O1}[cb]{$\Or(1)$}
\psfrag{0}[cb]{$0$}
\psfrag{b}[cb]{$\frac b\sigma$}
\psfrag{beta}[cb]{$\frac\beta\sigma$}
\psfrag{-thetat}[ct]{$-\frac{\theta}{\sigma}$}
\psfrag{+p4}[lt]{$\frac{p}{4\sigma}$}
\psfrag{-14}[ct]{$-\frac{1}{4\sigma}$}
\psfrag{c1}[ct]{$\tfrac{2b+\beta}{3\sigma}$}
\psfrag{c2}[ct]{$\tfrac{b+2\beta}{3\sigma}$}
\includegraphics[height=7cm]{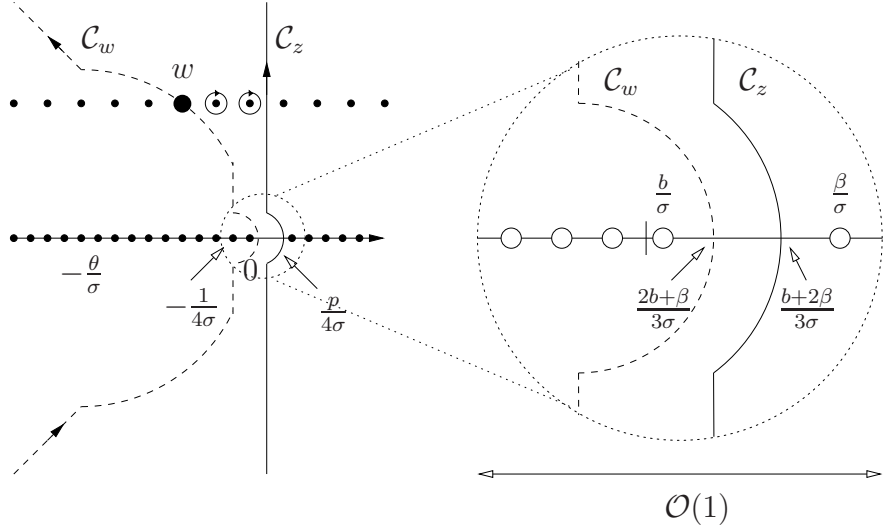}
\caption{A possible perturbation of the integration contours, compare with Figure~\ref{PFFigPathsKPZ} (right).
The dots are the singularities of $\Gamma(\sigma w-b)$ at $b/\sigma,(b-1)/\sigma,\dots$ and those of $\Gamma(\beta-\sigma z)$ at $\beta/\sigma,(\beta+1)/\sigma,\dots$.}\label{PFFigPathsKPZproof}
\end{center}
\end{figure}

The local modification of the contours has no influence on the bounds for large $z$ and/or for large $w$.
This is because $N G(\theta+\sigma b)-N G(\theta) =\Or(1)$ and the contour for $z$ is the same away from a distance $\Or(1)$ from the origin.
This shows that \eqref{restrictedKuconv} remains valid for any $b<\beta$.

\subsection{Reformulation of the kernel}

The following lemma about the reformulation along with its proof is the analogue of Lemma~8.6 in~\cite{BCF12}.

\begin{lemma}\label{lem:PFLemKPZreformuation}
For the kernels $\wt K_{b,\beta}$ defined in \eqref{defwtKbbeta} and for $K_{b,\beta}$ given in \eqref{defKbbeta}, it holds
\begin{equation}\label{reformulation}
\det(\Id-\wt K_{b,\beta})_{L^2(\Cv{w})}=\det(\Id-K_{b,\beta})_{L^2(\R_+)}.
\end{equation}
\end{lemma}

\begin{proof}
Assume first that \eqref{bbetaassumption} holds.
For this choice of $b$ and $\beta$, if $w'\in\Cv{w}$ and $z\in\Cv{z}$, then $\Re(z-w')>0$ and one can write
\[\frac1{z-w'}=\int_{\R_+}\d\lambda e^{-\lambda(z-w')}.\]
Using this equation, we have
\[\wt K_{b,\beta}(w,w')=\int_{\R_+}\d\lambda A(w,\lambda)B(\lambda,w')\]
where $A:L^2(\Cv{w})\to L^2(\R_+)$ with
\[A(w,\lambda)=\frac1{2\pi \I}\int_{\Cv{z}}\d z \frac{\sigma\pi S^{\sigma(z-w)}}{\sin(\sigma\pi(z-w))} e^{z^3/3-w^3/3-\lambda z}
\frac{\Gamma(\sigma w-b)}{\Gamma(\sigma z-b)}\frac{\Gamma(\beta-\sigma z)}{\Gamma(\beta-\sigma w)}\]
and $B:L^2(\R_+)\to
L^2(\Cv{w})$ with $B(\lambda,w')=e^{\lambda w'}$.

One checks easily that
\[BA(x,y)=\frac1{2\pi\I}\int_{\Cv{w}}\d w B(x,w)A(w,y)=K_{b,\beta}(x,y),\]
and \eqref{reformulation} follows since $\det(\Id-AB)_{L^2(\Cv{w})}=\det(\Id-BA)_{L^2(\R_+)}$.

It remains to relax condition \eqref{bbetaassumption}.
By Lemma~\ref{lem:Fredholm_anal}, both sides of \eqref{reformulation} are analytic functions of the parameters $b$ and $\beta$ for $b<\beta$.
We have proved above that the two analytic functions coincide if \eqref{bbetaassumption} holds,
therefore \eqref{reformulation} follows by analytic continuation for any $b<\beta$.
\end{proof}

\section{SHE/KPZ equation with stationary initial data -- Proof of Theorem~\ref{ThmFormulaStationary}}\label{SectStatSHE}
To prove Theorem~\ref{ThmFormulaStationary} using the formula of Theorem~\ref{ThmFormulaContinuous}, we need the following lower tail estimate of the solution to the SHE proven in~\cite{CH13}.

\begin{lemma}\label{lem:Ftailbound}
Fix $T>0$ and $X\in \R$ and consider $\mathcal Z_{b,\beta}(T,X)$. For any $b\in \R$ and $\delta>0$, there exist constants $c_1,c_2,c_3>0$ such that for all $\beta\in (b-\delta,b+\delta)$, and all $s\geq 1$
$$\PP\left(\mathcal{Z}_{b,\beta}(T,X)< e^{-c_3 s}\right)\le c_1 e^{-c_2 s^{3/2}}.$$
\end{lemma}
\begin{proof}
This follows the lower tail bound of~\cite[Corollary~1.13]{CH13}.
That result, however, is stated in such a way that the constants $c_1,c_2,c_3$ depend on $\beta$.
The lemma we are proving asks for constants which are independent as $\beta$ varies in $(b-\delta,b+\delta)$.
However, the desired uniformity follows via a simple coupling argument.
The stochastic heat equation is attractive, it means that if we couple our initial data
\mbox{$\mathcal{Z}_{b,\beta}(0,X)=\mathbf{1}_{X\leq 0} \big(B^l(X)+\beta X\big) + \mathbf{1}_{X>0} \big(B^r(X)+b X\big)$} to the same Brownian motions $B^l$ and $B^r$
(here $B^l:(-\infty,0]\to \R$ is a Brownian motion without drift pinned at $B^l(0)=0$, and $B^r:[0,\infty)\to \R$ is an independent Brownian motion pinned at $B^r(0)=0$),
then for $\beta>\beta'$ since $\mathcal{Z}_{b,\beta'}(0,X)\geq \mathcal{Z}_{b,\beta}(0,X)$ for all $X\in R$,
it follows that $\mathcal{Z}_{b,\beta'}(T,X)\geq \mathcal{Z}_{b,\beta}(T,X)$ for all $X\in R$ and $T>0$.
This immediately implies that to bound the lower tail of $\mathcal{Z}_{b,\beta}(T,X)$ as $\beta$ varies in $(b-\delta,b+\delta)$,
it suffices to choose $c_1,c_2,c_3>0$ corresponding to $\beta = b+\delta$.
\end{proof}

\begin{proof}[Proof of Theorem~\ref{ThmFormulaStationary}]
We need to show the convergence of both sides of \eqref{Bessel=Fredholm} to the $\beta=b$ expression. The convergence of the right-hand side of \eqref{Bessel=Fredholm} to that of \eqref{K0transform} up to a factor of $\sigma$ follows from Theorem~\ref{thm:limK0delta} since $\lim_{\beta\to b}\Gamma(\beta-b)(\beta-b)=1$. Now consider the left-hand side of~\eqref{Bessel=Fredholm}. Denote by $c_0=2\sqrt{Se^{\frac{X^2}{2T}+\frac T{24}}}$, $x=c_0\sqrt{\mathcal Z_{b,\beta}(T,X)}$, and set $\nu=\beta-b$. Then
\begin{equation}\label{eq7.1}
\textrm{l.h.s.~of}~\eqref{Bessel=Fredholm} = \EE(x^\nu \BesselK_{-\nu}(x)) = -\int_{\R_+}\d\xi \xi^\nu \BesselK_{1-\nu}(\xi)\PP(x\leq \xi)
\end{equation}
where we used integration by parts. Let $c_1$ be the constant in Lemma~\ref{lem:Ftailbound} and denote by $\xi_0=c_0 e^{-c_1/2}$. Decomposing the integral into $[0,\xi_0)$ and $[\xi_0,\infty)$ and making the change of variable $\xi(s)=c_0 e^{-c_1 s/2}$, we obtain
\begin{equation*}
\begin{aligned}
\eqref{eq7.1}=&-\int_{\xi_0}^\infty \d\xi \xi^\nu \BesselK_{1-\nu}(\xi)\PP(x\leq \xi) \\
&-\frac{c_0 c_1}{2}\int_{1}^\infty \d s e^{-c_1 s/2}\xi(s)^\nu \BesselK_{1-\nu}(\xi(s))\PP(\mathcal Z_{b,\beta}(T,X)\leq e^{-c_1 s}).
\end{aligned}
\end{equation*}
Using Lemma~\ref{LemmaModifiedBessel}~(c) and (d), the first integral is bounded uniformly in $\nu$ and we can therefore take $\nu\to 0$.
The same lemma implies also that $e^{-c_1 s/2}\xi(s)^\nu \BesselK_{1-\nu}(\xi(s))$ is bounded by $e^{c s}$ for some constant $c$ independent of $\nu$ and
by the bound on the tail of the probability of Lemma~\ref{lem:Ftailbound}, we also have that the second integrand is uniformly bounded by an integrable function.
Thus by dominated convergence, we can take the $\nu\to 0$ limit inside and we obtain
\begin{equation*}
\lim_{\beta\to b}\textrm{l.h.s.~of}~\eqref{Bessel=Fredholm} = -\int_{\R_+}d\xi \BesselK_{1}(\xi)\PP(c_0 \sqrt{\mathcal Z_{b}(T,X)} \leq \xi) = \EE(\BesselK_0(c_0 \sqrt{\mathcal Z_{b}(T,X)}))
\end{equation*}
where in the last step we integrated by parts.
\end{proof}

Later in this section, we will work in $L^2(\R_+)$, so the functions are defined on $\R_+$ and the scalar product of two functions is meant as
\[\langle f,g\rangle=\int_{\R_+}\d xf(x)g(x).\]
To extend the definition \eqref{defqss} for
$u,v\in\left(-\frac14,\frac14\right)$, let us define on $\R_+$ the functions
\begin{equation}\label{defquv}\begin{aligned}
q_{u,v}(x)&=\frac1{2\pi\I}\int_{-\frac1{4\sigma}+\I\R}\d w \frac{\sigma\pi S^{v-\sigma w}}{\sin(\pi(v-\sigma w))} e^{-w^3/3+wx} \frac{\Gamma(\sigma w-u)}{\Gamma(v-\sigma w)}\\
&=\frac1{2\pi\I}\int_{\frac1{4\sigma}+\I\R}\d z \frac{\sigma\pi S^{\sigma z+v}}{\sin(\pi(\sigma z+v))} e^{z^3/3-zx} \frac{\Gamma(-u-\sigma z)}{\Gamma(\sigma z+v)},
\end{aligned}\end{equation}
and recall \eqref{defrs}. Note that $r_s\not\in L^2(\R_+)$ if
$s\le0$.

Further, to extend \eqref{defbarKbb} for
$b,\beta\in\left(-\frac14,\frac14\right)$, we introduce the kernel
\begin{equation}\label{defbarK}
\bar K_{b,\beta}(x,y)=\frac1{(2\pi\I)^2}\int_{-\frac1{4\sigma}+\I\R}\d w\int_{\frac1{4\sigma}+\I\R}\d z \frac{\sigma\pi S^{\sigma(z-w)}}{\sin(\sigma\pi(z-w))}
\frac{e^{z^3/3-zy}}{e^{w^3/3-wx}} \frac{\Gamma(\beta-\sigma z)}{\Gamma(\sigma z-b)} \frac{\Gamma(\sigma w-b)}{\Gamma(\beta-\sigma w)}.
\end{equation}
Note that the only difference between $K_{b,\beta}$ and $\bar K_{b,\beta}$ is the two integration contours.
The one for $K_{b,\beta}$ is shown on the left-hand side of Figure~\ref{fig:KbarKcontours}, for $\bar K_{b,\beta}$, they are vertical lines.
Note also that $\bar K_{b,b}=\bar K_b$ of \eqref{defbarKbb} and $q_{b,b}=q_b$ of \eqref{defqss}.
Finally recall the function $\Xi$ defined in \eqref{defz}.

\begin{figure}
\psfrag{0}[c]{$\frac b\sigma$}
\psfrag{d}[c]{$\frac\beta\sigma$}
\psfrag{Cw}[cc]{$\Cv{w}$}
\psfrag{Cz}[cc]{$\Cv{z}$}
\psfrag{Cw1}[cc]{$\Cv{w}'$}
\psfrag{Cz1}[cc]{$\Cv{z}'$}
\psfrag{Cw2}[cc]{$\Cv{w}''$}
\psfrag{Cz2}[cc]{$\Cv{z}''$}
\includegraphics[height=4cm]{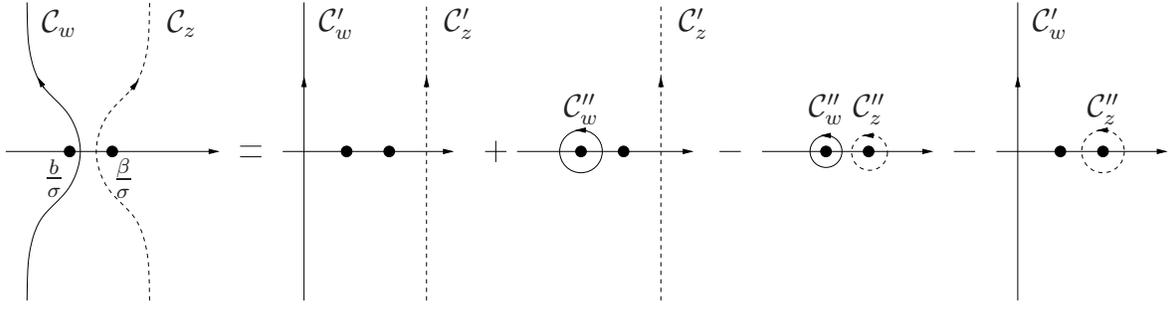}
\caption{The integration contours $\Cv{w}$ and $\Cv{z}$ for $K_{b,\beta}$ are on the left.
The other contours are: $\Cv{w}'=-\frac{1}{4\sigma}+\I\R$, $\Cv{z}'=\frac{1}{4\sigma}+\I\R$, $\Cv{w}''$ a small circle around $b/\sigma$ and $\Cv{z}''$ a small circle around $\beta/\sigma$.
By modifying the contours as shown here and applying the residue theorem, one gets \eqref{applyresidue}.
The dots show the poles of the integrands validating the manipulations of the contours.}\label{fig:KbarKcontours}
\end{figure}

\begin{remark}\label{rem:barKproduct}
We prove in Lemma~\ref{LemmaInvertibilityOfK00} that
\begin{equation}\label{detneq0}
\det(\Id-\bar K_{b,b})_{L^2(\R_+)}\neq0.
\end{equation}
Together with Lemma~\ref{lem:barKproperties} below, it shows that the right-hand side of \eqref{defz} is finite as follows.
In the first scalar product, if $b\le0$, then $r_b\not\in L^2(\R_+)$ (if $b\ge0$, then $r_{-b}\not\in L^2(\R_+)$), but by Lemma~\ref{lem:barKproperties}, $\bar K_{b,b}r_{-b}$ decays
exponentially with a faster rate than $r_b$ might blow up.
The second one is obviously finite. For the third one, we can write
\begin{equation}\label{barKproductrewrite}
\big\langle(\Id-\bar K_{b,b})^{-1}q_{b,b},r_b\big\rangle=\langle
q_{b,b},r_b\rangle+\big\langle\bar K_{b,b}(\Id-\bar
K_{b,b})^{-1}q_{b,b},r_b\big\rangle.
\end{equation}
Using Lemma~\ref{lem:barKproperties} again, $q_{b,b}$ decays exponentially with a faster rate than $r_b$ might blow up, hence $\langle q_{b,b},r_b\rangle$ is finite.
On the other hand,
\[\big\langle\bar K_{b,b}(\Id-\bar K_{b,b})^{-1}q_{b,b},r_b\big\rangle=\big\langle(\Id-\bar K_{b,b})^{-1}q_{b,b},\bar K_{-b,-b}r_b\big\rangle,\]
since the adjoint of $\bar K_{b,b}$ in the real $L^2(\R_+)$ is $\bar K_{-b,-b}$ which can also be seen from the representation \eqref{eqA9}.
The function $\bar K_{-b,-b}r_b$ is already in $L^2(\R_+)$, so the second term on the right-hand side of \eqref{barKproductrewrite} is also well-defined.
A similar argument works for the last scalar product in \eqref{defz}.
\end{remark}

Fix $b\in\left(-\frac14,\frac14\right)$. The kernel $K_{b,\beta}$ is defined for all $\beta\in\left(b,\frac14\right)$ by \eqref{defKbbeta}.
The following theorem describes the behaviour of the corresponding Fredholm determinant in the decreasing $\beta\to b$ limit which is that it goes to $0$ linearly in $\beta-b$.

\begin{theorem}\label{thm:limK0delta}
Let $b\in\left(-\frac14,\frac14\right)$ be fixed. For the kernel $K_{b,\beta}$, we have
\begin{equation}\label{limK0delta}
\lim_{\beta\to b}\frac1{\beta-b}\det(\Id-K_{b,\beta})=\frac1\sigma\Xi(S,b,\sigma)
\end{equation}
with $\Xi$ defined in \eqref{defz}. Recall that the notation $q_{b}$ and $\bar{K}_{b}$ from Definition~\ref{longdef} are related to that above via $q_b = q_{b,b}$ and $\bar{K}_b = \bar{K}_{b,b}$.
\end{theorem}

\begin{remark}\label{rem:bcondition}
Theorem~\ref{thm:limK0delta} is proved with the condition $b\in\left(-\frac14,\frac14\right)$ for technical convenience. One could likely extend the proof with minor modifications up to the range $b\in (-1,1)$. Beyond this range the integration contours which appear implicitly on the right-hand side of \eqref{limK0delta} in the definitions of $q_{b}$, $q_{-b}$ and $\bar K_{b}$ have to depend on $b$ so that
the $w$ contours cross the real axis on the right of the singularities of $\Gamma(\sigma w-b)$ at $(b-1)/\sigma,(b-2)/\sigma,\dots$ whereas the $z$ contours cross on the left of the poles $(b+1)/\sigma,(b+2)/\sigma,\dots$ coming from $\Gamma(b-\sigma z)$.
On the other hand, if $b$ is not in $(-1,1)$, then the kernel function $\bar K_{b}(x,y)$ does not decay in $x$ or $y$, hence the right-hand side of \eqref{limK0delta} is not well-defined via \eqref{defz}.
\end{remark}

Before proving Theorem~\ref{thm:limK0delta}, we give the following decay estimates.

\begin{lemma}\label{lem:barKproperties}
For each $\varepsilon>0$ fixed, there is a $C$ which only depends on $\varepsilon$ such that for all $b,\beta\in\left(-\frac14+\varepsilon,\frac14-\varepsilon\right)$, we have the following bounds:
\begin{align}
\left|\bar K_{b,\beta}(x,y)\right|&\le Ce^{-\frac1{4\sigma}(x+y)},\label{K00decay}\\
\left|q_{b,\beta}(x)\right|&\le Ce^{-\frac1{4\sigma}x}.\label{q00decay}
\end{align}
\end{lemma}

\begin{proof}
We can estimate by taking the absolute value of the integrand in
\eqref{defbarK} as follows
\begin{multline*}
\left|\bar K_{b,\beta}(x,y)\right|\\
\le\frac1{(2\pi)^2}\int_{-\frac1{4\sigma}+\I\R}|\d w|\int_{\frac1{4\sigma}+\I\R}|\d z| \left|\frac{\sigma\pi S^{\sigma(z-w)}}{\sin(\sigma\pi(z-w))} \frac{e^{z^3/3}}{e^{w^3/3}}
\frac{\Gamma(\beta-\sigma z)}{\Gamma(\sigma z-b)} \frac{\Gamma(\sigma w-b)}{\Gamma(\beta-\sigma w)}\right|\cdot\left|e^{-zy+wx}\right|.
\end{multline*}
The factor $e^{z^3/3}/e^{w^3/3}$ has Gaussian decay in $|\Im(w)|$ and in $|\Im(z)|$, the other factors are slower, in particular, see \eqref{gammalemma2} for the gamma ratio.
For any value of $w\in-\frac14+\I\R$ and $z\in\frac14+\I\R$,
\[\left|e^{-zy+wx}\right|=e^{-\frac1{4\sigma}(x+y)}.\]
Since the integration paths pass at least $\varepsilon$ far from the singularities of the integrand for any given $\varepsilon$, one can choose a uniform constant $C$ so that \eqref{K00decay} holds.
\eqref{q00decay} can be proved similarly.
\end{proof}

\begin{proof}[Proof of Theorem~\ref{thm:limK0delta}]
The first step is to rewrite the kernel $K_{b,\beta}$, since in the original form given in \eqref{defKbbeta},
the two integration contours intersect the real axis between the pole at $b/\sigma$ and the pole at $\beta/\sigma$, so the two contours would
collide in the $\beta\to b$ limit, see also Figure~\ref{fig:KbarKcontours}.
Hence by using the residue theorem, we cross the pole at $b/\sigma$ with the $w$ integration contour and cross the pole at $\beta/\sigma$
with the $z$ integration contour, both manipulations resulting in a residue term.

If we assume that
\[-\frac14<b<\beta<\frac14,\]
then the new integration contours can be chosen to be $-\frac1{4\sigma}+\I\R$ for $w$ and $\frac1{4\sigma}+\I\R$ for $z$ as shown on Figure~\ref{fig:KbarKcontours}.
That is, with the notation \eqref{defbarK}, we can write
\begin{equation}\label{applyresidue}\begin{aligned}
K_{b,\beta}(x,y)=\bar K_{b,\beta}(x,y)
&+\frac1{2\pi\I}\int_{-\frac1{4\sigma}+\I\R}\d w \frac{\sigma\pi S^{\beta-\sigma w}}{\sin(\pi(\beta-\sigma w))}\frac{e^{\beta^3/(3\sigma^3)-\beta y/\sigma}}{e^{w^3/3-wx}}
\frac{\Gamma(\sigma w-b)}{\Gamma(\beta-\sigma w)} \frac1{\sigma\Gamma(\beta-b)}\\
&+\frac1{2\pi\I}\int_{\frac1{4\sigma}+\I\R}\d z \frac{\sigma\pi S^{\sigma z-b}}{\sin(\pi(\sigma z-b))}\frac{e^{z^3/3-zy}}{e^{b^3/(3\sigma^3)-bx/\sigma}}
\frac{\Gamma(\beta-\sigma z)}{\Gamma(\sigma z-b)} \frac1{\sigma\Gamma(\beta-b)}\\
&+\frac{\sigma\pi S^{\beta-b}}{\sin(\pi(\beta-b))} \frac{e^{\beta^3/(3\sigma^3)-\beta y/\sigma}}{e^{b^3/(3\sigma^3)-bx/\sigma}} \frac1{\sigma^2\Gamma(\beta-b)^2}
\end{aligned}\end{equation}
by the residue theorem.

Using the functions defined in \eqref{defquv} and \eqref{defrs}, we have
\begin{multline}\label{Kresidues}
K_{b,\beta}(x,y)=\bar K_{b,\beta}(x,y)
+q_{b,\beta}(x)r_\beta(y)\frac1{\sigma\Gamma(\beta-b)}+r_{-b}(x)q_{-\beta,-b}(y)\frac1{\sigma\Gamma(\beta-b)}\\
+\frac{\sigma\pi S^{\beta-b}}{\sin(\pi(\beta-b))} r_{-b}(x)r_\beta(y)\frac1{\sigma^2\Gamma(\beta-b)^2}.
\end{multline}
The last equation shows that $K_{b,\beta}$ is a finite rank perturbation of
$\bar K_{b,\beta}$, i.e.\ we could write
\[K_{b,\beta}(x,y)=\bar K_{b,\beta}(x,y)+\sum_{i=1}^3 f_i(x)g_i(y)\]
with appropriate $f_i$ and $g_i$. In this case, for the Fredholm determinants, the following holds
\begin{equation}\label{rank3pert}
\det\bigg(\Id-\bar K_{b,\beta}-\sum_{i=1}^3 f_i\otimes g_i\bigg)_{L^2(\R_+)}
=\det\left(\Id-\bar K_{b,\beta}\right)_{L^2(\R_+)}
\det\left[\delta_{ij}-\left\langle\left(\Id-\bar K_{b,\beta}\right)^{-1}f_i,g_j\right\rangle\right]_{i,j=1}^{3}
\end{equation}
where $\delta_{ij}$ is the Kronecker's delta provided that $\det(\Id-\bar K_{b,\beta})_{L^2(\R_+)}\neq0$.
For $\beta$ close enough to $b$, this follows by continuity from \eqref{detneq0}.

By \eqref{Kresidues}, we define
\begin{equation}\label{deffigi}\begin{aligned}
f_1(x)&=\frac{q_{b,\beta}(x)}{\sigma\Gamma(\beta-b)},\qquad & g_1(y)&=r_\beta(y),\\
f_2(x)&=\frac{r_{-b}(x)}{\sigma\Gamma(\beta-b)},\qquad & g_2(y)&=q_{-\beta,-b}(y),\\
f_3(x)&=\frac{r_{-b}(x)}{\sigma\Gamma(\beta-b)},\qquad & g_3(y)&=\frac{\pi S^{\beta-b}}{\sin(\pi(\beta-b))}\frac{r_\beta(y)}{\Gamma(\beta-b)}.
\end{aligned}\end{equation}
With this choice of $f_i$ and $g_i$, the Fredholm determinant of $K_{b,\beta}$ is equal to \eqref{rank3pert}.
Since we are to take the limit of $(\beta-b)^{-1}\det(\Id-K_{b,\beta})_{L^2(\R_+)}$ as $\beta\to b$,
it is enough to consider the Taylor series up to first order in the second determinant on the right-hand side of \eqref{rank3pert}.
With the choice \eqref{deffigi}, this second determinant is equal to
\begin{equation}\label{detseries}
\det\left[\begin{matrix}
1-\frac{\beta-b}\sigma\langle Rq_{b,\beta},r_\beta\rangle
& -\frac{\beta-b}\sigma\langle Rq_{b,\beta},q_{-\beta,-b}\rangle
& -\frac{\beta-b}\sigma\langle Rq_{b,\beta},r_\beta\rangle\\
-1+\Or(\beta-b)
& 1-\frac{\beta-b}\sigma\langle Rr_{-b},q_{-\beta,-b}\rangle
& -1-\frac{\beta-b}\sigma(\xi+\langle R\bar K_{b,\beta}r_{-b},r_\beta\rangle\\
-1+\Or(\beta-b)
& -\frac{\beta-b}\sigma\langle Rr_{-b},q_{-\beta,-b}\rangle
& -\frac{\beta-b}\sigma(\xi+\langle R\bar K_{b,\beta}r_{-b},r_\beta\rangle)
\end{matrix}\right]
\end{equation}
where we neglect all the $\Or((\beta-b)^2)$ and higher order terms and we write $R=(\Id-\bar K_{b,\beta})^{-1}$ for simplicity in the above formula.
The value
\[\xi=b^2/\sigma^2+\sigma(2\gamma_{\rm E}+\ln S).\]
To obtain the first column of \eqref{detseries}, we use that
\begin{align*}
\big\langle(\Id-\bar K_{b,\beta})^{-1}f_2,g_1\big\rangle
&=\langle f_2,g_1\rangle+\big\langle K_{b,\beta}(\Id-\bar K_{b,\beta})^{-1}f_2,g_1\big\rangle\\
&=\frac{e^{\beta^3/(3\sigma^3)-b^3/(3\sigma^3)}}{(\beta-b)\Gamma(\beta-b)}
+\frac1{\sigma\Gamma(\beta-b)}\big\langle K_{b,\beta}(\Id-\bar K_{b,\beta})^{-1}r_{-b},r_\beta\big\rangle\\
&=1+\Or(\beta-b).
\end{align*}
The scalar product $\langle(\Id-\bar K_{b,\beta})^{-1}f_3,g_1\rangle$ is the same.

To get the last two entries in the third column of \eqref{detseries}, we do the separation
\[\big\langle(\Id-\bar K_{b,\beta})^{-1}f_2,g_3\big\rangle=\langle f_2,g_3\rangle+\big\langle\bar K_{b,\beta}(\Id-\bar K_{b,\beta})^{-1}f_2,g_3\big\rangle\]
where the first scalar product on right-hand side is of order $1$ whereas the rest is $\Or(\beta-b)$:
\begin{align*}
\langle f_2,g_3\rangle&=\frac{\sigma\pi S^{\beta-b}}{\sin(\pi(\beta-b))}\frac{e^{\beta^3/(3\sigma^3)-b^3/(3\sigma^3)}}{\Gamma(\beta-b)^2}\int_0^\infty\d x e^{-(\beta-b)x/\sigma}\\
&=1+(b^2/\sigma^2+2\sigma\gamma_{\rm E}+\sigma\ln S)\frac{\beta-b}\sigma+\Or((\beta-b)^2).
\end{align*}
This argument works again for $f_3$ instead of $f_2$.

The other terms in the determinant \eqref{detseries} are computed easily by \eqref{rank3pert} and \eqref{deffigi}.
Furthermore, all of these terms are finite which can be seen by using the idea of Remark~\ref{rem:barKproduct}.

By expanding the determinant in \eqref{detseries} and considering the terms up to first order, one can see that
\begin{equation}\label{detK0delta}\begin{aligned}
&\frac1{\beta-b}\det(\Id-K_{b,\beta})_{L^2(\R_+)}\\
&\quad=-\frac1\sigma\det(\Id-\bar K_{b,\beta})_{L^2(\R_+)}\big[b^2/\sigma^2+\sigma(2\gamma_{\rm E}+\ln S)
+\big\langle\bar K_{b,\beta}(\Id-\bar K_{b,\beta})^{-1}r_{-b},r_\beta\big\rangle\\
&\qquad+\big\langle(\Id-\bar K_{b,\beta})^{-1}q_{b,\beta},q_{-\beta,-b}\big\rangle
+\big\langle(\Id-\bar K_{b,\beta})^{-1}q_{b,\beta},r_\beta\big\rangle
+\big\langle(\Id-\bar K_{b,\beta})^{-1}r_{-b},q_{-\beta,-b}\big\rangle\big]\\
&\qquad+\Or(\beta-b).
\end{aligned}\end{equation}

What remains to show is that the right-hand side of
\eqref{detK0delta} converges to $\Xi(S,b,\sigma)/\sigma$. To see
that
\begin{equation}\label{convFred}
\det(\Id-\bar K_{b,\beta})_{L^2(\R_+)}\to\det(\Id-\bar K_{b,b})_{L^2(\R_+)},
\end{equation}
we use the Fredholm series expansion
\begin{equation}\label{KbarFredholm}
\det(\Id-\bar K_{b,\beta})_{L^2(\R_+)}=\sum_{n=0}^\infty\frac{(-1)^n}{n!}\int_{\R_+}\d x_1\dots\int_{\R_+}\d x_n\det\left[\bar K_{b,\beta}(x_i,x_j)\right]_{i,j=1}^{n}.
\end{equation}

The $n\times n$ determinant on the right-hand side of \eqref{KbarFredholm} is bounded by
\begin{equation}\label{boundntimesn}
\left|\det\left[\bar K_{b,\beta}(x_i,x_j)\right]_{i,j=1}^{n}\right|\le C^n e^{-\frac1{2\sigma}(x_1+\dots+x_k)} n^{n/2}
\end{equation}
using Lemma~\ref{lem:barKproperties} and the Hadamard bound on determinants with bounded entries.
Hence the integrand in the $n$th term on the right-hand side of \eqref{KbarFredholm} can be dominated uniformly as $\beta$ varies in $\left(-\frac14+\varepsilon,\frac14-\varepsilon\right)$.
In particular, by dominated convergence, as $\beta\to b$, the $n$th term of the expansion in \eqref{KbarFredholm} converges to the corresponding term of the expansion of $\det(\Id-\bar K_{b,b})$.
Further, by integrating the bound \eqref{boundntimesn}, the absolute value of the $n$th term of the series \eqref{KbarFredholm} is at most $(2\sigma C)^n n^{n/2}/n!$.
Since it is summable, a repeated application of the dominated convergence yields the convergence of Fredholm determinants \eqref{convFred}.

The last step is to show the convergence of the scalar products in \eqref{detK0delta}.
For this end, we first show that the resolvents converge in operator norm in $L^2(\R_+)$.
\begin{lemma}\label{lemma:resolventconv}
If $\lim_{\beta\to b}\|\bar K_{b,\beta}-\bar K_{b,b}\|=0$ and $\|(\Id-\bar K_{b,b})^{-1}\|<\infty$, then
\[\lim_{\beta\to b}\left\|(\Id-\bar K_{b,\beta})^{-1}-(\Id-\bar K_{b,b})^{-1}\right\|=0.\]
\end{lemma}

\begin{proof}
\begin{align*}
\left\|(\Id-\bar K_{b,\beta})^{-1}-(\Id-\bar K_{b,b})^{-1}\right\|
&=\left\|\left[\left(\Id-(\Id-\bar K_{b,b})^{-1}(\bar K_{b,\beta}-\bar K_{b,b})\right)^{-1}-1\right](\Id-\bar K_{b,b})^{-1}\right\|\\
&\le\left\|(\Id-\bar K_{b,b})^{-1}\right\|\sum_{n\ge1}\left\|(\Id-\bar K_{b,b})^{-1}(\bar K_{b,\beta}-\bar K_{b,b})\right\|^n
\end{align*}
which goes to $0$ as $\beta\to b$.
\end{proof}

To finish the proof of Theorem~\ref{thm:limK0delta}, we first check the conditions of Lemma~\ref{lemma:resolventconv}
and then we show that the right-hand side of \eqref{detK0delta} goes to $\Xi(S,b,\sigma)/\sigma$.

To verify the convergence condition of Lemma~\ref{lemma:resolventconv}, one can write $\bar K_{b,\beta}-\bar K_{b,b}$ as a common double integral with a difference of gamma ratios.
This difference goes to $0$ pointwise as $\beta\to b$, hence the Hilbert--Schmidt norm of $\bar K_{b,\beta}-\bar K_{b,b}$ goes to $0$ by dominated convergence as $\beta\to b$.
The finite norm condition for the resolvent is a direct consequence of \eqref{detneq0}.

Since
\[q_{b,\beta}\to q_{b,b}\qquad\mbox{and}\qquad q_{-\beta,-b}\to q_{-b,-b}\]
as $\beta\to b$ in $L^2(\R_+)$ by dominated convergence, we have
\begin{equation}\label{easiestscalarprod}
\big\langle(\Id-\bar K_{b,\beta})^{-1}q_{b,\beta},q_{-\beta,-b}\big\rangle\to\big\langle(\Id-\bar K_{b,b})^{-1}q_{b,b},q_{-b,-b}\big\rangle.
\end{equation}
In the other scalar products, the functions $r_b$ and $r_{-b}$ appear in the limit which may not be in $L^2(\R_+)$, therefore, as in Remark~\ref{rem:barKproduct}, we use again the identity
\[(\Id-\bar K_{b,b})^{-1}=\Id+(\Id-\bar K_{b,b})^{-1}\bar K_{b,b}\]
and note that $\bar K_{b,b}r_b$ and $\bar K_{b,b}r_{-b}$ are already in $L^2(\R_+)$, so the rest of the argument is the same as for \eqref{easiestscalarprod} and the use of dominated convergence.
This completes the proof of Theorem~\ref{thm:limK0delta}.
\end{proof}

\section{SHE/KPZ equation universality -- Proof of Theorem~\ref{CorUniversality}}\label{SectUniversality}
In this section, we prove Theorem~\ref{CorUniversality}. Let
\begin{equation}\label{deff}
f(t)=\BesselK_0(2e^t)
\end{equation}
where $\BesselK_0$ is the modified Bessel function and set
\begin{equation}\label{Sscaling}
S=e^{-\frac{\tau^2+r}\sigma}.
\end{equation}
Then using the scaling \eqref{XTscaling} and the notation \eqref{defF}, we have
\begin{equation}\label{freeenergytransform}\begin{aligned}
&\frac{\partial}{\partial r}\EE\left[2\sigma \BesselK_0\left(2\sqrt{Se^{\frac{X^2}{2T}+\frac T{24}}\mathcal Z_b(T,X)}\right)\right]\\
&\qquad=\frac{\partial}{\partial r}\EE\left[2\sigma f\left(\frac1{2\sigma}\left(\frac{\mathcal H_b(T,X)+\frac T{24}(1+12b^2)-2^{1/3}b\tau T^{2/3}}{(T/2)^{1/3}}-r\right)\right)\right]\\
&\qquad=\EE\left[2\sigma\frac{\partial}{\partial r}f\left(\frac1{2\sigma}\left(\frac{\mathcal H_b(T,X)+\frac T{24}(1+12b^2)-2^{1/3}b\tau T^{2/3}}{(T/2)^{1/3}}-r\right)\right)\right].
\end{aligned}\end{equation}
In the last step, we used the following property of the function $f$ given in \eqref{deff} along with Lemma~\ref{lem:Ftailbound}.
\begin{lemma}
If $H$ is a random variable with exponential negative tail, that is, if $\EE[\exp(-uH)]$ is finite for some $u>0$, then
\begin{equation}\label{derivativelemma}
\frac\partial{\partial r}\EE f(-r+H)=\EE\left[\frac\partial{\partial r}f(-r+H)\right]
\end{equation}
for the function $f$ given in \eqref{deff}.
\end{lemma}

\begin{proof}
The left-hand side of \eqref{derivativelemma} is written as
\[\lim_{h\to0}\EE\left[\frac1h(f(-(r+h)+H)-f(-r+H))\right].\]
Since $f'\in[0,1]$ and $f'(t)$ has a double exponential decay for large $t$ and it is bounded,
the dominated convergence applies and yields the equality with the right-hand side of \eqref{derivativelemma}.
\end{proof}

To conclude the proof of Theorem~\ref{CorUniversality}, we use the rotational invariance formula \eqref{shiftarg} with $b\to b+X/T$ and with the parameter setting \eqref{XTscaling}:
\begin{equation*}
Se^{\frac{X^2}{2T}+\frac T{24}}\mathcal Z_b(T,X)\stackrel{\d}=Se^{\frac T{24}}\mathcal Z_{\tau\sigma}(T,0).
\end{equation*}
Substituting this on the left-hand side of \eqref{freeenergytransform}, we see that
\[\eqref{freeenergytransform}=\frac{\partial}{\partial r}\EE\left[2\sigma \BesselK_0\left(2\sqrt{Se^{\frac T{24}}\mathcal Z_{\tau\sigma}(T,0)}\right)\right]
=\frac{\partial}{\partial r}\Xi\left(S=e^{-\frac{\tau^2+r}\sigma},\tau\sigma,\sigma\right)\]
where we applied Theorem~\ref{ThmFormulaStationary} in the last step with $S=e^{-(\tau^2+r)/\sigma}$ and $b=\tau\sigma$ (which is in $\left(-\frac14,\frac14\right)$ for $T$ large enough).
Then using Lemma~\ref{lem:TinflimitK} below and the definition \eqref{defBFP}, we see that the right-hand side of \eqref{freeenergytransform} converges to $F_\tau(r)$ for each $r\in\R$.
The functions $\{f_T\}_{T>0}$ with $f_T(x)=-f'(x/(2\sigma))$ are strictly decreasing in $x$ with a limit of $1$ at $x=-\infty$ and $0$ at $x=\infty$,
and for each $\delta>0$, on $\R\setminus[-\delta,\delta]$, $f_T$ converges uniformly to $\Id(x\le0)$.
Hence by rewriting the right-hand side of \eqref{freeenergytransform} with $f_T$ which converges to $F_\tau(r)$ and by using a continuous version of Lemma~\ref{problemma1},
we conclude the proof of Theorem~\ref{CorUniversality}.

\begin{lemma}\label{lem:TinflimitK}
We have
\begin{equation}\label{TinflimitK}
\lim_{T\to\infty}\Xi\left(e^{-\frac{\tau^2+r}\sigma},\tau\sigma,\sigma\right)=g(\tau,r)\det(\Id-P_r\wh K_{\Ai}P_r)_{L^2(\R)}
\end{equation}
where $g$ is defined by \eqref{defg}, the shifted Airy kernel $\wh K_{\Ai}$ is given by \eqref{defKAi}, and $P_s(x)=\Id_{\{x>s\}}$.
\end{lemma}

\begin{proof}
Let us introduce the following notation for this proof. Consider the operator
\[B_r(x,y)=\Ai(x+y+r)\]
acting on $L^2(\R_+)$ as an integral operator and the functions
\[e_\alpha(x)=e^{\alpha x}\]
for $\alpha\in\R$. Note that $e_\alpha\not\in L^2(\R_+)$ if $\alpha\le0$, but since it will always appear in this proof together with $B_r$,
the fast decay of the Airy function makes all the integrals convergent.

We take the $T\to\infty$ limit on the left-hand side of \eqref{TinflimitK} by using \eqref{defz} with $S=e^{-(\tau^2+r)/\sigma}$ and $b=\tau\sigma$.
If we take the $T\to\infty$ limit of $\bar K_{\tau\sigma,\tau\sigma}$, we observe that
\[\frac{\sigma\pi}{\sin(\sigma\pi(z-w))}\to\frac1{z-w}\]
since $\sigma\to0$. On the other hand, because $\Gamma(z)\simeq z^{-1}$ as $z\to0$, we also have
\[\frac{\Gamma(\tau\sigma-\sigma z)}{\Gamma(\sigma z-\tau\sigma)}\frac{\Gamma(\sigma w-\tau\sigma)}{\Gamma(\tau\sigma-\sigma w)}\to1.\]
Substituting \eqref{Sscaling}, one obtains
\begin{align*}
\bar K_{\tau\sigma,\tau\sigma}(x,y)&\to\frac1{(2\pi\I)^2}\int_{-\frac1{4\sigma}+\I\R}\d w\int_{\frac1{4\sigma}+\I\R}\d z \frac{e^{z^3/3-z(y+r+\tau^2)}}{e^{w^3/w-w(x+r+\tau^2)}}\frac1{z-w}\\
&=K_{\Ai}(x+r+\tau^2,y+r+\tau^2)=B_r^2(x+\tau^2,y+\tau^2)
\end{align*}
pointwise.
By using the same argument as in the proof of \eqref{convFred}, we see that we also have convergence in trace norm and hence the convergence of Fredholm determinants.
Moreover, by applying the analogue of Lemma~\ref{lemma:resolventconv}, we obtain the convergence of the resolvents.

Similarly, we have
\begin{align*}
q_{\tau\sigma,\tau\sigma}(x)&\to-e^{-\tau^3-\tau r}\frac1{2\pi\I}\int_{-\frac1{4\sigma}+\I\R}\d w\,\frac{e^{-w^3/3+w(x+r+\tau^2)}}{\tau-w}\\
&=-e^{-\tau^3-\tau r}\int_0^\infty\d\lambda\Ai(x+r+\tau^2+\lambda)\,e^{-\lambda \tau}
=-e^{-\tau^3-\tau r}(B_r e_{-\tau})(x+\tau^2)
\end{align*}
and
\[q_{-\tau\sigma,-\tau\sigma}(x)\to-e^{\tau^3+\tau r}(B_r e_\tau)(x+\tau^2).\]
The convergence holds also in $L^2(\R_+)$ by the dominated convergence theorem.

By writing $r_{\tau\sigma}(x)=e^{\frac43\tau^3}e_{-\tau}(x+\tau^2)$ and $r_{-\tau\sigma}(x)=e^{-\frac43\tau^3}e_\tau(x+\tau^2)$, we get
\begin{multline}\label{scalarprodcalc}
\big\langle(\Id-\bar K_{\tau\sigma,\tau\sigma})^{-1} q_{\tau\sigma,\tau\sigma},q_{-\tau\sigma,-\tau\sigma}\big\rangle\\
\to\int_0^\infty\d x\int_0^\infty\d y(\Id-B_r^2)^{-1}(x+\tau^2,y+\tau^2)(B_re_{-\tau})(y+\tau^2)(B_re_\tau)(x+\tau^2),
\end{multline}
\begin{multline}
\big\langle(\Id-\bar K_{\tau\sigma,\tau\sigma})^{-1} q_{\tau\sigma,\tau\sigma},r_{\tau\sigma}\big\rangle\\
\to -e^{\tau^3/3-\tau r}\int_0^\infty\d x\int_0^\infty\d y(\Id-B_r^2)(x+\tau^2,y+\tau^2)(B_re_{-\tau})(y+\tau^2)e_{-\tau}(x+\tau^2),
\end{multline}
\begin{multline}
\big\langle(\Id-\bar K_{\tau\sigma,\tau\sigma})^{-1} r_{-\tau\sigma},q_{-\tau\sigma,-\tau\sigma}\big\rangle\\
\to -e^{-\tau^3/3+\tau r}\int_0^\infty\d x\int_0^\infty\d y(\Id-B_r^2)(x+\tau^2,y+\tau^2)e_\tau(y+\tau^2)(B_re_\tau)(x+\tau^2)
\end{multline}
and
\begin{multline}
\big\langle\bar K_{\tau\sigma,\tau\sigma}(\Id-\bar K_{\tau\sigma,\tau\sigma})^{-1} r_{-\tau\sigma},r_{\tau\sigma}\big\rangle\\
\to\int_0^\infty\d x\int_0^\infty\d y\int_0^\infty\d z B_r^2(x+\tau^2,y+\tau^2)(\Id-B_r^2)^{-1}(y+\tau^2,z+\tau^2)e_\tau(z+\tau^2)e_{-\tau}(x+\tau^2).
\end{multline}
Finally observe that by \eqref{Sscaling}
\begin{equation}\label{firsttermslimit}
(\tau\sigma)^2/\sigma^2+\sigma(2\gamma_{\rm E}+\ln S)\to-r.
\end{equation}

To get a similar formulation on the right-hand side of \eqref{TinflimitK}, we substitute $s=r$ to the ingredients defining the function $g$:
\begin{equation}\label{BFPcalc}\begin{aligned}
\wh K_{\Ai}(x+r,y+r)&=B_r^2(x+\tau^2,y+\tau^2),\\
\mathcal R&=r+e^{\tau^3/3-\tau r}\int_0^\infty\d x(B_re_{-\tau})(x+\tau^2)e_{-\tau}(x+\tau^2),\\
\Phi(x+r)&=e^{\tau^3/3-\tau r}\int_0^\infty\d y B_r^2(x+\tau^2,y+\tau^2)e_{-\tau}(y+\tau^2)-(B_re_\tau)(x+\tau^2),\\
\Psi(y+r)&=e^{-\tau^3/3+\tau r}e_\tau(y+\tau^2)-(B_re_{-\tau})(y+\tau^2).
\end{aligned}\end{equation}
To obtain \eqref{TinflimitK} one needs to substitute the limits \eqref{scalarprodcalc}--\eqref{firsttermslimit} to \eqref{defz},
comparing the result with the right-hand side of \eqref{TinflimitK} using the expressions of \eqref{BFPcalc},
and rewriting the scalar product in the definition of the function $g$ in \eqref{defg} as an integral.
\end{proof}

\appendix

\section{Stationary semi-discrete directed random polymer}\label{AppStationary}
Though we will not draw upon this, it is worthwhile to explain why we used the term `stationary' to describe the partition function $\Zsd(\tau,N)$ described in Remark~\ref{remstat}
(in which $M=1$, $\alpha=a_1=a$ and all other $a_i=0$).
The following explanation goes back to O'Connell and Yor~\cite{OCY01} in the case of $N=2$ and to Sepp\"{a}l\"{a}inen and Valk\'o~\cite{SV10} for general $N$.
Consider two-sided Brownian motions $B_1,\ldots, B_N$ where $B_1$ has drift $\alpha$ and $B_i$ for $i>1$ has drift $0$.
By a two sided Brownian motion with drift $\alpha$, we mean that for $s\leq 0$, $B_i(s)=B^{-}_i(s)+ \alpha s$ and for $s\geq 0$, $B_i(s) = B^{+}_i(s) + \alpha s$
where $B^{\pm}_i$ are independent Brownian motions.
Denote by $B_k(s,t)=B_k(t)-B_k(s)$ the increment of the Brownian motion $k$ between time $s$ and $t$. Define
\begin{equation*}
\tilde{\Zsd}(\tau,N) = \int_{-\infty<s_1<\cdots<s_{N-1}<\tau} \exp\left[B_1(s_1) + B_2(s_1,s_2)+\cdots + B_N(s_{N-1},\tau)\right],
\end{equation*}
and, recursively, $r_k(\tau)$ by
\begin{equation}\label{rkeqn}
\sum_{j=1}^{k}r_j(\tau) =\ln \left[\tilde{\Zsd}(\tau,k+1) \right] + B_1(\tau)-2\alpha \tau.
\end{equation}
Let $r(\tau):=\left\{r_1(\tau),\ldots,r_{N-1}(\tau)\right\}$.

The following result is a subset of the results proved in Theorem~3.3 of~\cite{SV10} and is an extension of the Output Theorem for $M/M/1$ queues (sometimes called the Burke-type property).
\begin{proposition}[Theorem~3.3 of~\cite{SV10}]\label{SVburke}
For a given $\tau$, the random variables in each component of the vector $r(\tau)$ are independent and identically distributed as $r_k(\tau)\sim \LogGdist(\alpha)$; and as a process in $\tau$, $r(\tau)$ is stationary.
\end{proposition}

As a consequence of this, we have the following.

\begin{corollary}
\mbox{}
\begin{enumerate}
\item
In law, $\Zsd(\tau,N) = \tilde{\Zsd}(\tau,N)$,
\item $\EE[\ln\Zsd(\tau,N)] = -N\psi(\alpha) + \alpha \tau$, where $\psi$ is the digamma function,
\item The ordered set of $\Usd(\tau,k) := \ln\Zsd(\tau,k+1) - \ln\Zsd(\tau,k)$ for $1\leq k\leq N-1$ is stationary in $\tau$ with product measure of $\LogGdist(\alpha)$ distributions in each coordinate.
\end{enumerate}
\end{corollary}

\begin{proof}
To prove (1), we consider $\tau>0$ and note that we can partition the integral defining $\tilde{\Zsd}(\tau,N)$ based on for which $k\in \{1,\ldots, N\}$ the event $\{s_{k-1}<0<s_k\}$ occurs (as convention set $s_{0}=-\infty$). Thus
\begin{multline*}
\tilde{\Zsd}(\tau,N) = \sum_{k=1}^{N} \int_{-\infty<s_1<\cdots<s_{k-1}<0} \exp\left[B_1(s_1) + B_2(s_1,s_2)+\cdots + B_k(s_{k-1},0)\right] \\
\times \int_{0<s_k<\cdots<s_{N-1}<\tau} \exp\left[B_k(0,s_k) + B_{k+1}(s_k,s_{k+1})+\cdots + B_{N}(s_{N-1},\tau)\right].
\end{multline*}
By Proposition~\ref{SVburke}, the first integral above equals $\exp\left[\sum_{j=1}^{k-1} r_k(0)\right]$ and by the independent increment property of Brownian motion, the second integral is independent of the first. Using the fact that the $r_k(0)\sim \LogGdist(\alpha)$ completes the identification of $\tilde{\Zsd}$ with $\bar{\Zsd}$.

To prove (2), take expectations of both sides in \eqref{rkeqn} and recall the mean of a Log-Gamma distributed random variable, as well as the fact that $B_0$ has drift $\alpha$.
To prove (3), subtract equation \eqref{rkeqn} with $k$ from \eqref{rkeqn} with $k-1$ and use the stationarity and product distribution of the $r_k(t)$ coming from Proposition~\ref{SVburke}.
\end{proof}

\section{Analyticity of Fredholm determinants}\label{AppAnalyticity}

\begin{lemma}\label{lem:Fredholm_anal}
The Fredholm determinants $\det(\Id-K_{b,\beta})_{L^2(\R_+)}$ and $\det(\Id-\wt K_{b,\beta})_{L^2(\Cv{w})}$ are analytic functions of the parameters $b$ and $\beta$ as long as $b<\beta$.
\end{lemma}

The proof of the above lemma consists of two steps: first we show that the kernels are analytic functions of $b$ and $\beta$, then we prove it for the Fredholm determinants.
We use the following two complex analysis lemmas and Lemma~\ref{lem:Kbound} for the decay bound on the kernel $K_{b,\beta}$.
The first one is a slight modification of Theorem~7.37 in~\cite{Gon92}, hence for completeness, we give it with proof.

\begin{lemma}\label{lemma:gonzalez}
Let $f(z,\zeta)$ be a complex function in two variables and suppose that
\begin{enumerate}
\item $f$ is defined on $(z,\zeta)\in A\times C$ where $A$ is an open set and $C$ is a (possibly infinite) contour.
\item For each $z\in A$, define the contour $\gamma=\{z+re^{\I t}:0\le t\le2\pi\}$ with a sufficiently small $r$ such that also the disc around $z$ with radius $r$ lies in $A$.
Suppose that for each $z\in A$,
\begin{equation}\label{gonz:intcond}
\int_C\int_\gamma|f(u,\zeta)|\,|\d u|\,|\d\zeta|<\infty.
\end{equation}
\item For each $\zeta\in C$, $z\mapsto f(z,\zeta)$ is analytic in $A$.
\item For each $z\in A$, $\zeta\mapsto f(z,\zeta)$ is continuous on $C$.
\end{enumerate}
Then
\[F(z)=\int_C f(z,\zeta)\,\d\zeta\]
is analytic in $A$ with $F'(z)=\int_C \frac{\partial}{\partial z} f(z,\zeta)\,\d\zeta$.
\end{lemma}

\begin{proof}
By Cauchy's integral formula for the analytic function $z\mapsto f(z,\zeta)$, we get
\begin{equation}\label{CauchyF}
F(z)=\frac1{2\pi\I}\int_C\d\zeta\int_\gamma\frac{f(u,\zeta)\,\d u}{u-z}
\end{equation}
where $\gamma$ is defined in condition (2) of the lemma. If we choose $h$ such that $|h|<r/2$, then $|u-z|=r$ and $|u-z-h|>r/2$.
From \eqref{CauchyF}, we also have
\[F(z+h)=\frac1{2\pi\I}\int_C\d\zeta\int_\gamma\frac{f(u,\zeta)\,\d u}{u-z-h},\]
so that
\begin{equation}\label{differenceratio}\begin{aligned}
\frac{F(z+h)-F(z)}h&=\frac1{2\pi\I}\int_C\d\zeta\int_\gamma\frac{f(u,\zeta)\,\d u}{(u-z)(u-z-h)}\\
&=\frac1{2\pi\I}\int_C\d\zeta\int_\gamma\frac{f(u,\zeta)\,\d u}{(u-z)^2}+\frac1{2\pi\I}\int_C\d\zeta\int_\gamma\frac{hf(u,\zeta)\,\d u}{(u-z)^2(u-z-h)}.
\end{aligned}\end{equation}
The second term on the right-hand side of \eqref{differenceratio} is bounded as
\[\left|\frac1{2\pi\I}\int_C\d\zeta\int_\gamma\frac{hf(u,\zeta)\,\d u}{(u-z)^2(u-z-h)}\right|\le\frac1{2\pi}\frac{|h|}{r^3/2}\int_C\int_\gamma|f(z,\zeta)|\,|\d u|\,|\d\zeta|\]
which tends to $0$ as $h\to0$, because the double integral on the right-hand side is finite by condition (2). Hence \eqref{differenceratio} converges as $h\to0$, that is,
\[F'(z)=\frac1{2\pi\I}\int_C\d\zeta\int_\gamma\frac{f(u,\zeta)\,\d u}{(u-z)^2}=\int_C\frac{\partial}{\partial z}f(z,\zeta)\,\d\zeta\]
where we used Cauchy's differentiation formula in the last step.
\end{proof}

The second complex analysis lemma is due to Weierstrass and it is proved in~\cite{B06} as Theorem~7.12, so we omit the proof here.

\begin{lemma}\label{lem:series_anal}
Suppose $U$ is an open subset of $\C$ and that $\{f_n\}$ is a sequence of analytic functions on $U$ that converges uniformly to a function $f$.
Then $f$ is analytic on $U$.
\end{lemma}

We provide the following bound on the kernel $K_{b,\beta}$.

\begin{lemma}\label{lem:Kbound}
Fix $b<\beta$ so that $\beta-b<1$. There is a finite constant $C$ such that
\begin{equation*}
|K_{b,\beta}(x,y)|\le C\exp\left(-\frac\beta\sigma y+\frac b\sigma x\right)
\end{equation*}
for $x,y\in\R_+$.
\end{lemma}

\begin{proof}
In the general $b<\beta$ case, the contours $\Cv{w}$ and $\Cv{z}$ are vertical lines with local modifications at the origin, see Theorem~\ref{ThmFormulaContinuous} where the contours are defined. At least one of $b<0$ and $\beta>0$ is true. Suppose that $\beta>0$.
The other case is similar.
We chose an integer $k\ge0$ such that the shifted contour $\Cv{w}-\frac k\sigma$ lies completely on the left-hand side of the imaginary axis.
Using the residue theorem for the poles at $w=\frac b\sigma,\frac{b-1}\sigma,\dots,\frac{b-k+1}\sigma$ and at $w=z-\frac1\sigma,z-\frac2\sigma,\dots,z-\frac k\sigma$, we get
\begin{equation}\label{k_residues}\begin{aligned}
K_{b,\beta}(x,y)&=\frac1{(2\pi\I)^2}\int_{\Cv{w}-\frac k\sigma}\d w\int_{\Cv{z}}\d z \frac{\sigma\pi S^{\sigma(z-w)}}{\sin(\sigma\pi(z-w))}\frac{e^{z^3/3-zy}}{e^{w^3/3-wx}}
\frac{\Gamma(\beta-\sigma z)}{\Gamma(\sigma z-b)}\frac{\Gamma(\sigma w-b)}{\Gamma(\beta-\sigma w)}\\
&+\sum_{l=0}^{k-1}\frac1{2\pi\I}\int_{\Cv{z}}\d z \frac{\pi S^{\sigma z-(b-l)}}{\sin(\pi(\sigma z-(b-l)))}\frac{e^{z^3/3-zy}}{e^{(b-l)^3/(3\sigma^3)-(b-l)x/\sigma}}
\frac{\Gamma(\beta-\sigma z)}{\Gamma(\sigma z-b)}\frac1{\Gamma(\beta-(b-l))}\\
&+\sum_{m=1}^k\frac1{2\pi\I}\int_{\Cv{z}}\d z(-1)^mS^m\frac{e^{z^3/3-zy}}{e^{(z-m/\sigma)^3/3-(z-m/\sigma)x}}
\frac{\Gamma(\beta-\sigma z)}{\Gamma(\sigma z-b)}\frac{\Gamma(\sigma z-b-m)}{\Gamma(\beta-\sigma z+m)}.
\end{aligned}\end{equation}
The contour $\Cv{w}-\frac k\sigma$ intersects the real axis at a negative position.
Without crossing any pole coming from the sine in the denominator, we can replace the contours $\Cv{w}-\frac k\sigma$ and $\Cv{z}$ by vertical lines
crossing the real axis at the same positions as $\Cv{w}-\frac k\sigma$ and $\Cv{z}$ in \eqref{k_residues} everywhere.

Since $\Cv{z}$ is confined between $\Cv{w}$ and $\Cv{w}+\frac1\sigma$ up to an arbitrarily small error $\e$, and because $\beta-b<1$,
we can move the vertical integration contour for $z$ to the right-hand side of the pole at $z=\frac\beta\sigma$ picking up a residue term from the Gamma fuction but not from the sine.
The new vertical contour crosses the real axis at $\frac\beta\sigma+\delta$ for a small enough $\delta>0$,
and a new residue term appears in each summand of \eqref{k_residues} due to putting the contour on the other side of the residue at $z=\frac\beta\sigma$.

Then the largest $x$-dependent term comes from the residue at $w=\frac b\sigma$ resulting in the bound $e^{bx/\sigma}$ in the exponential order.
The largest $y$-dependence comes from the residue picked up at $z=\frac\beta\sigma$ giving $e^{-\beta y/\sigma}$ in the exponent.
\end{proof}

\begin{proof}[Proof of Lemma~\ref{lem:Fredholm_anal}]
It is a consequence of Lemma~\ref{lemma:gonzalez} that the kernels $K_{b,\beta}$ and $\wt K_{b,\beta}$
are analytic functions of $b$ and $\beta$ as long as $b<\beta$ holds.
To apply the lemma, the only non-trivial condition to check for $K_{b,\beta}$ is (2).
But $e^{z^3/3}$ decays along $\Cv{z}$ as $e^{-c|\Im(z)|^2}$.
From \eqref{gammalemma2}, we get that
\begin{equation}\label{gammaratiobound}
\left|\frac{\Gamma(\beta-\sigma z)}{\Gamma(\sigma z-b)}\right|\simeq|z|^{\beta+b-2\sigma\Re(z)},
\end{equation}
so if we vary $\beta$ in a small circle in the complex plane, we still have a uniform polynomial bound in \eqref{gammaratiobound}
which is enough to ensure the finiteness of the integral in \eqref{gonz:intcond} for $K_{b,\beta}$.
Checking the conditions of Lemma~\ref{lemma:gonzalez} for $\wt K_{b,\beta}$ can be done similarly.

What remains to prove is that the Fredholm determinants are also analytic.
We start with $K_{b,\beta}$ and the series expansion
\begin{equation}\label{Fredholmseries}
\det(\Id-K_{b,\beta})_{L^2(\R_+)}=\sum_{n=0}^\infty\frac{(-1)^n}{n!}\int_0^\infty\dots\int_0^\infty\d x_1\dots\d x_n\det\big[K_{b,\beta}(x_i,x_j)\big]_{i,j=1}^{n}.
\end{equation}
We are to apply Lemma~\ref{lem:series_anal} with the sequence of analytic functions $\{f_n\}$ being the partial sums of the series on the right-hand side of \eqref{Fredholmseries}.

If we rewrite the $n\times n$ determinant in the series on the right-hand side of \eqref{Fredholmseries} using the same $C$ that appears in Lemma~\ref{lem:Kbound}, then we have
\begin{equation}\label{Kdetrewrite}
\det\big[K_{b,\beta}(x_i,x_j)\big]_{i,j=1}^{n}=C^ne^{-\frac{\beta-b}{\sigma}\sum_{i=1}^n x_i}\det\big[C^{-1}e^{\frac{\beta}{\sigma}x_j-\frac{b}{\sigma}x_i}K_{b,\beta}(x_i,x_j)\big]_{i,j=1}^{n}.
\end{equation}
The entries in the determinant on the right-hand side here are at most $1$ in absolute value due to Lemma~\ref{lem:Kbound}, hence the determinant is at most $n^{n/2}$ by the Hadamard bound.
Therefore
\[\frac1{n!}\left|\int_0^\infty\dots\int_0^\infty\d x_1\dots\d x_n\det\big[K_{b,\beta}(x_i,x_j)\big]_{i,j=1}^{n}\right|\le\frac{\wt C^n n^{n/2}}{n!}\]
with $\wt C=C\sigma/(\beta-b)$ which is a summable upper bound, hence the series in \eqref{Fredholmseries} converges uniformly.
For the analyticity of the individual terms of the series,
we use Lemma~\ref{lemma:gonzalez} again for the integrand $\det\big[K_{b,\beta}(x_i,x_j)\big]_{i,j=1}^{n}$ for which \eqref{Kdetrewrite} provides the integrability condition \eqref{gonz:intcond}.
By using Lemma~\ref{lem:series_anal} for $\beta-b<1$, this proves the analyticity of $\det(\Id-K_{b,\beta})_{L^2(\R_+)}$ in $b$ and $\beta$.
If $\beta-b\ge1$, recall from the end of Section~\ref{ss:conv_bounds} that $\Cv{w}$ and $\Cv{z}$ intersect the real axis at $b+\frac1{4\sigma}$ and at $b+\frac3{4\sigma}$ respectively.
Hence in the bound of Lemma~\ref{lem:Kbound}, $\frac{b}{\sigma}$ and $\frac{\beta}{\sigma}$ in the exponent are replaced by $b+\frac1{4\sigma}$ and $b+\frac3{4\sigma}$ respectively,
but the rest of the proof remains the same.

The argument for $\det(\Id-\wt K_{b,\beta})_{L^2(\Cv{w})}$ is similar, but instead of the bound in Lemma~\ref{lem:Kbound}, one can see the stronger bound
\[|\wt K_{b,\beta}(w,w')|\le Ce^{-c|\Im(w)|^2}\]
for $w,w'\in\Cv{w}$ even more directly without modifying the contours.
\end{proof}

\section{Fredholm determinant bounds coming from kernel estimates}\label{BCFFredbdds}

\begin{lemma}[Proposition~1 of~\cite{TW08b}]\label{TWprop1}
Suppose $t\to \Gamma_t$ is a deformation of closed curves and a kernel $L(\eta,\eta')$ is analytic in a neighborhood of $\Gamma_t\times \Gamma_t\subset \C^2$ for each $t$.
Then the Fredholm determinant of $L$ acting on $\Gamma_t$ is independent of $t$.
\end{lemma}

\begin{lemma}[Lemma~B.2 of~\cite{BCF12}]\label{exponentialdecaycutoff}
Consider the Fredholm determinant $\det(\Id+K)_{L^2(\Gamma)}$ on an infinite complex contour $\Gamma$ and an integral operator $K$ on $\Gamma$.
Parameterize $\Gamma$ by arc length with some fixed point corresponding to $\Gamma(0)$.
Assume that $|K(v,v')|\leq C$ for some constant $C$ and for all $v,v'\in \Gamma$ and that the following exponential decay condition holds: there exists constants $c,C>0$ such that
\begin{equation*}
|K(\Gamma(s),\Gamma(s'))|\leq Ce^{-c|s|}.
\end{equation*}
Then the Fredholm series defining $\det(\Id+K)_{L^2(\Gamma)}$ is well-defined. Moreover, for any $\kappa>0$ there exists an $r_0>0$ such that for all $r>r_0$
\begin{equation*}
|\det(\Id+K)_{L^2(\Gamma)} - \det(\Id+K)_{L^2(\Gamma_r)}|\leq \kappa
\end{equation*}
where $\Gamma_r=\{\Gamma(s):|s|\leq r\}$.
\end{lemma}

\begin{lemma}[Lemma~B.3 of~\cite{BCF12}]\label{uniformptconvergence}
Consider a finite length complex contour $\Gamma$ and a sequence of integral operators $K^{\e}$ on $\Gamma$, as well as an additional integral operator $K$ also on $\Gamma$.
Assume that for all $\kappa>0$ there exists $\e_0$ such that for all $\e<\e_0$ and all $z,z'\in \Gamma$, \mbox{$|K^{\e}(z,z') - K(z,z')|\leq \kappa$}
and that there is some constant $C$ such that $|K(z,z')|\leq C$ for all $z,z'\in \Gamma$. Then
\begin{equation*}
\lim_{\e\to 0} \det(\Id+K^{\e})_{L^2(\Gamma)} = \det(\Id+K)_{L^2(\Gamma)}.
\end{equation*}
\end{lemma}

\section{Invertibility of the kernel $\Id-\bar K_{b,b}$}\label{AppInvertibility}
Let us define the function
\begin{equation*}
\Ai^{\rm pert}(x,b,\sigma)=\frac{1}{2\pi\I} \int_{\delta+\I\R} \d z \, e^{z^3/3-zx}\frac{\Gamma(-\sigma z+b)}{\Gamma(\sigma z-b)}
\end{equation*}
where $\sigma>0$ and $b\in\left(-\frac14,\frac14\right)$ are parameters such that $0<\delta<(b+1)/\sigma$.
Note that the integrand above has poles only at $(b+1)/\sigma,(b+2)/\sigma,\dots$.
Recall that the Airy function $\Ai$ has the integral representation
\begin{equation*}
\Ai(x)=\frac{1}{2\pi\I} \int_{\delta+\I\R} \d z \, e^{z^3/3-zx}
\end{equation*}
for any $\delta>0$.

\begin{lemma}\label{LemCompleteness}
Fix $\sigma>0$ and $b\in\left(-\frac14,\frac14\right)$.
The family of functions, indexed by $\lambda$, \mbox{$\Ai^{\rm pert}(x+\lambda,b,\sigma)$} satisfy the completeness relation
\begin{equation}\label{completenessrel}
\int_{\R}\d\lambda\, \Ai^{\rm pert}(x+\lambda,b,\sigma) \Ai^{\rm pert}(y+\lambda,b,\sigma)=\delta_{x-y}.
\end{equation}
\end{lemma}
Remark that since the function $\Ai^{\rm pert}(x+\lambda,b,\sigma)$ is a function of the sum of $x$ and $\lambda$, the completeness relation is equivalent to the orthogonality relation.

\begin{proof}
For this proof, we adapt an approach from~\cite{SI11} (Appendix A).
Let $H=-\partial_x^2+x$ be the Airy operator. It is known that for $t>0$,
\begin{equation*}
(e^{-tH})(x,y)=\int_{\R} \d\lambda\, e^{t \lambda} \Ai(x+\lambda)\Ai(y+\lambda)
\end{equation*}
and in particular, for $t\to 0$,
\begin{equation*}
\int_{\R} \d\lambda\, \Ai(x+\lambda)\Ai(y+\lambda) = \delta_{x-y}.
\end{equation*}

Consider now a $0<t<(b+1)/\sigma$. Then,
\begin{equation}\label{eqA6}
\begin{aligned}
\int_{\R} \d\lambda\, e^{t \lambda} \Ai^{\rm pert}(x+\lambda,b+\sigma t,\sigma)& \Ai^{\rm pert}(y+\lambda,b,\sigma) =
\frac{1}{2\pi\I} \int_{t/2+\I\R} \d z\, e^{z^3/3-zx}\frac{\Gamma(-\sigma z+b+\sigma t)}{\Gamma(\sigma z-b-\sigma t)}\\
\times \Bigg[
&\frac{1}{2\pi\I}\int_{t+\I\R} \d w\, e^{w^3/3-wy}\frac{\Gamma(-\sigma w+b)}{\Gamma(\sigma w-b)} \int_{\R_+} \d\lambda \, e^{-\lambda(w+z-t)}\\
+&\frac{1}{2\pi\I}\int_{t/4+\I\R} \d w\, e^{w^3/3-wy}\frac{\Gamma(-\sigma w+b)}{\Gamma(\sigma w-b)} \int_{\R_-} \d\lambda \, e^{-\lambda(w+z-t)}
\Bigg].
\end{aligned}
\end{equation}
Integrating over $\lambda$ gives $(w+z-t)^{-1}$ in the first case and $-(w+z-t)^{-1}$ in the second one. Joining the two integration contours over $w$, we get
\begin{equation*}
\begin{aligned}
\eqref{eqA6} &= \frac{1}{2\pi\I} \int_{t/2+\I\R} \d z\, e^{z^3/3-zx}\frac{\Gamma(-\sigma z+b+\sigma t)}{\Gamma(\sigma z-b-\sigma t)}
\frac{1}{2\pi\I}\oint_{\Gamma_{t-z}} \d w\, e^{w^3/3-wy}\frac{\Gamma(-\sigma w+b)}{\Gamma(\sigma w-b)}\frac{1}{w+z-t}\\
&=\frac{1}{2\pi\I} \int_{t/2+\I\R} \d z\, e^{z^3/3-zx} e^{-(z-t)^3/3+(z-t)y}=\int_{\R} \d\lambda\, e^{t \lambda} \Ai(x+\lambda)\Ai(y+\lambda)
\end{aligned}
\end{equation*}
where the last equality comes from the fact that we can redo the computations without the $\Gamma$ function and left-hand side of \eqref{eqA6} is the right-hand side of this last equation
and where $\Gamma_{t-z}$ is a small counterclockwise oriented circle around $t-z$.

Thus we have shown that
\begin{equation*}
\int_{\R} \d\lambda\, e^{t \lambda} \Ai^{\rm pert}(x+\lambda,b+\sigma t,\sigma) \Ai^{\rm pert}(x+\lambda,b,\sigma)= (e^{-tH})(x,y).
\end{equation*}
By taking $t\to 0$, we obtain the statement of the lemma.
\end{proof}

With the above notations and by using the identity
\[\frac{\sigma\pi S^{\sigma u}}{\sin(\sigma\pi u)}=\int_{\R}\d\lambda\frac{Se^{-u\lambda}}{S+e^{-\lambda/\sigma}}\]
which holds for $0<\Re(u)<1/\sigma$, we have
\begin{equation}\label{eqA9}
\bar K_{b,b}(x,y)=\int_{\R} \d\lambda \frac{S}{S+e^{-\lambda/\sigma}} \Ai^{\rm pert}(x+\lambda,-b,\sigma)\Ai^{\rm pert}(y+\lambda,b,\sigma)
\end{equation}

\begin{lemma}\label{lemNormK00is1}
$\bar K_{b,b}$ as operator on $L^2(\R)$ has operator norm $1$.
\end{lemma}
\begin{proof}
For given $\sigma>0$ and $b\in\left(-\frac14,\frac14\right)$, the sets $\{\Ai^{\rm pert}(\cdot+\lambda,b,\sigma),\lambda\in\R\}$ and
\mbox{$\{\Ai^{\rm pert}(\cdot+\lambda,-b,\sigma),\lambda\in\R\}$} are two orthonormal bases by the completeness relation \eqref{completenessrel}.
Hence the kernel
\[U_{-b,b}(x,y)=\int_{\R} \d\lambda \Ai^{\rm pert}(x+\lambda,-b,\sigma)\Ai^{\rm pert}(y+\lambda,b,\sigma)\]
defines a unitary operator which corresponds to a change of basis.
Then $\bar K_{b,b}$ can be written as
\begin{equation}\label{K=LU}
\bar K_{b,b}=L_{-b,-b}U_{-b,b}
\end{equation}
where the operator $L_{-b,-b}$ is defined by the kernel
\begin{equation}\label{defLbb}
L_{-b,-b}(x,y)=\int_{\R} \d\lambda \frac{S}{S+e^{-\lambda/\sigma}} \Ai^{\rm pert}(x+\lambda,-b,\sigma)\Ai^{\rm pert}(y+\lambda,-b,\sigma).
\end{equation}
Since $\{\Ai^{\rm pert}(\cdot+\lambda,-b,\sigma),\lambda\in\R\}$ is an orthonormal set of functions, \eqref{defLbb} is the spectral decomposition of $L_{-b,-b}$,
and the function $x\mapsto \Ai^{\rm pert}(x+\mu,-b,\sigma)$ is an eigenvector of $L_{-b,-b}$ with eigenvalue $\frac{S}{S+e^{-\mu/\sigma}}$.
Thus, by using \eqref{K=LU}, we get that
\begin{equation*}
\|\bar K_{b,b}\|=\sup_{\|f\|=1}\|\bar K_{b,b}f\|=\sup_{\|f\|=1}\|L_{-b,-b}U_{-b,b}f\|=\sup_{\|g\|=1}\|L_{-b,-b}g\|=\sup_{\mu\in\R} \frac{S}{S+e^{-\mu/\sigma}}=1.
\end{equation*}
\end{proof}

Finally, we show that $\Id-\bar K_{b,b}$ restricted to $\R_+$ is invertible.
\begin{lemma}\label{LemmaInvertibilityOfK00}
Let $P_0$ be the projection onto $\R_+$. Then,
\begin{equation*}
\| P_0 \bar K_{b,b} P_0\| <1,
\end{equation*}
which implies that $\Id- P_0 \bar K_{b,b} P_0$ is invertible and $\det(\Id-\bar K_{b,b})_{L^2(\R_+)}\neq 0$.
\end{lemma}
\begin{proof}
The proof is inspired by the one of Appendix B.3 of~\cite{FS05a}. It consists in a \emph{reductio ad absurdum}.
Assume that there exists an eigenvector $\psi$ of $P_0 \bar K_{b,b} P_0$ with eigenvalue $1$. Then
\begin{equation}\label{eqA13}
P_0 \bar K_{b,b} P_0 \psi = \psi
\end{equation}
implies that $\psi(x)=0$ for all $x\in \R_-$.
$P_0$ is a projector, thus $\|P_0\|=1$. This together with $\|\bar K_{b,b}\|=1$ from Lemma~\ref{lemNormK00is1} yields
\begin{equation*}
\|\psi\| \leq \|\bar K_{b,b} P_0 \psi\| \leq \|\psi\|.
\end{equation*}
Thus $\|\bar K_{b,b} P_0 \psi\|=\|\psi\|$.
Now let $\phi$ be the vector such that
\begin{equation*}
\bar K_{b,b} P_0 \psi = \psi + \phi.
\end{equation*}
From \eqref{eqA13}, we have that $P_0 \phi=0$, meaning that $\phi$ and $\psi$ are orthogonal. Thus, we have
\begin{equation*}
\|\bar K_{b,b} P_0 \psi\|^2 = \|\psi\|^2+ \|\phi\|^2
\end{equation*}
and from the relations above, we have $\|\phi\|=0$. Therefore, we have shown that
\begin{equation*}
\bar K_{b,b} P_0 \psi = \psi.
\end{equation*}
From the integral representation of $\bar K_{b,b}$, it follows that $\psi(x)$ is analytic in $x$ (as complex variable).
But since $\psi(x)=0$ on $\R_-$, then we conclude that $\psi(x)=0$ for all $x\in\R$, i.e.\ $\psi$ is not an eigenvector.
This ends the proof of the lemma.
\end{proof}

\section{Inverse Mellin Transform}\label{AppMellin}

\begin{proposition}\label{prop:mellin}
Let $R$ be a random variable and $\sigma>0$ a constant such that $\EE(\exp(-\delta R/\sigma))<\infty$ for some $\delta>0$. Assume that we have a formula for
\begin{equation*}
Q(x,\sigma):=\EE\big[2\sigma \BesselK_0(2 e^{(R-x)/(2\sigma)})\big].
\end{equation*}
Then, the distribution function of $R$ is given by
\begin{equation}\label{eqC2}
F(r):=\PP(R\leq r)=\frac{1}{\sigma^2}\frac{1}{2\pi\I}\int_{-\delta+\I\R} \frac{\d \xi}{\Gamma(-\xi)\Gamma(-\xi+1)} \int_{\R} \d x\, e^{x\xi/\sigma} Q(x+r,\sigma).
\end{equation}
\end{proposition}
\begin{proof}
This formula was contained in~\cite{NY92}, although not so explicitly written. Let us show that this holds. We have
\begin{equation*}
Q(x,\sigma)=\int_\R \d y\,  2 \sigma  \BesselK_0(2 e^{(y-x)/(2\sigma)})\frac{\d}{\d y}F(y).
\end{equation*}
We use the identity 9.6.27 of~\cite{AS84}, namely $\frac{\d}{\d y} K_0(y)=-K_1(y)$, and integrate by parts with the result
\begin{equation*}
Q(x,\sigma)=\int_\R \d y\, F(y) 2 e^{(y-x)/(2\sigma)} \BesselK_1(2 e^{(y-x)/(2\sigma)})
=\int_\R \d y\, F(y+x) 2 e^{y/(2\sigma)} \BesselK_1(2 e^{y/(2\sigma)}).
\end{equation*}
The boundary terms vanishes since $K_0(e^z)\to 0$ as $z\to\infty$ and $K_0(e^z)\sim -z$ as $z\to -\infty$ (see~9.6.8 of~\cite{AS84}) and by assumption on the distribution of $R$,
the distribution fuction $F$ goes to zero more rapidly than $1/(-z)$ as $z\to -\infty$. Therefore, we have
\begin{equation}\label{eqC4}
\begin{aligned}
\int_{\R} \d x\, e^{x\xi/\sigma} Q(x+r,\sigma)&=\int_\R\d x\, \int_\R \d y\, e^{x\xi/\sigma} F(y+x+r)2 e^{y/(2\sigma)} \BesselK_1(2 e^{y/(2\sigma)})\\
&=\int_\R\d z\, F(z+r) e^{z \xi/\sigma} \int_\R \d y\, 2 e^{y(1-2\xi)/(2\sigma)} \BesselK_1(2 e^{y/(2\sigma)})
\end{aligned}
\end{equation}
where we changed the variable $x=z-y$. From Formula~11.4.22 of~\cite{AS84}, one has the identity (after a change of variable)
\begin{equation*}
\int_{\R} \d y \, 2 e^{y \mu/(2\sigma)} \BesselK_\nu(2 e^{y/(2\sigma)})=\sigma \Gamma\left(\frac{\mu+\nu}{2}\right) \Gamma\left(\frac{\mu-\nu}{2}\right)
\end{equation*}
whenever $\Re(\mu\pm \nu)>0$. Applying this formula for $\mu=1-2\xi$ and $\nu=1$, we get
\begin{equation*}
\eqref{eqC4}=\int_\R\d z\, F(z+r) e^{z \xi/\sigma} \sigma \Gamma(-\xi+1) \Gamma(-\xi)
\end{equation*}
whenever $\Re(\xi)<0$. If $R$ has negative exponential moments as assumed, then (we also changed the variables $\xi\to\xi \sigma$)
\begin{equation*}
\textrm{RHS of }\eqref{eqC2}=\frac{1}{2\pi\I} \int_{-\delta+\I\R} \d\xi \int_\R\d z\, F(z+r) e^{z \xi} = F(r)
\end{equation*}
where the last equality holds since the middle expression is the inverse Laplace transform of a Laplace/Fourier transform.
\end{proof}

\section{Probability lemmas}

\begin{lemma}[Lemma~4.1.39 of~\cite{BC11}]\label{problemma1}
Consider a sequence of functions $\{f_n\}_{n\geq 1}$ mapping $\R\to [0,1]$ such that for each $n$, $f_n(x)$ is strictly decreasing in $x$ with a limit of $1$ at $x=-\infty$ and $0$ at $x=\infty$, and for each $\delta>0$, on $\R\setminus[-\delta,\delta]$, $f_n$ converges uniformly to $\mathbf{1}_{x\leq 0}$. Consider a sequence of random variables $X_n$ such that for each $r\in \R$,
\begin{equation*}
\EE[f_n(X_n-r)] \to p(r)
\end{equation*}
and assume that $p(r)$ is a continuous probability distribution function. Then $X_n$ converges weakly in distribution to a random variable $X$ which is distributed according to \mbox{$\PP(X\leq r) = p(r)$}.
\end{lemma}

\begin{lemma}[Lemma~4.1.40 of~\cite{BC11}]\label{problemma2}
Consider a sequence of functions $\{f_n\}_{n\geq 1}$ mapping $\R\to [0,1]$ such that for each $n$, $f_n(x)$ is strictly decreasing in $x$ with a limit of $1$ at $x=-\infty$ and $0$ at $x=\infty$, and $f_n$ converges uniformly on $\R$ to $f$. Consider a sequence of random variables $X_n$ converging weakly in distribution to $X$.
Then
\begin{equation*}
\EE[f_n(X_n)] \to \EE[f(X)].
\end{equation*}
\end{lemma}

\section{Useful q-deformations}\label{qSec}

We record some $q$-deformations of classical functions. Section 10 of~\cite{AAR04} is a good references for many of these definitions and statements. We assume throughout that $|q|<1$. The classical functions are recovered in all cases in the $q\to 1$ limit, though the exact nature of this convergence is explained below.

The $q$-Pochhammer symbol is written as $(a;q)_{n}$ and defined via the product (infinite convergent product for $n=\infty$)
\begin{equation*}
(a;q)_{n}=(1-a)(1-aq)(1-aq^2)\cdots (1-aq^{n-1}), \qquad (a;q)_{\infty}=(1-a)(1-aq)(1-aq^2)\cdots.
\end{equation*}
The $q$-factorial is written as either $[n]_{q}!$ or just $n_q!$ and is defined as
\begin{equation*}
n_q! = \frac{(q;q)_n}{(1-q)^n} = \frac{(1-q)(1-q^2)\cdots (1-q^n)}{(1-q)(1-q)\cdots (1-q)}.
\end{equation*}
The $q$-binomial theorem~\cite[Theorem~10.2.1]{AAR04} says that for all $|x|<1$ and $|q|<1$,
\begin{equation*}
\sum_{k=0}^{\infty} \frac{(a;q)_k}{(q;q)_k} x^k = \frac{(ax;q)_{\infty}}{(x;q)_{\infty}}.
\end{equation*}
One corollary of this theorem~\cite[Corollary~10.2.2]{AAR04} is that under the same hypothesis on $x$ and $q$,
\begin{equation*}
\sum_{k=0}^{\infty} \frac {x^k}{k_q!} = \frac{1}{\big((1-q)x;q\big)_{\infty}}.
\end{equation*}

There are two different $q$-exponential functions introduced by Hahn~\cite{Hahn}. We will only need the first which is denoted by $e_q(x)$ and defined as
\begin{equation*}
e_q(x) = \frac{1}{\big((1-q)x;q\big)_{\infty}}.
\end{equation*}
For compact sets of $x$, $e_q(x)$ converges uniformly to $e^{x}$ as $q\to 1$. In fact, the convergence is uniform over $x\in (-\infty,0)$ as well.

The $q$-Gamma function is defined as
\begin{equation*}
\Gamma_q(x) = \frac{(q;q)_{\infty}}{(q^x;q)_{\infty}} (1-q)^{1-x}.
\end{equation*}
For $x$ in compact subsets of $\C\setminus\{0,-1,\cdots\}$, $\Gamma_q(x)$ converges uniformly to $\Gamma(x)$ as $q\to 1$. Owing to its definition, the $q$-Gamma functions satisfies $\Gamma_q(x+1) = \frac{1-q^x}{1-q} \Gamma_q(x)$.

\bibliographystyle{plain}
\bibliography{Biblio}

\end{document}